\documentclass[11pt,a4paper]{amsart}
\usepackage{amsmath}
\usepackage{amsfonts}
\usepackage{amssymb}
\usepackage{amsthm}
\usepackage{txfonts}
\usepackage[colorlinks]{hyperref} % for coloured links
\usepackage{wasysym}
\usepackage{graphicx}
\usepackage{dashrule}
\usepackage[all]{xy}
\usepackage{multirow}
\usepackage{enumerate}
\usepackage{tikz}
\usetikzlibrary{arrows,decorations.pathmorphing,backgrounds,positioning,fit,calc}
\usepackage{tikz-cd}
\usepackage{framed}
\theoremstyle{plain}
\newtheorem{theorem}{Theorem}[section]
\newtheorem{corollary}[theorem]{Corollary}

\newtheorem{conjecture}[theorem]{Conjecture}
\newtheorem{lemma}[theorem]{Lemma}
\newtheorem{proposition}[theorem]{Proposition}
\theoremstyle{definition}
\newtheorem{definition}[theorem]{Definition}
\newtheorem{example}[theorem]{Example}
\newtheorem{remark}[theorem]{Remark}

\DeclareMathOperator{\GL}{GL}
\DeclareMathOperator{\SL}{SL}

\DeclareMathOperator{\SU}{SU}

\DeclareMathOperator{\proj}{Proj}
\DeclareMathOperator{\spec}{Spec}
\DeclareMathOperator{\Hom}{Hom}

\DeclareMathOperator{\sym}{Sym}

\DeclareMathOperator{\stab}{Stab}

\DeclareMathOperator{\Aut}{Aut}
\DeclareMathOperator{\Lie}{Lie}

\newcommand{\nocontentsline}[3]{}
\newcommand{\tocless}[2]{\bgroup\let\addcontentsline=\nocontentsline#1{#2}\egroup}
\newcommand{\act}{\curvearrowright}

\newcommand{\CC}{\mathbb{C}}
\newcommand{\PP}{\mathbb{P}}
\newcommand{\A}{\mathbb{A}}
\newcommand{\GG}{\mathbb{G}}
\newcommand{\OO}{\mathcal{O}}

\newcommand{\env}{/\!/}

\newcommand{\ten}{\otimes}
\newcommand{\mc}{\mathcal}
\newcommand{\mf}{\mathfrak}
\newcommand{\mb}{\mathbb}

\newcommand{\kk}{\Bbbk}

\newcommand{\nss}{\mathrm{nss}}
\newcommand{\ssfg}{\mathrm{ss,fg}}

\newcommand{\rms}{s}

\newcommand{\newmmxu}{\mathfrak{m}_{\hat{U},X,\Omega}}
\newcommand{\newmmxh}{\mathfrak{m}_{H,X,\Omega}}

\newcommand{\newmmxg}{\mathfrak{m}_{G,X,\Omega}}

\newcommand{\newmm}{\mathfrak{m}_{G,Y,\Omega}}

\newcommand{\Spec}{{\rm Spec}}

\newcommand{\nc}{\newcommand}
\nc{\bla}{\phantom{bbbbb}}

\newcommand{\beq}{\begin{equation}}
\newcommand{\eeq}{\end{equation}}
\newcommand{\barr}{\begin{array}}
\newcommand{\earr}{\end{array}}
\newcommand{\beqar}{\begin{eqnarray}}
\newcommand{\eeqar}{\end{eqnarray}}

\newcommand{\RR}{{\mathbb R }}
\nc{\FF}{ {\mathbb F} }
\nc{\HH}{ {\mathbb H} }

\newcommand{\QQ}{{\mathbb Q }}

\newcommand{\calf}{\mathcal{F}}

\newcommand{\calr}{\mathcal{R}}

\DeclareMathOperator{\Stab}{Stab}

\nc{\umax}{{U_{\max}}}
\newcommand{\Hilb}{\mathrm{Hilb}}
\newcommand{\CHilb}{\mathrm{CHilb}}
\newcommand{\weight}{\omega}

\newcommand{\reg}{\mathrm{reg}}

\newcommand{\grass}{\mathrm{Grass}}

\newcommand{\bz}{\mathbf{z}}

\newcommand{\flag}{\mathrm{Flag}}
\newcommand{\Diff}{\mathrm{Diff}}

\newcommand{\liek}{{\mathfrak k}}
\newcommand{\lieu}{{\mathfrak u}}

\newcommand{\hU}{\hat{U}}

\newcommand{\lieks}{{\liek}^*}
\newcommand{\liet}{{\mathfrak t}}
\newcommand{\lier}{{\mathfrak r}}

\nc{\lieq}{{\mathfrak q}}
\nc{\lieh}{{\mathfrak h}}
\nc{\liez}{{\mathfrak z}}
\nc{\lieqs}{{\lieq}^*}
\nc{\lieg}{{\mathfrak g}}
\nc{\liev}{{\mathfrak v}}
\nc{\lieb}{{\mathfrak b}}
\nc{\liegs}{{\lieg}^*}
\nc{\liep}{{\mathfrak p}}
\nc{\liew}{{\mathfrak w}}
\nc{\lieps}{{\liep}^*}

\newcommand{\GHilb}{\mathrm{GHilb}}
\newcommand{\Tp}{\mathrm{Tp}}

\def\a{\alpha}
\def\b{\beta}
\def\g{\gamma}

\def\l{\lambda}

\def\x{\xi}

\title{Moment maps and cohomology of non-reductive quotients}

\author{Gergely B\'erczi}
\address{Department of Mathematics, Aarhus University} 
\email{gergely.berczi@math.au.dk}
\author{Frances Kirwan}
\address{Mathematical Institute, University of Oxford} 
\email{frances.kirwan@maths.ox.ac.uk}

\begin{document}

\vspace{.5in}
\maketitle
\begin{center}
{\em In memoriam Michael Atiyah}
\end{center}

\begin{abstract} Let $H$ be a complex linear algebraic group with internally graded unipotent radical acting on a complex projective variety $X$. Given an ample linearisation of the action and an associated Fubini--Study K\"ahler form which is invariant for a maximal compact subgroup $Q$ of $H$, we define a notion of moment map for the action of $H$, and under suitable conditions (that the linearisation is well-adapted and semistability coincides with stability) we describe the (non-reductive) GIT quotient $X/\!/H$ introduced in \cite{BDHK2} in terms of this moment map. Using this description we derive formulas for the Betti numbers of $X/\!/H$ and express the rational cohomology ring of $X/\!/H$ in terms of the rational cohomology ring of the GIT quotient $X/\!/T^H$, where $T^H$ is a maximal torus in $H$. We relate intersection pairings on $X/\!/H$ to intersection pairings on $X/\!/T^H$, obtaining a residue formula for these pairings on $X/\!/H$ analogous to the residue formula of \cite{jeffreykirwan}. As an application, we announce a proof of the Green--Griffiths--Lang and Kobayashi conjectures for projective hypersurfaces with polynomial degree. 
\end{abstract}

\section{Introduction}

The aim of this paper is to extend the results of  \cite{jeffreykirwan,francesthesis,SM} on computing the rational cohomology of GIT quotients of complex projective varieties by reductive group actions to cases when the group which acts is not necessarily reductive. The methods developed in this paper are used in \cite{bkGGL} to prove the Green--Griffiths--Lang conjecture and Kobayashi conjecture on hyperbolicity for generic projective hypersurfaces with polynomial degree, see \S \ref{sec:applications}.

In \cite{BDHK,BDHK2}  projective GIT quotients of complex projective varieties by suitable linear actions of linear algebraic groups with internally graded unipotent radicals are studied. Here we say that a linear algebraic group $H=U\rtimes R$ over $\CC$, with unipotent radical $U$ and Levi subgroup $R$, has internally graded unipotent radical if $R$ has  a central one-parameter subgroup $\lambda:  \GG_m \to Z(R)$ whose weights for the adjoint action on the Lie algebra of $U$ are all strictly positive; here $\GG_m = \CC^*$ is the multiplicative group. When $H$ acts linearly on a projective variety such that semistability coincides with stability in a suitable sense, % for the action of $\hat{U} = U \rtimes \lambda(\GG_m)$,
it is shown in \cite{BDHK2} that --- after appropriate adjustment of the linearisation --- many of the key features and computational strengths of reductive GIT  extend to this non-reductive situation. The present paper studies the rational cohomology of these non-reductive GIT quotients. 

When $X$ is a nonsingular complex projective variety acted on by a complex reductive group $G$ with respect to an ample linearisation, we have a quotient $X/\!/G$ in the sense of Mumford's geometric invariant theory (GIT) \cite{GIT}. In this situation we can choose a maximal compact subgroup $K$ of $G$ and a $K$-invariant Fubini--Study K\"ahler form on $X$ with corresponding moment map $\mu:X\to \mathfrak{k}^*$, where $\mathfrak{k}$ is the Lie algebra of $K$ and $\mathfrak{k}^* = \Hom_{\RR}(\mathfrak{k},\RR)$ is its dual. Here $\mathfrak{k}^*$ embeds naturally in the complex dual $\mathfrak{g}^* = \Hom_{\CC}(\mathfrak{k},\CC)$ of the 
Lie algebra $\mathfrak{g} = \mathfrak{k} \otimes \CC$ of $G=K_\CC$, as $\mathfrak{k}^* = \{\xi \in \mathfrak{g}^*: \xi(\mathfrak{k}) \subseteq \RR\}$. Using this identification we can regard $\mu: X \to \mathfrak{g}^*$, or rather the $G$-orbit of $\mu$ parametrised by  the maximal compact subgroups $\{g^{-1}Kg:g \in G\}$ of $G$, as a \lq moment map'  for the action of $G$ (with respect to the $G$-equivariant K\"{a}hler structure on $X$ given by the family of K\"{a}hler forms $\{g^*\omega:g\in G\}$); of course it is not a moment map for $G$ in the traditional sense of symplectic geometry. This point of view does not offer anything very new in the case of reductive group actions, but it allows a useful generalisation of the concept of a moment map to non-reductive complex linear algebraic group actions.

Moment maps have been used since the 1970s to study the geometry and topology of GIT quotients $X/\!/G$ by reductive group actions. The second author in \cite{francesthesis} used the norm-square $f=||\mu||^2$ of the moment map $\mu:X\to \mathfrak{k}^*$ as a  (\lq minimally degenerate') Morse function and showed that the open stratum of the corresponding Morse stratification, which retracts $K$-equivariantly onto the zero level set $\mu^{-1}(0)$ of the moment map, coincides with the GIT semistable locus $X^{ss,G}$ for the linear action of $G$ on $X$. Moreover, $\mu^{-1}(0)$ is $K$-invariant and its inclusion in $X^{ss,G}$ induces a homeomorphism %in good cases when $X^s=X^{ss}$ one gets the fundamental relation  
\begin{equation}\label{fundamental}
\mu^{-1}(0)/K  \cong X/\!/G.
\end{equation} 
The equivariant Morse inequalities for $||\mu||^2$ relate the equivariant Betti numbers of $X$ to those of $X^{ss,G}$ and the other (unstable) strata. In fact it is shown in  \cite{francesthesis}  that this Morse stratification is equivariantly perfect, so that the equivariant Morse inequalities are equalities, and they provide explicit formulas for the $G$-equivariant Betti numbers of $X^{ss,G}$ in terms of the $G$-equivariant Betti numbers of $X$ and those of the other Morse strata, which can in turn be described inductively in terms of semistable loci for actions of reductive subgroups of $G$ on nonsingular projective subvarieties of $X$. %These Morse strata, in turn, coincide with the GIT for the $G$ action on $X$ strata, according to \cite{francesthesis}.
 If the stable locus $X^{s,G}$ coincides with $X^{ss,G}$ then the $G$-equivariant rational cohomology of $X^{ss,G}$ is isomorphic to the ordinary rational cohomology of its geometric quotient $X^{ss,G}/G$, which is the GIT quotient $X/\!/G$, and therefore we get expressions for the Betti numbers of $X/\!/G$ in terms of the equivariant Betti numbers of the unstable GIT strata and of $X$.

Martin \cite{SM} used \eqref{fundamental} to relate the topology of the quotient $X/\!/G$  and the associated quotient $X/\!/T$, where $T$ is a maximal (complex) torus of $G$. He proved a formula expressing the rational cohomology ring of $X/\!/G$ in terms of the rational cohomology ring of $X/\!/T$, and an integration formula which expresses intersection pairings on $X/\!/G$ in terms of intersection pairings on $X/\!/T$. This integration formula, combined with methods from abelian localisation, leads to %Jeffrey-Kirwan type 
residue formulas for cohomology pairings on $X/\!/G$ which are closely related to those of \cite{jeffreykirwan}.

In this paper we follow a similar path for suitable non-reductive actions.  Let $X$ be a nonsingular complex projective variety endowed with a linear action of a complex linear algebraic group $H=U \rtimes R$ with respect to an ample line bundle $L$. Here $U$ is the unipotent radical of $H$, and $R$ is the complexification of a maximal compact subgroup $Q$ of $H$ (so that $R$ and $Q$ are determined up to conjugation in $H$). %and maximal torus $T^H$ (which are then respectively a maximal compact subgroup and maximal torus of $H$). 
 Assume that the unipotent radical $U$ of $H$ is internally graded by a central 1-parameter subgroup
$\l: \CC^* \to Z(R)$  of the Levi subgroup $R$, and let $\hU=U \rtimes \l(\CC^*) \leqslant H$; then $\l(S^1) \leqslant \l(\CC^*) \leqslant \hU$ is a maximal compact subgroup of $\hU$. Assume also that semistability coincides with stability for the $\hU$-action, in the sense of \cite{BDHK2} (see Definition \ref{def:s=ss} below). In \cite{BDHK2} it is shown that then, after a suitable shift of the linearisation by a (rational) character of $H$ --- we call the resulting linearisation  well-adapted --- the algebra of invariants $\oplus_{k\ge 0}H^0(X,L^{\otimes k})^H$ is finitely generated, and the corresponding projective variety $X/\!/H$ satisfies some of the key properties of Mumford's GIT: (i) it is the set of $S$-equivalence classes for the $H$ action on a semistable locus $X^{ss,H}_{\min+}$, (ii) it contains as an open subvariety a geometric quotient of a stable locus $X^{s,H}_{\min+} \subset X^{ss,H}_{\min+}$, and (iii) the semistable and stable loci are described by Hilbert--Mumford type numerical criteria. In fact it suffices to prove these properties in the case that $H=\hU$, since they then follow in general by applying reductive GIT to the induced linear action of the reductive group $R/\l(\CC^*)$ on $X\env \hU$ and using the isomorphism
$X\env H \cong (X \env \hU)\env(R/\l(\CC^*))$ arising from the identity 
$ \oplus_{k\ge 0}H^0(X,L^{\otimes k})^H = (\oplus_{k\ge 0}H^0(X,L^{\otimes k})^{\hU})^{R/\l(\CC^*}.$

Using the embeddings $X \subseteq \PP^n$ defined by very ample tensor powers of $L$, and corresponding Fubini--Study K\"ahler forms $\omega$ invariant under maximal compact subgroups $Q$ of $H$, with associated $H$-orbits 
$$\Omega = \{ (h^{-1}Qh,h^*\omega):h \in H\}, $$ % = H(Q,\omega) \cong H/Q, $$
%or $H$-equivariant  families of such metrics parametrised by the maximal compact subgroups of $H$), %=\PP(\oplus_{i=1}^\infty H^0(L^{\otimes i}))$,
    in this paper we define \lq $\Omega$-moment maps' $$\mathfrak{m}_{H,X,\Omega}:\Omega \times X\to \mathfrak{h}^* = (\mathrm{Lie}H)^* = \Hom_\CC(\mathrm{Lie}H,\CC)$$
 for the action of $H$ on $X$ with respect to $\Omega$,   
     by composing  the restrictions to $X$ of traditional moment maps for unitary groups $K=U(n+1)$ acting on $\PP^n$ such that $Q = K \cap H$ with the inclusions
$\mathfrak{k}^* \to \mathfrak{k}^*_\CC$ and the maps of complex duals $\mathfrak{k}^*_\CC \to \mathfrak{h}^*$ coming from the representations $H \to K_\CC = \GL(n+1)$; cf. %\cite{K18} and 
Remark \ref{GrebMiebach} for links to \cite{grebmiebach}. %When $H$ is reductive this is compatible with the previous definition. 

 Let $\weight_{\min} = \weight_0 < \weight_1 <\ldots < \weight_r$ be the weights of the linear action of $\l(\CC^*)$ on $X$, in the sense that $\l(\CC^*)$ acts on the fibres of the line bundle $L^*$ over the fixed point set $X^{\lambda(\CC^*)}$ with these weights. Set 
\[Z_{\min}(X)=\left\{
\begin{array}{c|c}
\multirow{2}{*}{$x \in X$} & \text{$x$ is a $\l(\CC^*)$-fixed point and} \\ 
 & \text{$\l(\CC^*)$ acts on $L^*|_x$ with weight $\weight_{0}$} 
\end{array}
\right\}\]
and
\[X_{\min}^{0}=\left\{
\begin{array}{c|c}
\multirow{2}{*}{$x \in X$} & \text{$\lim_{t\to 0} \l(t)x$ lies } \\ 
 & \text{in $Z_{min}$} 
\end{array}
\right\}.\]
\noindent If $\weight_0 < 0 < \weight_1$ then $X^{s,\l(\CC^*)} = X^{ss,\l(\CC^*)} = X^0_{\min} \setminus Z_{\min}(X)$, and we set $X^{s,\hU}_{\min+} = \bigcap_{u \in U} uX^{s,\l(\CC^*)} = X^0_{\min} \setminus UZ_{\min}(X).$ Provided that the action of $ \l(\CC^*)$ is nontrivial so that $r \geq 1$, the condition that $\weight_0 < 0 < \weight_1$ can be achieved (without changing the action of $H$ on $X$) by multiplying the linearisation by a suitable (rational) character; this corresponds to adding a suitable constant to the $\Omega$-moment map restricted to $\{(K,\omega)\}\times X$ for $(K,\omega) \in \Omega$.
Similarly we make a linearisation  \emph{well-adapted} to the action of  $H$ on $X$ by multiplying it by a rational character so that $0 \in (\weight_0,\weight_1)$ is sufficiently close to $\weight_0$ (see Definition \ref{def:welladaptedaction}). %Provided that the action of $ \l(\CC^*)$ is nontrivial, this can be achieved (without changing the action of $H$ on $X$) by multiplying the linearisation by a suitable (rational) character; this corresponds to adding a central constant to the $\Omega$-moment map.
Our first theorem is the following non-reductive analogue of \eqref{fundamental}.
\begin{theorem}[\textbf{Moment map description of the GIT quotient}] \label{thm:mainA} 
Let $X\subseteq \PP^n$ be a smooth complex projective variety endowed with a well-adapted action of a complex linear algebraic group $H=U \rtimes R \leqslant \GL(n+1,\CC)$ with internally graded unipotent radical $U$ and grading one-parameter subgroup $\lambda:\CC^* \to Z(R)$. % and maximal compact subgroup $Q\leqslant R$. 
Let $\Omega$ be an $H$-equivariant K\"ahler structure on $X$ consisting of Fubini--Study K\"ahler forms $\omega$ invariant under unitary groups $K$ containing maximal compact subgroups $Q = K\cap H$ of $H$, with $\Omega$-moment map $\mathfrak{m}_{H,X,\Omega}:\Omega \times X\to \mathfrak{h}^*$ 
(see Definition \ref{defn:omegaH}).
If $H$-stability coincides with $H$-semistability in the strong sense of Definition \ref{def:s=ss}, then the stable and semistable loci for the $H$-action satisfy
$$ X^{s,H}_{\min+} = X^{ss,H}_{\min+} = \{ x \in X:0 \in \mathfrak{m}_{H,X,\Omega}(\Omega \times \{x\})\} 
= H(\mu_{(K,\omega)}^{H})^{-1}(0) $$
 for any $(K,\omega) \in \Omega$, where $\mu_{(K,\omega)}^H:X \to \mathfrak{h}^*$ sends $x \in X$ to $\mathfrak{m}_{H,X,\Omega}(K,\omega,x)$. Moreover 0 is a regular value of $\mu_{(K,\omega)}^{H}$ and
 the embedding  of $(\mu_{(K,\omega)}^{H})^{-1}(0) $ in $X^{s,H}_{\min+}$ induces a diffeomorphism of orbifolds
\[(\mu_{(K,\omega)}^{H})^{-1}(0)/(K \cap H) \to \{ x \in X:0 \in \mathfrak{m}_{H,X,\Omega}(\Omega \times \{x\})\} /H = X^{s,H}_{\min+}/H = X/\!/H .\] 
\end{theorem}

In the special case when $H=\hU=U \rtimes \lambda(\CC^*)$, this tells us that the embedding $(\mu_{(K,\omega)}^{\hU})^{-1}(0) \hookrightarrow X^{s,\hU}_{\min+}$ induces a %surjective local 
diffeomorphism of orbifolds  $(\mu_{(K,\omega)}^{\hU})^{-1}(0)/\lambda(S^1) \to X/\!/\hU$. 

Using Theorem \ref{thm:mainA} we prove three theorems about the non-reductive GIT quotient $X\env H$.

\begin{theorem}[\textbf{Betti numbers}]\label{thm:mainB} Let $X$ be a smooth complex projective variety endowed with a well-adapted action of a linear algebraic group $H=U \rtimes R$ with internally graded unipotent radical $U$ and grading one-parameter subgroup $\l:\CC^* \to Z(R)$.  Assume  that semistability coincides with stability for $H$ in the strong sense of  Definition \ref{def:s=ss}. 
%\begin{enumerate} \item 
Then the stratification $X_{\min}^{0}=X^{s,\hU} \sqcup UZ_{\min}(X)$ is equivariantly perfect for the actions of $\hU = U \rtimes \l(\CC^*)$ and $H$, and the Poincar\'e series of the GIT quotients  $X/\!/\hU$ and $X\env H$ satisfy 
\[P_t(X/\!/\hU)=P_t(Z_{\min}(X))\frac{1-t^{2(\dim(X)-\dim(Z_{\min}(X))-\dim(U))}}{1-t^2}\]
%\item 
%If in addition stability=semistability  holds in the sense of classical GIT for the $R/\l(\CC^*)$ action on $Z_{\min}(X)$, then 
and
\[P_t(X/\!/H)=P_t(Z_{\min}(X)/\!/(R/\l(\CC^*)))\frac{1-t^{2(\dim(X)-\dim(Z_{\min}(X)) - \dim(U))}}{1-t^2}.\]
%\end{enumerate}
\end{theorem}
Note that %a key difference from the reductive picture of \cite{francesthesis} is that
 for well-adapted actions on nonsingular varieties we have a distinguished subset $Z_{\min}(X)$ with an action of a reductive group $R/\l(\CC^*)$ which -- as the formulas make clear -- carries most of the cohomological information about $X/\!/\hU$ and $X/\!/H$. 

\begin{remark} 
Generalisations of this theorem and its consequences to situations when semistability does not coincide with stability in the strong sense of Definition \ref{def:s=ss} will be studied in \cite{toappear}.
\end{remark}

In the second half of the paper we adapt the argument of Martin \cite{SM} described above to actions of complex linear algebraic groups with internally graded unipotent radicals, and extend his formulas to express the rational cohomology ring of $X/\!/H$ in terms of the rational cohomology ring of the torus quotient $X/\!/T^H$ where $T^H$ is a maximal (complex) torus of $R$ (and hence of $H$). We study the following key diagram, where $(K,\omega) \in \Omega$ and $T^Q$ is a maximal (compact) torus of  a maximal compact subgroup $Q=K\cap H$ of $R$, with 
complexification $T^H = (T^Q)_{\CC}$:
\begin{equation}\label{diagrammartinintro}
\xymatrix{(\mu_{(K,\omega)}^{H})^{-1}(0)/T^Q \,\,\, \ar@{^{(}->}[r]^-{i} \ar[d]^{\pi} \,\,\, & (\mu_{(K,\omega)}^{T^H})^{-1}(0)/T^Q=X/\!/T^H  \\ 
%\mu_{H,X,Q,\omega}^{-1}(0)/Q=
X/\!/H}.
\end{equation}
Here $\pi$ is the composition of the  %surjective local
 diffeomorphism 
$(\mu_{(K,\omega)}^{H})^{-1}(0)/Q \to X/\!/H$ induced by the 
 embedding  of $(\mu_{(K,\omega)}^{H})^{-1}(0)$ in $X^{ss,H}_{\min+}$ with the natural map $(\mu_{(K,\omega)}^{H})^{-1}(0)/T^Q \to (\mu_{(K,\omega)}^{H})^{-1}(0)/Q$.

\begin{theorem}[\textbf{Cohomology rings}]\label{thm:maincohomology} Let $X$ be a smooth complex projective variety endowed with a well-adapted action of $H=U \rtimes R$ such that $H$-stability=$H$-semistability holds in the strong sense of Definition \ref{def:s=ss}. Then there is a natural ring isomorphism
\[H^*(X/\!/H,\QQ)\simeq \frac{H^*(X/\!/T^H,\QQ)^W}{ann(e)}\]
where $W$ denotes the Weyl group of $R$, which acts naturally on $X/\!/T^H$, while $e=\mathrm{Euler}(V^*) \in H^*(X/\!/T^H)^W$ is the Euler class of the dual $V^*$ of the bundle associated to the roots of $R$ and weights of the adjoint action of $T^H$ on $\Lie(U)$ (see Definition  \ref{def:bundles} for details), and
\[\mathrm{ann}(e)=\{c \in H^*(X/\!/T,\QQ)\, |\, c \cup e=0\} \subset H^*(X/\!/T^H,\QQ).\]
is the annihilator ideal. 
\end{theorem}

Diagram \eqref{diagrammartinintro} provides a natural way to define a lift of a cohomology class on $X/\!/H$ to a class on $X/\!/T^H$: we say that $\tilde{a}\in H^*(X/\!/T^H)$ is a lift of $a\in H^*(X/\!/H)$ if $\pi^*a=i^*\tilde{a}$.

\begin{theorem}[\textbf{Integration formula}]\label{thm:mainintegration} Let $X$ be a smooth projective variety endowed with a well-adapted action of $H=U \rtimes R$ such that $H$-stability=$H$-semistability holds  in the strong sense of Definition \ref{def:s=ss}. %Assume that the stabiliser in $H$ of a generic $x \in X$ is trivial.
 Given a cohomology class $a \in H^*(X/\!/H)$ with a lift $\tilde{a}\in H^*(X/\!/T^H)$, then 
\[\int_{X/\!/H}a=\frac{n_0}{|W|}\int_{X/\!/T^H} \tilde{a} \cup e\]
where we use the notation of Theorem \ref{thm:maincohomology}, and $n_0>0$ is determined by  the sizes of the stabilisers in $H$ and $T^H$ of a generic $x \in X$. % and the size of a generic fibre of the  surjective local diffeomorphism  $(\mu_{(K,\omega)}^{H})^{-1}(0)/Q \to X/\!/H$ induced by the  embedding  of $(\mu_{(K,\omega)}^{H})^{-1}(0)$ in $X^{ss,H}$.
\end{theorem}

Using these results, we prove a residue formula for intersection  pairings on $X/\!/H$; see Theorem \ref{jeffreykirwannonred}.

The layout of the paper is as follows. We start in \S \ref{sec:applications} with an outline of some applications of our results, followed in \S \ref{sec:redGIT} and \S \ref{sec:nrgit} by  brief overviews of the relationship between classical GIT and moment maps and then   of 
 non-reductive GIT and results of \cite{BDHK,BDHK2} which will be needed later, proving a modified version (Theorem \ref{mainthm}) and a generalisation (Theorem \ref{mainthmextended}) of the main result  of \cite{BDHK2}. In \S \ref{sec:momentmaps} we define moment maps for linear actions of complex linear algebraic groups with internally graded unipotent radicals and prove  Theorem \ref{thm:mainA}. %, first for $\hU$-actions and then for more general $H$-actions.
  In \S \ref{sec:betti} we %study the Morse stratification of $X$ and
  prove Theorem \ref{thm:mainB}, while  \S \ref{sec:martin} uses the topological argument of \cite{SM} in combination with Theorem \ref{thm:mainA} to prove Theorems \ref{thm:maincohomology} and \ref{thm:mainintegration}.  
  
The authors would like to thank the anonymous referees for very helpful comments on earlier versions of this paper.

\section{Applications}\label{sec:applications}

The results of this paper, combined with ideas discussed in \cite{BDHK0},
provide machinery for studying cohomological intersection theory on moduli spaces constructed as quotients of suitable quasi-projective varieties by non-reductive linear algebraic group actions, even without assumptions that semistability coincides with stability. When a moduli space can be described as an open subset of a projective variety $X$ up to a symmetry group $H=U \rtimes R$ with internally graded unipotent radical $U$, it is often possible as in \cite{BDHK0} to construct a non-reductive GIT quotient $\tilde{X}/\!/H$ of an $H$-equivariant blow-up $\tilde{X}$ of $X$ for which semistability does coincide with stability, giving a projective completion of the moduli space which is canonically associated to the linear action of $H$ on $X$. When $X$ is nonsingular $\tilde{X}$ can be chosen to be nonsingular, with the induced action of $H$ on $\tilde{X}$ well-adapted and $H$-semistability coinciding with $H$-stability. Then Theorems \ref{thm:mainA}--\ref{thm:mainintegration} tell us that 
this projective completion
$\tilde{X}/\!/H$ has the best cohomological intersection theory over $\QQ$ we can hope for: as in the case of quotients by reductive groups \cite{francesthesis,SM} its rational cohomology and cohomological intersection numbers can be computed from the $T$-action on the normal bundles to the connected components of the $T$-fixed point set, where $T \subset R$ is a maximal torus in $H$. %In short, an analogue of Shaun Martin's reductive abelianisation \cite{SM} works for non-reductive GIT quotients. 

As a first important application of this theory, in this section we outline the proof of the polynomial Green--Griffiths--Lang and polynomial Kobayashi conjectures, which is given in detail  in the companion paper \cite{bkGGL}.

A projective variety $X$ is called Brody hyperbolic if there is no non-constant entire holomorphic curve in $X$; i.e. any holomorphic map $f: \CC \to X$ must be constant. Hyperbolic algebraic varieties have attracted considerable attention, in part because of their conjectured diophantine properties. For instance, Lang \cite{lang} has conjectured that any hyperbolic complex projective variety over a number field $K$ can contain only finitely many rational points over $K$. In 1970 Kobayashi \cite{kob2} formulated the following conjecture.
\begin{conjecture}[\textbf{Kobayashi conjecture, 1970}]
A generic hypersurface $X\subseteq \PP^{n+1}$ of degree $d_n$ is Brody hyperbolic if $d_n$ is sufficiently large.
\end{conjecture}

\begin{remark}
This is a slightly stronger version of Kobayashi's original conjecture in which \lq very general' replaced \lq generic'.
\end{remark}

This conjecture has a vast literature, and for more details on recent results we recommend the survey papers \cite{demsurvey,dr}. The conjectured optimal degree bound is $d_1=4, d_n=2n+1$ for $n=2,3,4$ and $d_n=2n$ for $n\ge 5$, see \cite{demsurvey}. Siu \cite{siu4} and Brotbek \cite{brotbek} proved  Kobayashi hyperbolicity for projective hypersurfaces of sufficiently high (but not effective) degree. Based on the work of Brotbek effective degree bounds were worked out by Deng \cite{deng}, Demailly \cite{demsurvey}, Merker and The-Anh Ta \cite{merker2}. The best known bound based on these existing techniques is $(n \log n)^n$.     

A related conjecture is the Green--Griffiths--Lang (GGL) conjecture formulated in 1979 by Green and Griffiths \cite{gg} and in 1986 by Lang \cite{lang}. 
\begin{conjecture}[\textbf{Green-Griffiths-Lang conjecture, 1979}]
Any projective algebraic variety $X$ of general type contains a proper algebraic subvariety $Y\subsetneqq X$ such that every
nonconstant entire holomorphic curve $f:\CC \to X$ satisfies $f(\CC) \subseteq Y$. 
\end{conjecture}

In particular, a generic projective hypersurface $X\subseteq \PP^{n+1}$ is of general type if $\deg(X)\ge n+3$.    
The GGL conjecture has been proved for surfaces by McQuillan \cite{mcquillan} under the assumption that the second Segre number $c^2_1-c_2$ is positive. Following the strategy developed by Demailly \cite{dem} and Siu \cite{siu1,siu2,siu3,siu4}  for generic hypersurfaces $X\subseteq \PP^{n+1}$ of high degree, and using techniques of Demailly \cite{dem}, the first effective lower bound for the degree of a generic hypersurface in the GGL conjecture was given by Diverio, Merker and Rousseau \cite{dmr}, where the conjecture for generic projective hypersurfaces $X\subseteq \PP^{n+1}$ of degree $\deg(X)>2^{n^5}$ was confirmed. 
The current best bound for the Green--Griffiths--Lang Conjecture is $\deg(X)>(\sqrt{n}\log n)^n$ due to Merker and The-Anh Ta \cite{merker2}. 

The two hyperbolicity conjectures are strongly related: Riedl and Yang \cite{riedl} recently showed that if there are integers $d_n$ for all positive $n$ such that the GGL conjecture for hypersurfaces of dimension $n$ holds for degree at least $d_n$ then the Kobayashi conjecture is true for hypersurfaces with degree at least $d_{2n-1}$.

The key object in Demailly's strategy is the jet bundle $\pi_k:J_kX \to X$ of order $k$, whose elements are holomorphic $k$-jets of maps $f: (\CC,0) \to X$ with $\pi_k(f)=f(0)$, and its open subset $(J_kX)^{\circ}$ consisting of $f: (\CC,0) \to X$ with $f'(0) \neq 0$. The fibre of $J_kX$ at $x \in X$ is the vector space formed by $k$-tuples $(f'(0),f''(0),\ldots, f^{(k)}(0))$ of derivatives for holomorphic $f: \CC \to X$ with $f(0)=x$. Polynomial reparametrisations $(\CC,0)\to (\CC,0)$ of order $k$ with nonzero first derivative form a group $\Diff_k$, and this group acts on $J_kX$ by reparametrisation of jets: $\alpha \in \Diff_k$ sends $f$ to $f \circ \alpha$.  The quotient stack $[J_kX/\Diff_k]$ is the moduli stack of $k$-jets in $X$, and it plays a central role in several classical problems in geometry, and in particular in hyperbolicity. 

The main technical difficulty which arises in the study of $[J_kX/\Diff_k]$ is that the reparametrisation group $\Diff_k$ is not reductive. 
In \cite{bkGGL} we use the non-reductive GIT quotient $\tilde{J_kX}/\!/\Diff_k(1)$  of a suitable blow-up of a compactification of $J_kX$, together with sections of the tautological bundle over $\tilde{J_kX}/\!/\Diff_k(1)$, which give (semi-)invariant jet differentials. We use the cohomological intersection theory developed in this paper to prove the hyperbolicity conjectures with polynomial degree bound.

\begin{theorem}[\textbf{Polynomial Green-Griffiths-Lang theorem for projective hypersurfaces \cite{bkGGL}}] \label{mainthmtwo}
Let $X\subseteq \PP^{n+1}$ be a generic smooth projective hypersurface
of degree $\deg(X)\ge 16n^3(5n+4)$. Then there is a proper algebraic subvariety $Y\subsetneqq X$ containing all nonconstant entire holomorphic curves in $X$. 
\end{theorem}  

Using the recent results of Riedl and Yang \cite{riedl} this implies 

\begin{theorem}[\textbf{Polynomial Kobayashi theorem \cite{bkGGL}}] \label{mainthmtwob}
A generic smooth projective hypersurface $X\subseteq \PP^{n+1}$  
of degree $\deg(X)\ge 16(2n-1)^3(10n-1)$ is Brody hyperbolic. 
\end{theorem}

\begin{remark}
Polynomial reparametrisation groups and jet moduli spaces play a central role not just in hyperbolicity questions, but in several other classical enumerative geometry problems. The reason for this is that the curvilinear component of the Hilbert scheme of points on a smooth variety naturally forms a compactification of a  jet moduli space constructed as a non-reductive quotient as explained below. Here we collect the main directions of ongoing work which use the results of this paper. 

\noindent (i) \textsl{Thom polynomials.} 
Let $f: N \to M$ be a holomorphic map between two complex manifolds, $N$ and $M$, of dimensions $n\le m$, and let $A$ be a complex nilpotent algebra, finite dimensional over $\CC$.  A fundamental problem in global singularity theory is the characterisation of the locus $\Sigma_A$ in $N$ where $f$ has local algebra $A$ (see \cite{arnold}). Assuming $N$
is compact and $f$ is sufficiently generic, the closure $\overline{\Sigma_A}$ is an analytic subvariety representing a homology cycle. Thom's pioneering work \cite{thom} in the 50's showed that the Poincar\'e dual class $[\overline{\Sigma_A}]\in H^*(N,\mathbb{Z})$ is a polynomial $\Tp_A^{n,m}(c_1,c_2,\ldots)$ in the Chern classes of the difference bundle $TN-f^*TM$. This is now called \textit{the Thom polynomial of the singularity}; its full description has been a major open problem ever since then (see \cite{rimanyi,kazarian2}). 
In \cite{bsz} the Morin case is studied; this is the case  when $A=t\CC[t]/t^{k+1}$ for some $k$. 

Let $J_k(1,n)$ denote the space of $k$-jets of holomorphic map germs $(\mathbb{C},0) \to (\mathbb{C}^n,0)$. This space is acted on by the polynomial reparametrisation group $\Diff_k=J_k^{reg}(1,1)$ formed by jets with nonzero first derivative. The main message of \cite{bsz} is that the {Thom polynomial can be regarded as an equivariant intersection number on a projective compactification of the quasi-projective non-reductive quotient $J_k(1,n)/\Diff_k$}. 

In \cite{b0} the first author showed that the projective completion of $J_k(1,n)/\Diff_k$ used in \cite{bsz} for Thom polynomials is (a partial blow-up of) the curvilinear component $\CHilb^{k+1}(\CC^n)$ of the punctual Hilbert scheme of points on $\CC^n$. In \cite{berczithom} the relationship is explored between the two birational compactifications of $J_k(1,n)^\circ /\Diff_k$ given by $\CHilb^{k+1}(\CC^n)$ and the non-reductive GIT quotient $J_k(1,n)/\!/\Diff_k$. In particular, a blow-up $\widehat{J_k(1,n)/\!/\Diff_k}$ of the GIT quotient is constructed with a birational morphism $\widehat{J_k(1,n)/\!/\Diff_k} \to \CHilb^{k+1}(\CC^n)$ so that the integral formulas of the present paper lead to new Thom-polynomial formulas (\cite{berczithom}).

\noindent (ii) \textsl{Multisingularity classes.} These are multipoint versions of Thom polynomials, but our knowledge about them is much more limited, because the study of their geometry naturally involves some understanding of the full Hilbert scheme of $M$, not just its curvilinear component. To every holomorphic map $f:N\to M$ and tuple of nilpotent algebras $A=(A_1,\ldots, A_s)$ one can assign the multisingularity cycle $\Sigma_A$, sitting in the Cartesian product $N^s=N\times \ldots \times N$ as the closure of the set of tuples $(x_1,\ldots, x_s) \in  N^s$ of pairwise distinct points such that $f(x_1)=\ldots =f(x_s)$ and $f$ has local algebra isomorphic to $A_i$ at $x_i$ for $i=1,\ldots, s$. The multisingularity class $n_A \in H^*(N)$ is the dual of the image of $\Sigma_A$ under the projection of $N^s$ to the first factor. Multipoint formulas (or $s$-fold point formulas) represent the simplest situation when $A_i$ is trivial for $i=1, \ldots, s$ and $n_A=n_s$ is the $s$-fold locus of the map. 
Computation of multisingularity classes is a classical problem (see e.g \cite{kleiman1,kleiman2,katz,kazarian,rimanyi2}). Kazarian \cite{kazarian} and Rim\'anyi \cite{rimanyi2} formulate several important structure theorems and conjectures.

 In \cite{bsz3}  multipoint classes are related to certain tautological integrals over the main component $\GHilb^k(N)$ (which is the closure of the open locus where the $k$ points are different), and we reduce integration to the curvilinear component $\CHilb^k(N)$. As a result,  iterated residue formulas for multisingularity classes are developed, and conjectures of Kazarian and Rim\'anyi on the structure of these polynomials are proved.

\noindent(iii)  \textsl{Tautological integrals on Hilbert scheme of points in higher dimension.}
The Hilbert scheme of $k$ points $\Hilb^k(X)$ on a smooth variety $X$ plays a central role in several classical problems in enumerative geometry. When $X=S$ is a surface, the Hilbert scheme is a nonsingular $2k$ dimensional variety, which has been extensively studied over the last 50 years; it occupies a central position in many important branches of mathematics and physics (see e.g \cite{nakajima}). While Hilbert schemes on surfaces are well-understood and broadly studied, the Hilbert schemes of points on manifolds of dimension three or higher are extremely wild objects with several unknown irreducible components and bad singularities \cite{vakil}, and our understanding of them is very limited. 

However, for most enumerative geometry applications it is enough to study the main component $\GHilb^k(X)$. % (which is the closure of the open locus where the $k$ points are different).
 For example, according to G\"ottsche \cite{gottsche}, tautological integrals on certain geometric subsets of $\GHilb^k(X)$ give counts of hypersurfaces in generic linear systems over $X$ with given sets of singularities. A special case of these integrals gives the famous nodal curve count formula of G\"ottsche on surfaces. 

Unfortunately, the punctual part of $\GHilb^k(X)$ is not irreducible, but the curvilinear component $\CHilb^k(X)$ turns out to be a distinguished component. The theory developed in  \cite{berczitau2,berczitau3,berczitau4} reduces integration over $\GHilb^k(X)$ to integration over $\CHilb^{k}(X)$. As we have seen, this latter is a natural compactification of the non-reductive quotient $J_{k-1}X/\Diff_k$, and the results of the present paper are used in \cite{berczithom} to derive tautological integral formulas over $\CHilb^{k}(X)$.
 
 \noindent (iv) \textsl{Moduli and cohomology of unstable Higgs bundles.} Ongoing work studies the cohomology of moduli spaces of unstable Higgs bundles of low rank constructed as non-reductive GIT quotients \cite{hamilton}. 
%\end{enumerate}
\end{remark}

\section{Reductive GIT and moment maps}\label{sec:redGIT}

Mumford's geometric invariant theory (GIT) \cite{GIT} for reductive group actions in (complex) algebraic geometry is closely related to the construction of quotients for compact group actions in symplectic geometry via symplectic reduction \cite{MarsdenWeinstein}. In this section we will recall how a GIT quotient of a complex projective variety $X$ by a complex reductive group action which is linear with respect to an ample line bundle $L$ (in the sense that the action on $X$ lifts to a linear action on $L$) can be identified with a symplectic reduction of $X$ by a maximal compact subgroup.

Suppose that a compact Lie group $K$ with Lie algebra $\mathfrak{k}$ acts smoothly on a symplectic manifold $X$ preserving its symplectic form $\omega$. For simplicity assume that $K$ is connected. If the vector field determining the infinitesimal action of $a \in \mathfrak{k}$ on $X$ is denoted by $x \mapsto a_x$, then a moment map for the action of $K$ on $X$ is a smooth $K$-equivariant map $\mu: X \to \liek^*$ satisfying
$$ d\mu(x)(\xi).a = \omega_x(\xi,a_x)$$
for all $x \in X, \xi \in T_xX$ and $a \in \liek$; that is, the component $\mu_a:X \to \RR$ given by $x \mapsto \mu(x)\cdot a$ of $\mu$ along any $a\in \liek$ is a Hamiltonian function of the vector field $x \mapsto a_x$ on $X$ induced by $a$. This means that if $X$ is K\"ahler then the gradient flow of $\mu_a$ is given by
\begin{equation} \label{doubledagger} \mathrm{grad} \mu_a (x) = i a_x \end{equation}
for all $a \in \liek$.

\begin{remark} \label{morsebott} If $a\in \liek$ then the component $\mu_a:X \to \RR$ of $\mu$ along $a$ is a Morse--Bott function on $X$ (see \cite{atiyah}); that is, the connected components of the set of its critical points (which is just the fixed point set of the action on $X$ of the subtorus  $T_a = \overline{\exp(\RR a)}$ of $K$ generated by $\exp(\RR a)$) are submanifolds of $X$, and the restriction of the Hessian of $\mu_a$ at any point $p$ of one of these connected components $C$ to the normal to $C$ is nondegenerate. Indeed we can find Darboux coordinates $x=(x_1, y_1, \ldots ,x_{m},y_m)$ in a neighbourhood of $p$ with $x(p)=0$, such that $C$ is given locally by 
$\{x: x_j = y_j = 0 \mbox{ when } j>k   %\chi_j \neq 0
 \}$,
the torus $T_a$ acts linearly with weights $\chi_1,\ldots,\chi_m$ where $\chi_j = 0$ if and only if  $j>k$, and locally $\mu_a$ is given in these coordinates by
$$ \mu_a(x) = \mu_a(0) + (\chi_1\cdot a) (x_1^2 + y_1^2) + \cdots + (\chi_k \cdot a) (x_k^2 + y_k^2).$$
In particular note that if $\mu_a$ takes its minimum value on $C$,  then $ \mu_a(x) = \mu_a(0) + (\chi_1\cdot a) (x_1^2 + y_1^2) + \cdots + (\chi_m \cdot a) (x_k^2 + y_k^2)$ where $ \chi_j \cdot a > 0$ for $j=1, \ldots , k$.
\end{remark} 

If the stabiliser $K_\zeta$ of $\zeta \in \liek^*$ with respect to the coadjoint action of $K$ acts freely on $\mu^{-1}(\zeta)$, then $\mu^{-1}(\zeta)$ is a submanifold of $X$ and the restriction to $\mu^{-1}(\zeta)$ of the symplectic form $\omega$ becomes degenerate along the orbit directions and induces a symplectic structure on $\mu^{-1}(\zeta)/K$ \cite{MarsdenWeinstein}. When the action of $K_\zeta$ on $\mu^{-1}(\zeta)$ is not assumed to be free then the symplectic reduction $\mu^{-1}(\zeta)/K_{\zeta}$ at $\zeta$ of the $K$-action on $X$ may have singularities but still inherits a stratified symplectic structure \cite{Sjamaar}. The reduction $\mu^{-1}(0)/K_{\zeta}$ at 0 is sometimes called the symplectic quotient of $X$ by the action of $K$; when $K$ acts on $\mu^{-1}(0)$ with at worst finite stabilisers (or equivalently 0 is a regular value of $\mu$) then the symplectic quotient has only orbifold singularities.

Now suppose that $X$ is a complex projective variety (which for simplicity can be assumed to be nonsingular and connected) embedded in a complex projective space $\PP^n$, and let $G$ be a complex connected Lie group acting on $X$ via a complex linear representation $\rho:G \to \GL(n+1)$. Assume that $G$ is reductive; equivalently $G=K_\CC$ is the complexification of  a maximal compact subgroup $K$. Examples are given by the complexification $C^*$ of the circle $S^1$, and more generally by the complexifications $\GL(m)$ and $\SL(m)$ of the unitary and special unitary groups $\mathrm{U}(m)$ and $\SU(m)$. By choosing coordinates on $\PP^n$ appropriately we can assume that K acts unitarily, so that $\rho(K) \subseteq \mathrm{U}(n+1)$ and the standard Fubini--Study K\"ahler form $\omega$ on $\PP^n$ restricts to a $K$-invariant K\"ahler form on $X$. There is a moment map $\mu:X \to \lieks$ given (up to multiplication by a constant scalar depending on conventions) for $a \in \lieks$ by
\begin{equation} \label{mommap} \mu([x_0:\cdots :x_n]).a = \frac{(\bar{x}_0,\ldots , \bar{x}_n) \rho_*(a) (x_0,\ldots ,x_n)^T   }{ 2 \pi i || (x_0,\ldots ,x_n)||^2  },\end{equation}
and thus a symplectic quotient $\mu^{-1}(0)/K$.

Mumford's GIT \cite{GIT} was developed to construct quotients of algebraic varieties by reductive group actions in the presence of ample linearisations. Here an ample linearisation for an action of a reductive group $G$ on a projective variety $X$ is a lift of the action of $G$ on $X$ to an ample line bundle $L$. Taking sections of a sufficiently large power of $L$ gives us an embedding of $X$ in a  projective space $\PP^n$ such that $G$ acts on $X$ via a linear representation $\rho:G \to \GL(n+1)$ as above. There is an induced action of $G$ on the finitely generated graded algebra
$$\mathcal{O}_L(X) = \bigoplus_{k \geqslant 0} H^0(X,L^{\otimes k}),$$
and the subalgebra $\mathcal{O}_L(X)^G$ of $G$-invariants is itself a finitely generated graded algebra since $G$ is reductive. So we can define the GIT quotient $X/\!/G$ to be the projective variety $$ X/\!/G = \mathrm{Proj}(\mathcal{O}_L(X)^G)$$ associated to this algebra of invariants.

A categorical quotient of a variety $Y$ for an action of $G$ is a $G$-invariant morphism $\phi:Y \to Z$ such that any other $G$-invariant  morphism $\psi:Y \to W$ factors uniquely through $\phi$ (see e.g. \cite{Newstead} 2 \S 4). An orbit space for the action is a categorical quotient $\phi:Y \to Z$ such that every fibre of $\phi$ is a single $G$-orbit, and a geometric quotient is an affine morphism $\phi:Y \to Z$ which is an orbit space and satisfies

(i) if $U$ is open in $Z$ then $\phi^*:\mathcal{O}(U) \to \mathcal{O}(\phi^{-1}(U))$ induces an isomorphism of the algebra of regular functions $\mathcal{O}(U)$ on $U$ onto $\mathcal{O}(\phi^{-1}(U))^G$, and

(ii) if $W_1$ and $W_2$ are disjoint closed $G$-invariant subvarieties of $Y$ then their images under $\phi$ are disjoint closed subvarieties of $Z$.

When $G$ acts linearly on a projective variety $X$ as above, then the inclusion of the algebra of invariants $\mathcal{O}_L(X)^G$ in $\mathcal{O}_L(X)$ determines a rational map $\phi $ from $X$ to $X/\!/G$, but usually there will be points of $X$ where every $G$-invariant polynomial vanishes and so this rational map will not be well-defined everywhere on $X$. So we define the set $X^{ss}$ of \emph{semistable} points in $X$ to be the set of those $x \in X$ for which there exists some $f \in \mathcal{O}_L(X)^G$ which does not vanish at $x$, and then the rational map $\phi $ restricts to a surjective $G$-invariant morphism
$$\phi :X^{ss} \to X/\!/G$$
which is a categorical quotient for the action of $G$ on $X^{ss}$. It is not necessarily an orbit space; in fact when $x$ and $y$ lie in $X^{ss}$ then $\phi (x) = \phi (y)$ if and only if the closures of their $G$-orbits meet in $X^{ss}$. Every point of $X/\!/G$ is represented by a unique semistable $G$-orbit which is closed in $X^{ss}$.
A \emph{stable} point of $X$ (called \lq properly stable' in \cite{GIT}) is $x \in X^{ss}$ with a $G$-invariant open neighbourhood $U$ in $X^{ss}$ such that every $G$-orbit in $U$ is closed in $X^{ss}$ and has dimension $\dim G$. If $U$ is any $G$-invariant open subset of the set $X^s$ of stable points of $X$, then $\phi (U)$ is open in $X/\!/G$ and $\phi |_U:U \to \phi (U)$ is a geometric quotient for the action of $G$ on $U$. Thus $\phi (X^s) = X^s/G$ is a geometric quotient for the action of $G$ on $X^s$, and (provided that $X^s \neq \emptyset$ and $X$ is irreducible) the projective variety $X/\!/G$ is a compactification of $X^s/G$.

\begin{remark} \label{rational} It is useful to note that $X^s$, $X^{ss}$ and $X/\!/G$ are unchanged if the line bundle $L$ is replaced by a positive tensor power of itself, so we can formally work with rational linearisations $L^{\otimes \ell/m}$ for positive integers $\ell$ and $m$. 
\end{remark}

The stable and semistable loci 
$X^s$ and $X^{ss}$ are determined by the \emph{Hilbert--Mumford criteria} \cite{GIT,Newstead}, which can be expressed in the following form.

\begin{theorem} \label{HilbMum} Let $T$ be a maximal torus of $G$. Then 
\begin{enumerate}[(i)]
\item $x \in X$ is semistable (respectively stable) for the action of $G$ on $X$ if and only if every $gx$  in the $G$-orbit of $x$ is semistable (respectively stable) for the action of $T$ on $X$;

\item $x = [x_0: \ldots : x_n] \in X \subseteq \PP^n$ is semistable (respectively stable) for the action of $T$ acting diagonally on $\PP^n$ with weights $\alpha_0,\ldots,\alpha_n$ if and only if the convex hull
$\mathrm{Conv} \{\alpha_i: x_i \neq 0\}$
contains 0 (respectively contains 0 in its interior).
\end{enumerate}
\end{theorem}
Recall also that the categorical quotient map $\phi :X^{ss} \to X/\!/G$ is surjective with the property that if $x, y \in X^{ss}$ then $\phi (x) = \phi (y)$ if and only if the closures of the $G$-orbits of $x$ and $y$ meet in $X^{ss}$, and that every point of $X/\!/G$ is represented by a unique semistable $G$-orbit which is closed in $X^{ss}$; moreover $X/\!/G$ is a projective variety containing the geometric quotient $X^s/G$ as an open subset. 

The GIT quotient $X/\!/G$ and the stable and semistable loci 
$X^s$ and $X^{ss}$ can also be described in terms of the moment map $\mu:X \to \lieks$ given at (\ref{mommap}) % as above for $a \in \lieks$ by $$\mu([x_0:\cdots :x_n]).a = \frac{(\bar{x}_0,\ldots , \bar{x}_n) \rho_*(a) (x_0,\ldots ,x_n)^T   }{ 2 \pi i |\!| (x_0,\ldots ,x_n)|\!|^2  }$$
\cite{francesthesis}. Any $x \in X$ is semistable if and only if the closure of its $G$-orbit meets $\mu^{-1}(0)$, and $x$ is stable if and only if its $G$-orbit meets the regular part $\mu^{-1}(0)_{\mathrm{reg}}$ of $\mu^{-1}(0)$ where the derivative of $\mu$ is surjective. The inclusions of $\mu^{-1}(0)$ into $X^{ss}$ and of $\mu^{-1}(0)_{\mathrm{reg}}$ into $X^s$ induce homeomorphisms 
\begin{equation} \label{reductivemmap} \mu^{-1}(0)/K \to X/\!/G \mbox{ and  $\mu^{-1}(0)_{\mathrm{reg}}/K \to X^s/G$.} \end{equation}

\begin{remark} \label{VGIT} The GIT quotient $X/\!/G$ depends on the choice of an ample linearisation of the $G$-action on $X$. The theory of variation of GIT (VGIT) \cite{Dolg,Thaddeus} describes the way in which the quotient changes when the linearisation is varied. The space of possible (rational) linearisations can be identified with a convex subset of a %finite-dimensional
 vector space over $\QQ$, which is divided into convex \lq chambers' by walls in such a way that the GIT quotient is constant up to isomorphism in the interior of any chamber, and when a wall between chambers is crossed then the GIT quotient undergoes a particular kind of  birational transformation (sometimes called a Thaddeus flip). 
Understanding this picture reduces to the simple case when $G$ is the multiplicative group $\CC^*$ and we change the linearisation by multiplying it by a (rational) character $\chi: G \to \CC^*$ of $G$ \cite{Thaddeus}; it follows from (\ref{mommap}) that this corresponds to adding a constant to the moment map for the action of $K = S^1$.
Suppose that $\CC^*$ acts diagonally on $X \subseteq \PP^n$ with weights $r_0, \ldots, r_n$.
Then by the Hilbert--Mumford criteria (Theorem \ref{HilbMum}) the semistable locus is
$$\{[x_0:\ldots:x_n] \in X : \exists i,j \mbox{ such that } x_i \neq 0 \neq x_j \mbox{ and } r_i \leq 0 \leq r_j \}$$
while the stable locus is
$$\{[x_0:\ldots:x_n] \in X : \exists i,j \mbox{ such that } x_i \neq 0 \neq x_j \mbox{ and } r_i < 0 < r_j \},$$
and multiplying the linearisation by a rational character $s \in \QQ$ corresponds to replacing each $r_j$ with $r_j + s$. Thus the space of linearisations obtained in this way can be identified with $\QQ$, and it is decomposed by finitely many walls (which are just points) into finitely many intervals such that in the interior of each interval $X^{ss} = X^s$ and $X/\!/G = X^s/G$ are constant (and empty on the unbounded intervals). The semistable locus on a wall is the union of the semistable (or equivalently stable) loci on each side of the wall, whereas the stable locus on a wall is the intersection of the stable loci on each side. The inclusion of the stable locus on one side of a wall into the semistable locus on the wall induces a surjection morphism on the GIT quotients; this fails to be an isomorphism because $G$-orbits which were stable away from the wall become identified with $G$-orbits in their closures which are semistable on the wall.
\end{remark}

\section{Non-reductive GIT}\label{sec:nrgit}
This section is a brief summary of results in \cite{BDHK0,BDHK,BDHK2}, with the main focus on \cite{BDHK2} which is most relevant to this paper. We also extend the main result  of \cite{BDHK2} (Theorem \ref{mainthm} below) to a more general version (Theorem \ref{mainthmextended} below).

Let $H=U \rtimes R$ be a linear algebraic group acting linearly with respect to an ample line bundle $L$ on a projective variety  $X$ over an algebraically closed field $\kk$ of characteristic 0, which for the purposes of this paper we can take to be $\CC$. Here $U$ is the unipotent radical of $H$ and $R\cong H/U$ is a Levi subgroup of $H$. 
When $H=R$ is reductive, using classical geometric invariant theory (GIT) developed by Mumford in the 1960s \cite{GIT} as in $\S$\ref{sec:redGIT}, we can find $H$-invariant open subsets $X^s \subseteq X^{ss}$ of $X$ with a geometric quotient $X^s/H$ and  projective completion $X/\!/H \supseteq X^s/H$ which is the projective variety associated to the algebra  of invariants
 $\bigoplus_{k \geq 0} H^{0}(X,L^{\otimes k})^H$. %, which is a finitely generated graded algebra. 
The variety $X/\!/H$ is the image
of a surjective morphism $\phi$ from the open subset $X^{ss}$ of $X$ (consisting of the
semistable points for the action, while $X^s$ consists of the stable points)  such that if $x,y \in X^{ss}$ then $\phi(x) = \phi(y)$ if and only if the closures of the $H$-orbits of $x$ and $y$ meet in $X^{ss}$ (that is, $x$ and $y$ are \lq S-equivalent').  Moreover the subsets $X^s$ and $X^{ss}$ can be described using the Hilbert--Mumford criteria for (semi)stability (see Theorem \ref{HilbMum} in $\S$\ref{sec:redGIT}). 

Some aspects of  Mumford's GIT can be made to work when $H$ is not reductive (cf. for example 
\cite{BDHK,dorankirwan,F2,F1,GP1,GP2,KPEN,W}), although it cannot be extended directly to non-reductive linear algebraic group actions since the algebra of invariants $\bigoplus_{k \geq 0} H^{0}(X,L^{\otimes k})^H$ is not necessarily finitely generated as a graded algebra when $H$ is not reductive. 
We can still define (semi)stable subsets $X^{ss}$ and $X^s$, the latter having a geometric quotient $X^s/H$ which is an open subset of an \lq enveloping quotient' $X\env H$ with an $H$-invariant morphism $\phi: X^{ss} \to X\env H$, and if the algebra of invariants 
 $\bigoplus_{k \geq 0} H^{0}(X,L^{\otimes k})^H$ is finitely generated then $X\env H$ is the associated projective variety
 \cite{BDHK,dorankirwan}. However in general the enveloping quotient $X\env H$ is not necessarily projective, the morphism $\phi$ is not necessarily surjective (indeed its image may not be a subvariety of $X\env H$, but only a constructible subset) and we have no obvious analogues of the Hilbert--Mumford criteria for (semi)stability. Moreover there are subtleties in the definitions of the semistable and stable loci, with different candidates for these defined in \cite{BDHK,dorankirwan}
  (some of which are described in $\S$\ref{sec:4.1} below), although in good situations they coincide. These good situations include that under consideration in this paper, which is studied in \cite{BDHK2} and described in $\S$\ref{sec:4.2}.

\subsection{Enveloping quotients}\label{sec:4.1} 
Let $H$ be a linear algebraic group acting on a variety $X$ over $\kk$ equipped with a
linearisation with respect to an ample line bundle $L \to X$ (that is, a lift of the action of $H$ to $L$). 
Note that (as in Remark \ref{VGIT}) we can multiply any such linearisation by a character $\chi:H \to \GG_m$ of $H$ (where $\GG_m = \kk^*$ is the multiplicative group) to obtain a new linearisation of the same action of $H$ on $X$ with respect to the same line bundle $L$.

Elsewhere in this paper we will assume that $X$ is projective, but we will sometimes want to restrict to affine open subsets of $X$. So in this subsection  $X$ can be projective or affine, or more generally projective-over-affine; this is the most general situation in which classical GIT for reductive group actions works well. When $X$ is projective-over-affine and the line bundle $L \to X$ is ample, then the graded algebra $\bigoplus_{r \geq 0} H^0(X,L^{\ten r})$ is finitely generated; this is not necessarily the case when $X$ is only assumed to be quasi-projective.

In \cite{BDHK} a framework is described for constructing \lq enveloping quotients' $X\env H$ for such actions, based on \cite{dorankirwan} which covers the case when $H=U$ is unipotent, using a non-reductive version of GIT. We let
$S=\bigoplus_{r \geq 0} H^0(X,L^{\ten r})$ 
be the graded $\kk$-algebra of global sections of non-negative tensor powers of
$L$,  and let $S^H$ be the subalgebra consisting of $H$-invariant sections. 
There is an associated scheme $\proj(S^H)$; in general $S^H$ is not finitely generated as an algebra over $\kk$, and so $\proj(S^H)$ is not a variety, but if $S^H$ is finitely generated %as a $\kk$-algebra
 (which will be the case in the situation considered in this paper) then $\proj(S^H)$ is indeed a projective variety.

The inclusion $S^H
\hookrightarrow S$ defines an $H$-invariant rational 
map of schemes
%%% eq:GiFi1 %%%
\begin{equation} \label{eq:GiFi1}
\phi:X \dashrightarrow \proj(S^H),
\end{equation}
which is well defined on the open subset of $X$ consisting of points where some
invariant section of a positive tensor power of $L$ does not vanish. 
\begin{definition} \label{def:GiFi1} (\cite{BDHK} Definition 3.1.1)
Let $H$ be a linear algebraic group acting on a variety $X$ (which we will assume to be projective-over-affine) and $L \to X$ an ample
linearisation of the action. The \emph{naively semistable locus} is the open subset   \index{naively semistable}
\[
X^{\nss,H}:=\bigcup_{f \in I^{\nss,H}} X_f
\]
of $X$, where $I^{\nss,H}:=\bigcup_{r>0}H^0(X,L^{\ten r})^H$ is the set of invariant
sections of positive tensor powers of $L$ and $X_f$ is $\{x \in X: f(x) \neq 0\}$. The
\emph{finitely generated semistable locus} is the open subset  \index{finitely generated semistable}
\[
X^{\ssfg,H}:= \bigcup_{f \in I^{\ssfg,H}} X_f
\]
of $X^{\nss,H}$, where 
\[
I^{\ssfg,H}:=\left\{f \in \textstyle{\bigcup_{r>0}} H^0(X,L^{\ten
r})^H \mid (S^H)_{(f)} \text{ is a finitely generated $\kk$-algebra}\right\},
\]  
and $(S^H)_{(f)} = \{ s/f^r: s \in S^H, r \in \mathbb{N} \}$ is the localisation of the $\kk$-algebra $S^H$.  We will also write $X^{ss,H}$ for $X^{\ssfg,H}$.
\end{definition}

\begin{remark} \label{rem:GiFi1.1}
When $S^H$ is finitely generated as a $\kk$-algebra then so is $(S^H)_{(f)}$ for any $f \in \textstyle{\bigcup_{r>0}} H^0(X,L^{\ten
r})^H$, and so
$$X^{\ssfg,H} = X^{\nss,H}.$$
This will be the case in the situation considered in this paper. However
the finitely generated semistable locus $X^{\ssfg}$ is in general strictly
contained in $X^{\nss}$, since the subalgebra of invariant sections
can be non-noetherian even if $S$ is a finitely generated
$\kk$-algebra. When $H$ acts on an irreducible affine variety $X=\spec A$ and $L=\OO_X$
is equipped with the canonical $H$-linearisation, %defined by the trivial character $1:H \to \mb{G}_m$,
 then 
$X^{\nss}=X$ while $X^{\ssfg}$ is the union of all $X_f$ where $(A^H)_f$ is
finitely generated over $\kk$. It follows from  \cite[Proposition
2.10]{dk08} and  \cite{gro76}  that if $A^H$ is not 
finitely
generated (as is the case for the famous Nagata example) then   $\emptyset \neq
X^{\ssfg} \neq X^{\nss}$.
\end{remark} 

%The rational map of \eqref{eq:GiFi1} restricts to define a morphism on $X^{\ssfg}$ whose image is contained in the following open subscheme of $\proj(S^H)$.

%%% def:GiFi3 %%%
\begin{definition}
\label{def:GiFi3} %Let $H$ be a linear algebraic group and $H \act L \to X$ a linearisation of an $H$-variety $X$. 
(\cite{BDHK} Definition 3.1.6)
Let $H$ be a linear algebraic group acting on a variety $X$ (assumed to be projective-over-affine), and $L \to X$ an ample
linearisation of the action. 
The \emph{enveloping quotient} is the scheme   
%%% deg-0-localisation: first use of S_{(f)} %%%
\[ \label{deg-0-localisation}
X \env H:= \bigcup_{f \in I^{\ssfg}} \spec \left((S^H)_{(f)}\right)
\subseteq \proj(S^H) 
\]
where $S = \bigoplus_{r \geq 0} H^0(X,L^{\otimes r})$,
equipped with the canonical map $\phi:X^{\ssfg} \to X \env H$. The image
$\phi(X^{\ssfg})$ of this map is the \emph{enveloped quotient}.   
\end{definition}

\begin{remark} \label{envquot} The enveloping quotient $X\env H$ is a canonically defined reduced, separated 
scheme
locally of finite type over $\kk$. 
Suppose now that $X$ is projective. By \cite{BDHK} Corollary 3.1.21 the enveloping quotient is a projective variety if and only if the algebra of invariants $S^H$ is finitely generated where $S = \bigoplus_{r \geq 0} H^0(X,L^{\otimes r})$, and if so then the enveloping quotient is the projective variety $\proj(S^H)$ associated to the algebra of invariants. 
If $N$ is a normal subgroup of $H$ containing $U$ so that $H/N$ is reductive, and if $S^N$ is finitely generated, then since $S^H = (S^N)^{H/N}$, both $X\env N$ and $X\env H$ are projective and
$$X\env H = (X\env N) \env (H/N).$$
\end{remark}

\begin{definition} (\cite{BDHK} Definition 3.3.2)
Let $H=U\rtimes R$ be a linear algebraic group with unipotent radical $U$ acting linearly on a %n irreducible
 variety $X$ (assumed to be projective-over-affine) with respect to an ample line bundle $L \to X$. The \emph{stable locus} is the open subset  \index{stable}
\[
X^{\rms,H}:= \bigcup_{f \in I^{\rms,H}} X_f
\]
of $X^{\ssfg,H}$, where $I^{\rms,H} \subseteq \bigcup_{r>0} H^0(X,L^{\ten r})^H$ is the 
subset
of $H$-invariant sections satisfying the following conditions:
%%% itm:GiSt1.1 %%%
%\begin{enumerate}
%\item \label{itm:GiSt1.1-1} 

(i) the open set $X_f$ is affine,
%\item \label{itm:GiSt1.1-2}

(ii)  the action of $H$ on $X_f$ is closed with all
  stabilisers finite groups, and
%\item \label{itm:GiSt1.1-3} 

(iii) the restriction of the $U$-enveloping quotient map
 \[
\phi_U:X_f \to \spec((S^{U})_{(f)})
\]
is a principal $U$-bundle for the action of $U$ on $X_f$. 
%\end{enumerate}

Note that condition (i) is equivalent to $f \neq 0$ if $X$ is projective.
\end{definition}

The enveloping quotient 
$\phi:X^{ss,H} = X^{\ssfg,H} \to X \env H$ restricts to define a geometric quotient 
\[
X^{\rms,H} \to X^{\rms,H}/H = \phi(X^{\rms,H}) 
\]
which factorises through the restriction of the enveloping quotient map for $U$
in a natural way, giving the following commutative diagram, with
all inclusions open (see \cite{BDHK} $\S$3.3):
\[
\begin{tikzcd}[column sep=3.5em]
X^{\ssfg,H} \arrow{dd}{\phi} \arrow[r,phantom,"\supseteq"] & X^{\rms,H}
\arrow{d}{\phi_U}[swap]{\text{geometric quotient}} \arrow[r,phantom,"\subseteq"] & X^{\ssfg,U} 
\arrow{d}{\phi_U} \\
 & X^{\rms,H}/U \arrow{d}{\phi_R}[swap]{\text{geometric quotient}} \arrow[r,phantom,"\subseteq"] & X 
\env U \\
X \env H \arrow[r,phantom,"\supseteq"] & X^{\rms,H}/H &  
\end{tikzcd}
\] 
where $\phi_R$ is the restriction of the enveloping quotient map for the induced linear action of $R \cong H/U$ on $X\env U$.

\begin{remark} \label{remark8+}
Suppose that, as in Remark \ref{envquot}, that $X$ is projective, that $U \trianglelefteq N \trianglelefteq  H$ (so that $H/N$ is reductive) and that $S^N$ is finitely generated where $S = \bigoplus_{r \geq 0} H^0(X,L^{\otimes r})$. Suppose also that $X^{ss,N} = X^{s,N}$ so that $X\env N = X^{s,N}/N$. If $x \in X^{s,N}$ then $Nx$ is semistable for the induced action of $H/N$ on $X\env N = X^{s,N}/N$ if and only if there is some $r>0$ and $f \in H^0(X,L^{\otimes r})^H$ such that $f(x) \neq 0$ and $\phi_U:X_f \to \Spec((S^U)_{(f)})$ is a principal $U$-bundle, which happens if and only if $x \in X^{ss,H}$. In particular this implies that $x \in X^{ss,N} = X^{s,N}$, so $X_f \subseteq X^{s,N}$ and hence the action of N on $X_f$ is closed with all stabilisers finite. Then $Nx$ is stable for the induced action of $H/N$ on $X\env N$ if and only if there is such an $f$ such that the action of $H/N$ on $X_f/N$ is closed with all stabilisers finite, or equivalently the action of $H$ on $X_f$ is closed with all stabilisers finite. Thus
   $$ (X\env N)^{ss,H/N} = X^{ss,H}/N \mbox{ and } (X\env N)^{s,H/N} = X^{s,H}/N. $$
\end{remark}

We finish this subsection with a few technical statements needed later. 
Lacking a suitable reference, we include a proof of the first of these.

%%% lem:Co2Qu0 %%%
\begin{proposition} \label{lem:Co2Qu0} 
Let $U = N \rtimes U_1$ be a unipotent linear algebraic group, % where with normal subgroup $N$ such that the projection $U \to U/N$ splits
 and let $X$ be an affine $U$-variety (that is, an affine variety with an action of $U$). Suppose $X$
has the structure of a principal $N$-bundle, and the
quotient $X/N$ is a principal $U/N$-bundle. %, for the canonical action of $U/N$ on $X$. 
Then $X$ is a principal $U$-bundle.  
\end{proposition}

\begin{proof}
Let $\pi_N:X \to X/N$ be the quotient map for the $N$-action on $X$ and
$\pi_{U/N}:X/N \to (X/N)/(U/N)$ the quotient map for the $U/N$-action on
$X/N$. 
The composition $\pi_{U/N} \circ \pi_N:X \to (X/N)/(U/N)=X/U$  is a geometric quotient for the $U$-action on $X$. Because $U$ and $N$ are
unipotent the quotients $\pi_{U/N}$ and $\pi_U$ are locally trivial in the
Zariski topology \cite[Proposition 14]{ser58}, so by choosing sufficiently fine open covers it suffices to treat the
case where $X$ and $X/N$ are trivial bundles for $U$ and $U/N \cong U_1$, respectively. %, where we identify $U_1$ with $U/N$ in the natural way. 
So
suppose that $X = N \times (X/N)$ and $(X/N)=U_1 \times (X/U)$. In this case we obtain a $U$-equivariant isomorphism
$$ X \cong U \times (X/U)$$ using the identification of $U$ with the semi-direct product $N \rtimes U_1$.\end{proof} 

%The next result is taken from \cite{ad07}.

%%% prop:TrRe3.1 %%%
\begin{proposition} \label{prop:TrRe3.1} \cite[Theorem 3.12]{ad07} Suppose $X$ is an affine  variety  acted
  upon by a unipotent group $U$ and a locally trivial quotient $X \to X/U$
  exists. Then $X/U$ is affine if and only if $X \to X/U$ is a trivial
  $U$-bundle.   
\end{proposition}

\begin{proposition} \label{cor:GiFi5} \cite[Corollary 3.1.21]{BDHK} Suppose $H$ is a linear algebraic group, $X$ an
  irreducible $H$-variety  and $L \to X$ a linearisation. If the
  enveloping quotient $X \env H$ is projective, then 
%$X \env   H=\proj(S^H)$. Furthermore,
 for suitably
  divisible integers $c>0$  %the sheaf $\OO_{X \env H}(r)$ is an ample line bundle    on $X \env H$
the algebra of invariants $\bigoplus_{k \geq 0} H^0(X, L^{\otimes ck})^H$ is
finitely generated and the enveloping quotient $X \env H$ is the associated
projective variety ; moreover the line bundle $L^{\otimes c}$ induces an ample
line bundle $L^{\otimes c}_{[H]}$ on $X\env H$ such that  the natural structure
map  
\[
\bigoplus_{k \geq 0} H^0(X, L^{\otimes ck})^H \to \bigoplus_{k \geq 0} H^0(X\env H, L^{\otimes ck}_{[H]})
\]
is an isomorphism. 
\end{proposition}

\subsection{Linear algebraic  groups with graded unipotent radicals %and the $\hat{U}$-theorem
}\label{sec:4.2}

In general when a linear algebraic group $H$ acts linearly on a projective variety $X$ then $X\env H$ is not necessarily a projective variety, the enveloping quotient map $\phi:X^{ss,fg,H} \to X\env H$ is not necessarily surjective (indeed its image may not be a subvariety of $X\env H$ but only a constructible subset: see \cite{dorankirwan} $\S$6) and we have no obvious analogues of the Hilbert--Mumford criteria for (semi)stability. There is, however, a class of non-reductive groups to which some of  the key features and computational flexibility of reductive GIT can be extended. % \cite{BDHK0,BDHK2}. 
These are linear algebraic groups $H$ with unipotent radical $U$ graded %either internally,
by a central one-parameter subgroup of a Levi subgroup of $H$, 
in the following sense, studied in \cite{BDHK0,BDHK2}. 

\begin{definition}\label{def:gradedunipotent} 
We say that a linear algebraic group $H = U \rtimes R$  has {\em internally graded unipotent radical} $U$ if there is a central one-parameter subgroup $\l:\GG_m = \kk^* \to Z(R)$ of a Levi subgroup $R$ of $H$ %unipotent radical $U$ of $H$ has an extension $\hat{U} = U \rtimes \GG_m$ in $H$ by the multiplicative group $\GG_m$ of $\kk$
 such that the induced adjoint action of the multiplicative group $\GG_m$ on the Lie algebra of $U$ has all its weights strictly positive. 
 Then $$\hat{U} := U \rtimes \l(\GG_m)$$ is a normal subgroup of $H$ and $H/\hU \cong R/\l(\GG_m)$ is reductive.
 \end{definition}
 
\begin{remark} One can also work with actions of a linear algebraic group $H$ with {\em externally graded unipotent radical $U$}, meaning that %if $H$ is a linear algebraic group over $\kk$ with unipotent radical $U$ and
 there is a semi-direct product $\hat{H} = H \rtimes \l(\GG_m)$ such that the induced adjoint action of $\GG_m$ on the Lie algebra of $U$ has all weights strictly positive, and $\l(\GG_m)$ commutes with a Levi subgroup $R \cong H/U$ of $H$ (cf. \cite{BJK} Definition 2.1 and Remark 2.14). Then $\hat{H} \cong U \rtimes (R \times \l(\GG_m))$ has internally graded unipotent radical, and we can try to construct quotients for $H$-actions on projective varieties $X$ which extend to linear actions of $\hat{H}$, for example by considering quotients by $\hat{H}$ of the form $(X \times \PP^1)\env \hat{H}$ (cf. \cite{BDHK0}). In this paper we will only consider the internally graded case.
 
\end{remark}

Let $H = U \rtimes R$ be a linear algebraic group with internally graded unipotent radical $U$ and grading one-parameter subgroup $\l:\GG_m \to Z(R)$ acting linearly on an irreducible projective variety $X$ with respect to an ample line bundle $L$.
If $\chi: H \to \GG_m$ is a character of $H$, then the kernel of $\chi$ contains $U$ (since unipotent groups have no nontrivial characters), and the restriction of $\chi$ to $\hU$ can be identified with an integer in such a way  that the integer 1 corresponds to the character of $\hU$ which fits into the exact sequence $U \hookrightarrow \hat{U} \to \l(\GG_m)$. 

Let $\weight_{\min}$ be the minimal weight for the action of
the one-parameter subgroup $\l(\GG_m) \leq \hU \leq H$ 
on
$$V:=H^0(X,L)^*$$ and let $V_{\min}$ be the weight space of weight $\weight_{\min}$ in
$V$.  Then $\weight_{\min}=\weight_0$ where $\weight_0 < \weight_{1} < 
\cdots < \weight_{\max} $ are the weights with which 
$\l(\GG_m)$
acts on the fibres of $L^*$ over the fixed point set 
 for its action on $X$. Note that if $\xi \in \lieu$ is a weight vector with weight $\chi$ for the adjoint action of $\l(\GG_m)$ on the Lie algebra $\lieu$ of $U$, then the infinitesimal action of $\xi$ on $V$ takes a weight vector with weight $\weight$ for $\l(\GG_m)$ to one of weight $\weight + \chi$.
Thus we can assume that $\weight_{\min} < \weight_{\max} $
 since otherwise the action of $\l(\GG_m)$ on $X$ is trivial, and by the grading assumption so is the action of the unipotent radical $U$ of $H$, which means that the action of $H$ is via an action of the reductive group $R/\l(\GG_m) \cong H/\hU$. % and we can use classical GIT.

\begin{definition}\label{def:welladapted}
We will call a (rational) linearisation for an action of $\hU$ on $X$ %\emph{well adapted} if $\weight_{\min} <0 < \weight_1$, and
 \emph{borderline adapted} if $\weight_{\min} = 0$. 
\end{definition}

\begin{remark} Let $\chi:H \to \GG_m$ be a character of $H$ whose restriction $\chi|_{\hU}$ to $\hU$ is identified with an integer as above, and let $c$ be a positive integer. Given a linearisation of the $H$-action on $X$ with respect to an ample line bundle $L\to X$, the induced linearisation of the action of $H$ on $X$ with respect to the ample line bundle $L^{\otimes c}$ can be twisted by the character $\chi$ so that the weights $\weight_j$ are replaced with $\weight_jc-\chi|_{\hU}$;
let $L_\chi^{\otimes c}$ denote this twisted linearisation.  
Recall from Remark \ref{rational} that the semistable and stable loci and GIT quotient by a reductive group such as $\GG_m$ are unchanged if the line bundle $L$ is replaced by a positive tensor power of itself, so when considering the action of $\l(\GG_m)$ on $X$ we can formally work with the rational linearisation $L_{\chi/c}$ obtained from twisting the original linearisation by the rational character $\chi/c$. If $\chi|_{\hU} \neq 0$ then we can twist any linearisation of the $H$-action by a rational multiple of $\chi$ to obtain a rational linearisation which is borderline adapted.
\end{remark} 

\begin{definition}\label{def:p} Suppose that $H$ acts linearly on a projective variety $X$ with respect to an ample line bundle $L$. Let
\[
Z_{\min}(X):=X \cap \PP(V_{\min})=\left\{
\begin{array}{c|c}
\multirow{2}{*}{$x \in X$} & \text{$x$ is a $\GG_m$-fixed point and} \\ 
 & \text{$\GG_m$ acts on $L^*|_x$ with weight $\weight_{\min}$} 
\end{array}
\right\}
\]
where $V=H^0(X,L)^*$, and
\[
X_{\min}^0:=\{x\in X \mid p(x)  \in Z_{\min}(X)\}  \quad \mbox{ where } \quad  p(x) =  \lim_{\substack{ t \to 0\\ t \in \GG_m }} \l(t) \cdot x \quad \mbox{ for } x \in X.
\]   
Let $X^{s,\GG_m}_{\min+}$ denote the stable subset of $X$ for the linear action of $\l(\GG_m)$ with respect to the linearisation $L_\chi^{\otimes c}$ for any $\chi$ and $c$ such that the rational character $\chi/c$ satisfies
$$ %\weight_{\min} <
\weight_{\min} = \weight_0 < {\chi |_{\hU}}/{c} < \weight_1.$$
\end{definition}

\noindent By the theory of variation of (classical) GIT \cite{Dolg,Thaddeus} described in Remark \ref{VGIT},  $X^{s,\GG_m}_{\min+}$ is independent of the choice of such a rational character $\chi/c$.
 Indeed it follows from the description in Remark \ref{VGIT} of the stable locus for a linear action of $\GG_m$ that
$$ X^{s,\GG_m}_{\min+} = X_{\min}^0 \setminus Z_{\min}(X). $$
Since the infinitesimal action of a weight vector $\xi \in \lieu$ with weight $\chi$ for the adjoint action of $\l(\GG_m)$ on the Lie algebra $\lieu$ of $U$ takes a weight vector in $V$ with weight $\weight$ to one of weight $\weight + \chi$, the action of $U$ on $X$ preserves $X_{\min}^0$, although it  does not in general preserve $Z_{\min}(X)$. Thus if $u \in U$ we have 
\begin{equation} \label{xsminplus} uX^{s,\GG_m}_{\min+} = X_{\min}^0 \setminus uZ_{\min}(X). \end{equation}

\begin{remark}
Note that the intersection $X \cap \PP(V_{\min})$ is always non-empty, since $V=H^0(X,L)^*$ so $X$ is not contained in any hyperplane in $\PP(V)$.
\end{remark}

\begin{lemma}\label{lemmaA} %Suppose that $H=U \rtimes R$ acts linearly on a projective variety $X$ with respect to an ample line bundle $L$. Then 
 $p: X_{\min}^0 \to Z_{\min}(X)$ is $U$-invariant and $R$-equivariant; that is, 
\[p(rux)=rp(x) \text{ for } x \in X_{\min}^0, r \in R, u\in U.\]  
\end{lemma}
\proof
If $r \in R, x\in X_{\min}^0$ then 
\[p(rx) = \lim_{\substack{ t \to 0\\ t \in \GG_m }} \l(t) \cdot rx=\lim_{\substack{ t \to 0\\ t \in \GG_m }} r \cdot (\l(t)x)=rp(x)\]
since $\l(t) \in Z(R)$.
If $u = \exp(\xi)$ for some $\xi \in \Lie U$ which is a weight vector for the adjoint action of $\l(\GG_m)$, then %$p(ux)=x$ because the adjoint action of  $\l(\GG_m)$ on $\mathrm{Lie}U$ has only
this weight is strictly positive by the grading assumption, and we can choose coordinates on $\PP(H^0(X,L)^*)$ with respect to which the action of $\GG_m$ is diagonal and the infinitesimal action of $\xi$ is in Jordan form.  It now follows from the definition of $X^0_{\min}$ and $p$ that $p(ux)=x$ for this $u$, and hence for all $u \in U$, as required.
\qed

%\medskip

\begin{definition} \label{cond star}(%Conditions %$(*) - (***)$  generalising \lq semistability equals stability' 
\textit{cf.}\ \cite{BDHK0}) 
With the notation of Definition \ref{def:p}, we define the following condition for the $\hU$-action on $X$:
\begin{equation}\label{star}
 \stab_{U}(z) = \{ e \}  \textrm{ for every } z \in Z_{\min}(X). \tag{$*$}
\end{equation}
Note that \eqref{star} holds if and only if we have $\stab_{U}(x) = \{e\}$ for all $x \in X^0_{\min}$. This is because $\stab_{U}(x) = \{e\}$ is an open condition on $x\in X$ and is invariant under the action of $\l(\CC^*)$ as $\l(\CC^*)$ normalises $U$.

This  is also referred to in \cite{BDHK0} as the condition that \emph{semistability coincides with stability} for the action of $\hU$ (or, when $\lambda:\GG_m \to Z(R)$ is fixed, for the linear action of $U$);  a slightly different version of this condition appears in \cite{BDHK2} (see \cite{BDHK2} Remark 2.2 for a comparison). %; see Definition \ref{def:s=ss} below. 
\end{definition}

\begin{definition}\label{def min stable}   
When (\ref{star}) holds for a linear action of $\hU$ the \emph{min-stable locus} for the $\hU$-action is 
\[ X^{s,{\hU}}_{\min+}= \bigcap_{u \in U} u X^{s,\lambda(\GG_m)}_{\min+} .  \]
It follows from (\ref{xsminplus}) that $X^{s,{\hU}}_{\min+}= X^0_{\min} \setminus U Z_{\min}(X).$
\end{definition}

\begin{definition}\label{def:welladaptedaction} A {\it well-adapted linear action} %of the linear algebraic group $H$
on an irreducible projective variety $X$ is given by the following data: %data $X,L,H,\hat{U},\chi$ where
\begin{enumerate}
\item a linear algebraic group $H = U \rtimes R$ with unipotent radical $U$ and Levi subgroup $R$, and a central one-parameter subgroup $\l:\GG_m \to Z(R)$ of $R$ which grades $U$ in the sense of Definition \ref{def:gradedunipotent}, with a \lq complementary' connected subgroup $Z^\perp$ of $Z(R)$ such that $\mathrm{Lie}Z(R) = \mathrm{Lie}Z^\perp \oplus \mathrm{Lie}\lambda(\GG_m)$;
\item a linear action of $H$ on $X$ with respect to an ample line bundle $L$: % (and by abuse of notation $L$ will be used to denote both the line bundle and the linearisation given by the linear action of $H$ on $L$), while
\item a character $\chi: H \to \GG_m$ of $H$ whose restriction to $Z^\perp$ is zero and whose restriction to $\hU = U \rtimes \lambda(\GG_m)$ is nonzero, and  a strictly positive integer $c$ such that the rational character $\chi/c$
satisfies
$$(\chi |_{\hU})/c = \weight_{\min} + \epsilon \mbox{  where }0< \epsilon <\!< 1.$$
\end{enumerate}
Of course one way to choose $Z^\perp$ is so that its Lie algebra is the orthogonal complement to $\mathrm{Lie}\lambda(\GG_m)$ in $\mathrm{Lie}Z(R)$ with respect to a suitable inner product. $Z^{\perp}$ determines and is determined by the ray in the direction of $\chi$ (that is, consisting of positive scalar multiples of $\chi$); %after extending or truncating this ray so that it starts at the point labelled $\weight_{\min}$,
 the rational character $\chi/c$ is required to lie on its truncation (or extension) from the point labelled by $\weight_{\min}$, just beyond this minimum point.
\end{definition}

Here we certainly require that $\epsilon < \weight_1 - \weight_{\min}$ so that 
$(\chi |_{\hU})/c \in (\weight_{\min},\weight_1)$, but we may well require $\epsilon$ to be smaller than this (cf. the proof of Theorem \ref{mainthmextended} below). The main requirement is that $\epsilon>0$ should be sufficiently small to ensure via the theory of variation of GIT (described in Remark \ref{VGIT}) that twisting the linearisation of the $H$-action on $X$ by the rational character $\chi/c$ results in (semi)stable loci for reductive subgroups of $H$ which are independent of the choice of $ %\chi/c$ subject to the requirement that $(\chi |_{\hU})/c = \weight_{\min} + 
\epsilon$. By the Hilbert--Mumford criteria (Theorem \ref{HilbMum}) this can be expressed in terms of the geometry of convex hulls of sets of weights for the action on $X$ of a maximal torus of $R$.

\begin{remark}
For simplicity of exposition we assume that $X$ is irreducible, to ensure that $Z_{\min}(X)$ is connected and $X^0_{\min}$ is dense in $X$. If $X$ had two irreducible components, $X_1$ and $X_2$, then $\weight_{\min}$ could take different values for $X_1$ and for $X_2$.
\end{remark}

\subsection{The $\hat{U}$-theorem}\label{sec:4.3}

Our next aim is to prove the following theorem.

\begin{theorem} \label{mainthm} 
Let $(X,L,H,\hat{U},\chi,c)$ be ingredients for a well-adapted linear action (as in Definition \ref{def:welladaptedaction}) such that \eqref{star} holds (as in Definition \ref{cond star}), and suppose that $UZ_{\min}(X) \neq X^0_{\min}$.  
Then 
\begin{enumerate}[(i)]
\item the algebras of invariants 
\[\oplus_{m=0}^\infty H^0(X,L_{m\chi}^{\otimes cm})^{\hat{U}} \text{ and } \oplus_{m=0}^\infty H^0(X,L_{m\chi}^{\otimes cm})^{H} = (\oplus_{m=0}^\infty H^0(X,L_{m\chi}^{\otimes cm})^{\hat{U}})^{R}\]
are finitely generated;   

\item  the enveloping quotient $X\env \hat{U}$ is the projective variety associated to the algebra of invariants $\oplus_{m=0}^\infty H^0(X,L_{m\chi}^{\otimes cm})^{\hat{U}}$ and is a geometric quotient of the open subset $X^{s,\hU}_{\min+}$ of $X$ by $\hat{U}$;

\item the enveloping quotient $X\env H$ is the projective variety associated to the algebra of invariants $\oplus_{m=0}^\infty H^0(X,L_{m\chi}^{\otimes cm})^{\hat{H}}$ and is the classical GIT quotient of $X \env \hat{U}$ by the induced action of $R/\l(\GG_m)$ with respect to the linearisation induced by a sufficiently divisible tensor power of $L$. 
\end{enumerate}
\end{theorem}

\begin{remark} The proof we present here follows a shorter and simplified argument than the proof presented in the preprint \cite{BDHK0}. The same result is proved in \cite{BDHK2} under a different condition from \eqref{star}, which is also called `stability coincides with semi-stability for the $\hU$ action' (see \cite{BDHK2} Remark 2.2 for a comparison). In \cite{yikun} Y. Qiao studies the relationship between these and alternative geometric notions for stability to coincide with semi-stability. 
\end{remark}

Note that $X^{nss,\hU}$ is $U$-invariant and contained in $X^{nss,\GG_m} = X^{\rms,\GG_m}_{\min+}$, so that 
\begin{equation}\label{Xnss}
 X^{nss,\hU} \subseteq \bigcap_{u \in U} u X^{\rms,\GG_m}_{\min+} = X^{\rms,\hU}_{\min+} = X^0_{\min} \setminus UZ_{\min}(X)
 \end{equation}

A key ingredient in the proof of Theorem \ref{mainthm} is to show that 
\begin{equation}\label{XminisUstable}
X^0_{\min}  \subseteq X^{\rms,U},
\end{equation}
which will be done in Proposition \ref{prop:ExCo4} below.
The proof of Proposition \ref{prop:ExCo4} will rely on using the following lemma in an argument using induction on $\dim(U)$. 

%%% lem:ExCo2 %%%
\begin{lemma} \label{lem:ExCo2} \cite[Lemma 4.7.5]{dix96} Suppose $X$ is an affine
  variety  with an action of $\GG_a$ and let $\xi \in \Lie(\GG_a)$. If there is
  $f \in \OO(X)$ such that $\xi(f)=1 \in \OO(X)$, then $X$ is a trivial
  $\GG_a$-bundle.  
\end{lemma} 

The next lemma constitutes the main technical ingredient for the base case when $\dim(U)=1$; that is, for the action of $\hat{\GG_a}=\GG_a \rtimes \GG_m$ on an affine variety.
%%% lem:ExCo3 %%%
\begin{lemma} \label{lem:ExCo3}
Let $X$ be an affine variety  with an action of $\hat{\GG_a}=\GG_a \rtimes
\GG_m$, where $\GG_m$ acts on $\Lie (\GG_a)$ with strictly positive weight, and let $\xi$ be
a generator of $\Lie(\GG_a)$. Suppose
$\GG_m$ acts on $\OO(X)$ with  weights less than or equal to 0. Then every point in
$X$ has a limit in $X$ under the action of $t \in \GG_m$ as $t \to 0$; let $Z$ be
the set of such limit points in $X$. If $\stab_{\GG_a}(z)=\{e\}$ for each $z
\in Z$, then there is $f \in \OO(X)$ such that $\xi(f)=1 \in \OO(X)$.      
\end{lemma}

\proof
We first show that every point in $X$ has a limit under the action of $t \in
\GG_m$, as $t \to 0$. Fix $x \in X$. To say that $\lim_{t \to 0}
t \cdot x$ exists in $X$ means that the morphism $\phi_x:\GG_m \to X$,
$\phi_x(t)=t \cdot x$, extends to a morphism $\phi_x:\kk \to X$ under the usual
open inclusion $\GG_m \subseteq \kk$ (and then $\lim_{t \to 0} t \cdot x
:=\phi_x(0)$). This is equivalent to saying that the pullback homomorphism
$(\phi_x)^{\#}:\OO(X) \to \OO(\GG_m)=\kk[t,t^{-1}]$ factors through the
localisation map $\kk[t] \to \kk[t,t^{-1}]$. But if $a \in \OO(X)$ is a
$\GG_m$-weight vector of weight $m\leq
0$ then
\[
(\phi_x)^{\#}(a)(t)=a(t \cdot x) = (t^{-1} \cdot a)(x) = t^{-m}a(x)
\]
with $-m \geq 0$. Since $\OO(X)$ is generated by such weight vectors, we see
that $(\phi_x)^{\#}:\OO(X) \to \kk[t,t^{-1}]$ indeed factors through $\kk[t]
\to \kk[t,t^{-1}]$. Let $Z$ be the set of limit points in $X$. 
   
If $\xi \in
\Lie U$ is a
$\l(\GG_m)$-weight vector, then it has positive weight $\ell >0$, say. If $W$ is any
representation of $\hat{U}$, then $\xi$ defines a derivation $\xi:W \to W$ and
any weight vector in $W$ of weight $\omega 
\in \mb{Z}$ gets sent to a weight vector of weight $\omega + \ell$ under $\xi$. In
particular, if $W_{\max}$ denotes the $\GG_m$-weight space in $W$ of maximal
possible weight, then we have
\[
W_{\max} \subseteq \bigcap_{\xi \in \Lie U} \ker(\xi:W \to W) = W^U.
\]
    
Now, using the $\GG_m$-action, write $W=\OO(X)=\bigoplus_{m \leq 0} W_m$ as a
negatively graded algebra, where 
$W_m \subseteq \OO(X)$ is the subspace of weight vectors in $\OO(X)$ of
weight $m \leq 0$. Suppose $\GG_m$ acts on $\xi$ with weight
$\ell >0$. Then we have
\[
\xi(W_m) 
\begin{cases} 
=0 & \text{if $m > -\ell$}, \\
\subseteq W_0 & \text{if $m=-\ell$}, \\
\subseteq \bigoplus_{m<0} W_m & \text{if $m< -\ell$}.
\end{cases}
\]
Let $\tilde{W}=\xi(W_{-\ell})$ be the image of the weight space $W_{-\ell}$
under $\xi$, and consider the vector subspace 
\[
I:=\tilde{W} \oplus \bigoplus_{m<0} W_m.
\]
We claim that $I$ is a $\hat{\GG_a}$-stable ideal of $\OO(X)$. Indeed, $I$ is
$\GG_m$-stable and closed under the action of $\xi$ by construction, so we see
immediately that it is stable under the $\hat{\GG_a}$-action. Let $f
\in I$ and $a \in \OO(X)$. We need to 
show that $af \in I$, for which we may assume that $a \in W_p$ for some $p \leq
0$, without loss of generality. Now if $p<0$, then because multiplication
respects the grading we have $af \in \bigoplus_{m<0} W_m \subseteq
I$. So suppose $p=0$. Write $f=\tilde{f} + g$ where $\tilde{f} \in \tilde{W}$
and $g \in \bigoplus_{m<0} W_m$, so $af=a\tilde{f} + ag$. Then $ag
\in \bigoplus_{m<0} W_m \subseteq I$. Furthermore there is $h \in
W_{-\ell}$ such that $\xi(h)=\tilde{f}$, and because $\xi(a)=0$ we therefore
have $\xi(ah)=a\tilde{f}$, with $ah \in W_{-\ell}$, thus $a\tilde{f} \in
\tilde{W} \subseteq I$. Hence $af \in I$, and the claim is established.    

To finish the proof, we will show that $I=\OO(X)$. We may find a non-trivial
$\GG_m$-invariant complementary subspace $W'$ of $W_0$ such that 
$\OO(X)=W' \oplus I$ as vector spaces. It is easy to see that 
\[
Z=\{x \in X \mid f(x)=0 \text{ for all } f \in \textstyle{\bigoplus_{m<0}} W_m\}
\]
and so the subvariety  $V(I):=\{ x\in X \mid f(x) = 0 \text{ for all } f\in I\}$
defined by $I$ is contained in $Z$. Suppose now, for a contradiction, that
$I$ is a proper ideal of $\OO(X)$ and $\mf{m}$ is a
maximal ideal of $\OO(X)$ that contains $I$, and so  $\mf{m}$ defines a point
in $V(I)$. Given $a \in \mf{m}$, write
$a=a'+f$ with $a' \in W'$ and $f \in I$. Since $I \subseteq \mf{m}$
we have $a' \in \mf{m}$, and since $a' \in W' \subseteq \OO(X)^{\hat{\GG_a}}$ and 
$I$ is stable under the $\hat{\GG_a}$-action, we have $\hat{\GG_a} \cdot a'
\subseteq \mf{m}$. So $\mf{m}$ is stable under the $\hat{\GG_a}$-action. But then
$\mf{m}$ defines a point of $V(I) \subseteq Z$ that is fixed by $\hat{\GG_a}$,
which is a contradiction. Hence $I=\OO(X)$. In particular, the constant function
$1 \in W_0=\tilde{W}$, so 
there is $f \in \OO(X)$ such that $\xi(f)=1$.     
\qed

\medskip

Now, given a linearisation $\hat{U} \act L \to X$, let
$H^0(X,L)_{\max}$ be the $\GG_m$-weight space in $H^0(X,L)$ of maximal possible
weight (which is equal to $-\weight_{\min}$, in our previously defined notation). The
open subset $X^0_{\min}$ is covered by the affine open 
subvarieties 
$X_{\sigma}$, with $\sigma \in H^0(X,L)_{\max}$. Each $X_{\sigma}$ is invariant
under the $\hat{U}$-action, because $H^0(X,L)_{\max} \subseteq H^0(X,L)^U$ and
$\sigma$ is a $\GG_m$-weight vector, and if we choose each $\sigma$ to be a weight vector for the action of $T$, as we may, then $X_{\sigma}$ is invariant
under the $\hat{U}_T$-action. Given nonzero $\sigma 
\in H^0(X,L)_{\max}$, we can embed $X$ into a projective space $\PP^n \cong
\PP(H^0(X,L)^*)$ via $L$ using a 
basis of $n+1$ linear sections which are weight vectors for the
$\GG_m$-action, and which includes $\sigma$. Then $X_{\sigma}$ is contained in
an affine coordinate patch $\A^n=(\PP^n)_{\sigma}$ such that the action of
$\GG_m$ on $\A^n$ is diagonal, with all weights $\geq 0$; or equivalently,
$\GG_m$ acts on $\OO(X_{\sigma})$ with all weights $\leq 0$. Note also that each
point $x \in X_{\sigma}$ has a limit point in
  $X_{\sigma} \cap Z_{\min}(X)$ under the 
  action of $t \in \GG_m$ as $t \to 0$. 

We are now in a position to prove \eqref{XminisUstable}. % which we formulate in the following
%%% prop:ExCo4 %%%
\begin{proposition} \label{prop:ExCo4} Assume that condition \eqref{star} in Definition \ref{cond star} is satisfied. 
For each nonzero $\sigma \in
H^0(X,L)_{\max}$, the natural map $X_{\sigma} \to 
\spec(\OO(X_{\sigma})^U)$ is a trivial $U$-quotient for the $U$-action on
$X_{\sigma}$. Thus, $X^0_{\min} \subseteq X^{\rms,U}$ and the
restriction of the enveloping quotient map for $U \act L \to X$ restricts to
define a locally trivial $U$-quotient of $X^0_{\min}$.  
\end{proposition}

\proof
%As discussed above, we may choose a subnormal series   
%\[
%1=U_0 \trianglelefteq U_1 \trianglelefteq \dots \trianglelefteq U_m=U
%\] 
By diagonalising the action of
$\GG_m$-action on $\Lie U$ and using the  
exponential map we may choose a subnormal series   
\[
1=U_0 \trianglelefteq U_1 \trianglelefteq \dots \trianglelefteq U_m=U
\] 
which is preserved by each automorphism in the family $\lambda:\GG_m \to
\Aut(U)$ and such that each successive quotient $U_{j+1}/U_j \cong \GG_a$, with
$\GG_m$ acting on $\Lie(U_{j+1}/U_j)$ with positive weight.
We will prove that
$X_{\sigma} \to \spec(\OO(X_{\sigma})^{U_j})$ is a trivial $U_j$-quotient by 
induction on $j$, for $1 \leq j \leq m$.  

For the base case, let $\xi_1 \in \Lie(U_1)$ be non-zero. The affine subset $X_{\sigma}$ satisfies the conditions
needed to apply Lemma \ref{lem:ExCo3} with respect to the semi-direct product
$\hat{U_1}=U_1 \rtimes \GG_m$, so there is $f \in \OO(X)$ such that
$\xi_1(f)=1$. It follows from Lemma \ref{lem:ExCo2} that the natural map
$X_{\sigma} \to \spec(\OO(X_{\sigma})^{U_1})$ is a trivial $U_1$-quotient.  

For the induction step, suppose the canonical map $q_j:X_{\sigma} \to
\spec(\OO(X_{\sigma})^{U_j})$ is 
a trivial $U_j$-quotient, for $1 \leq j \leq m-1$. The action of $U_{j+1}
\rtimes \GG_m$ on $X_{\sigma}$ descends to an action of the induced semi-direct product $(U_{j+1}/U_j)
\rtimes \GG_m$ on $\spec(\OO(X_{\sigma})^{U_j})$. Fixing a $\GG_m$-weight vector
$\xi_{j+1} \in \Lie(U_{j+1}) \setminus \Lie(U_j)$, we obtain a generator of
$\Lie(U_{j+1}/U_j)=\Lie(U_{j+1})/\Lie(U_j)$ which acts on
$\OO(X_{\sigma})^{U_j}$ by restricting the action of $\xi_{j+1}$ on
$\OO(X_{\sigma})$ to the subring $\OO(X_{\sigma})^{U_j}$. It is immediate that all
weights for the natural $\GG_m$-action on $\OO(X_{\sigma})^{U_j}$ are
non-positive so, by Lemma \ref{lem:ExCo3}, given a point $y \in
\spec(\OO(X_{\sigma})^{U_j})$ the limit of $y$ under the natural action of
$t \in \GG_m$ as $t \to 0$ exists. If $y=q_j(x)$ for $x \in X_{\sigma}$, then
because $q_j$ is $\GG_m$-equivariant we have
\[
\lim_{t \to 0} t \cdot y = \lim_{t \to 0} q_j(t \cdot x) = q_j\left( \lim_{t \to 0} t
\cdot x \right),
\]
and thus all points in $\spec(\OO(X_{\sigma})^{U_j})$ have limits in
$q_j(X_{\sigma} \cap Z_{\min}(X))$. Let $z \in X_{\sigma} \cap Z_{\min}(X)$ and
suppose $u \in U_{j+1}$ is such that $(uU_j) \in
\stab_{U_{j+1}/U_j}(q_j(z))$. Then there 
is $\tilde{u} \in U_j$ such that $u^{-1}\tilde{u}z=z$. Since $\stab_U(z)$ is
trivial, we conclude that $u=\tilde{u} \in U_j$, so $uU_j=eU_j$. Hence we may
apply Lemma \ref{lem:ExCo3} to the 
action of $(U_{j+1}/U_j) \rtimes \GG_m$ on $\spec(\OO(X_{\sigma})^{U_j})$ to
conclude that there is $f \in \OO(X_{\sigma})^{U_j}$ such that
$\xi_{j+1}(f)=1$. By Lemma \ref{lem:ExCo2}, the natural map
$\spec(\OO(X_{\sigma})^{U_j}) \to \spec(\OO(X_{\sigma})^{U_{j+1}})$ is a trivial 
$U_{j+1}/U_j$-bundle. Since the projection $U_{j+1} \to U_j$ splits, the compositon
\[
X_{\sigma} \to \spec(\OO(X_{\sigma})^{U_j}) \to \spec(\OO(X_{\sigma})^{U_{j+1}})
\]
is a principal $U_{j+1}$-bundle by Proposition \ref{lem:Co2Qu0}, which is in fact
trivial by Proposition \ref{prop:TrRe3.1}. This establishes the induction step. 

Therefore $X_{\sigma} \to \spec(\OO(X_{\sigma})^U)$ is a trivial
$U$-quotient. The rest of the statement of the proposition follows immediately
from the definition of the stable locus for $U \act L \to X$. 
\qed

Note that this argument gives us the following more general result, which does not require condition \eqref{star} in Definition \ref{cond star}.

\begin{proposition} \label{prop:fff}
Let $X$ be an irreducible projective scheme acted upon by $\hat{U}$
with a very ample linearisation
$L \to X$.  
Then
 we have 
$$\{ x \in X^0_{\min} | \stab_U(p(x)) = \{ e \} \} \subseteq X^{\rms,U},$$
  with the restriction of the enveloping quotient map
  defining a locally trivial $U$-quotient of open subvarieties  $$\{ x \in X^0_{\min} | \stab_U(p(x)) = \{ e \} \}  \to
 \{ x \in X^0_{\min} | \stab_U(p(x)) = \{ e \} \} /U.$$   
\end{proposition}

Now we can prove Theorem \ref{mainthm}.

\medskip

\noindent {\it Proof of Theorem \ref{mainthm}}  
It follows from \eqref{XminisUstable}  that $X^0_{\min}$ and $X^0_{\min} \setminus UZ_{\min}(X)$ both have quasi-projective geometric quotients $X^0_{\min}/U$ and $(X^0_{\min} \setminus UZ_{\min}(X))/U$ with induced linear actions of $\l(\GG_m)$.

But by \cite{BB} $p: X^0_{\min} \to Z_{\min}(X)$ induces a $\l(\GG_m)$-invariant morphism $p_U: X^0_{\min}/U \to Z_{\min}(X)$ with affine fibres $F_z$ on which $\l(\GG_m)$ acts linearly with strictly positive weights; its restriction to $(X^0_{\min} \setminus UZ_{\min}(X))/U$ has fibres of the form $F_z\setminus \{z\}$.

Since $X$ is irreducible and $UZ_{\min}(X) \neq X^0_{\min}$, it follows that $X^0_{\min} \setminus UZ_{\min}(X)$ is a dense open subset of $X^{\rms,\hU}$ with a geometric quotient $(X^0_{\min} \setminus UZ_{\min}(X))/\hU$ and that $p$ induces a morphism 
$$p_{\hU}: (X^0_{\min} \setminus UZ_{\min}(X))/\hU \to Z_{\min}(X)$$ whose fibres are weighted projective spaces. Hence $(X^0_{\min} \setminus UZ_{\min}(X))/\hU$ is projective and is a dense open subset of the enveloping quotient $X\env \hU$. So $(X^0_{\min} \setminus UZ_{\min}(X))/\hU = X\env \hU$ and $X \env \hU$ is projective. Thus 
$$ X^{nss,\hU} = X^{ss,fg,\hU} = X^{\rms,\hU} = X^0_{\min} \setminus UZ_{\min}(X) = X^{s,\hU}_{\min+},$$
which is irreducible,
and the algebra of $\hU$-invariants is finitely generated by Proposition \ref{cor:GiFi5}. Hence the algebra of $H$-invariants is also finitely generated with
$$X\env H = (X\env \hU)\env (R/\lambda(\GG_m))$$ (see Remark \ref{envquot}).
\qed

\begin{definition}\label{def:s=ss} Let $(X,L,H,\hat{U},\chi,c)$ be ingredients for a well-adapted linear action (as in Definition \ref{def:welladaptedaction}) such that \eqref{star} holds (as in Definition \ref{cond star}). 
We denote by $X^{s,{{H}}}_{\min+}$ and $X^{ss,{{H}}}_{\min+}$ the pre-images in $X^{s,{{\hU}}}_{\min+} %=X^{ss,{{\hU}}}_{\min+}
$ of the stable and semistable loci for the induced linear action of the reductive group $H/\hU = R/\l(\GG_m)$ on 
$X\env \hat{U} = X^{s,\hU}_{\min+}/\hat{U}$.
We denote by $Z_{\min}(X)^{s(s),R/\l(\GG_m)}$  the stable and semistable loci for the action of the reductive group $R/\l(\GG_m)$ on $Z_{\min}(X)$, with the (rational) linearisation induced by twisting the original linearisation of the $R$-action by a rational multiple of $\chi$ chosen so that after twisting $\l(\GG_m)$ acts trivially on the restriction of $L$ to $Z_{\min}(X)$ (equivalently the twisted linearisation is borderline adapted in the sense of Definition \ref{def:welladapted}).

There are two versions of \emph{$H$-stability=$H$-semistability} for this well-adapted linear action such that (\ref{star}) holds. For the weak version we require %in addition
 that 
$X^{s,{{H}}}_{\min+} = X^{ss,{{H}}}_{\min+}$, while for the strong version we require
$$Z_{\min}(X)^{s,R/\l(\GG_m)} = Z_{\min}(X)^{ss,R/\l(\GG_m)} \neq \emptyset.$$
These versions both coincide with condition \eqref{star} when $H=\hat{U}$. 
The weak version can depend on the choice of the character $\chi$, whereas the strong version is independent of this choice since $\l(\GG_m)$ acts trivially on $Z_{\min}(X)$. We will only consider the strong version in this paper, but the weak version will be studied in \cite{toappear}.
\end{definition}

A modification of the proof of Theorem \ref{mainthm} gives us the following result.
\begin{theorem}\label{mainthmextended} 
Let $(X,L,H,\hat{U},\chi,c)$ be ingredients for a well-adapted linear action (as in Definition \ref{def:welladaptedaction}) such that semistability coincides with stability for the action of $H$  in the strong sense of Definition \ref{def:s=ss}. 
Suppose also that $UZ_{\min}(X)$ is not dense in $X$. 
Then  the projective variety $X\env H$ is a geometric quotient of $X^{s,{{H}}}_{\min+}$ % =X^{ss,{{H}}}_{\min+}$
% of $X$
 by the action of $H$, the stabiliser  $\Stab_H(x)$ is finite for all $x \in X^{s,{{H}}}_{\min+}$
and 
$$ X^{nss,H} = X^{ss,fg,H} = X^{\rms,H} = X^{s,H}_{\min+} =X^{ss,H}_{\min+} = p^{-1}(Z_{\min}(X)^{s,R/\l(\GG_m)}) \setminus UZ_{\min}(X),$$
where $p:X^0_{\min} \to Z_{\min}(X)$ is  as in Definition \ref{def:p} and $Z_{\min}(X)^{s,R/\l(\GG_m)}$ is  as 
 in Definition \ref{def:s=ss}. 
\end{theorem}
\proof
First note that $p^{-1}(Z_{\min}(X)^{s,R/\l(\GG_m)})$ and $UZ_{\min}(X)$ are invariant under the action of $H = U \rtimes R$, since $p$ is $R$-equivariant and $U$-invariant, while $Z_{\min}(X)$ and $Z_{\min}(X)^{s,R/\l(\GG_m)}$ are $R$-invariant. 

Recall that $\mathrm{Lie}Z(R) = \mathrm{Lie} \l(\GG_m) \oplus \mathrm{Lie} Z^{\perp} $ where $Z^{\perp}$ is the complementary connected subgroup to $\l(\GG_m)$ in the centre $Z(R)$ of $R$ such that 
$$ \chi|_{Z^{\perp}} = 0 \,\, \mbox{ and } \,\, \chi|_{\hU} /c  = \weight_{\min} + \epsilon \, \mbox{ for } \, 0 < \epsilon << 1$$
as in Definition \ref{def:welladaptedaction}. Thus $\mathrm{Lie}R = \mathrm{Lie} \l(\GG_m) \oplus \mathrm{Lie} Z^{\perp} \oplus \mathrm{Lie}[R,R]$ where $[R,R]$ is the commutator subgroup of $R$. 
Let $T$ be a maximal torus of $R$; then $\l(\GG_m) \leqslant T$ since by assumption $\l(\GG_m)$ is central in $R$. We can assume that
 $$ \mathrm{Lie} T = \mathrm{Lie} \l(\GG_m) \oplus \mathrm{Lie} T^{\perp} $$
 where $T^{\perp} = Z^{\perp} T_{[R,R]}$ for some maximal torus $T_{[R,R]}$ of $[R,R]$. These direct sum decompositions of $\mathrm{Lie}R$ and $\mathrm{Lie}T$ induce decompositions $(\mathrm{Lie}R)^* = ( \mathrm{Lie} \l(\GG_m))^* \oplus (\mathrm{Lie} Z^{\perp})^* \oplus (\mathrm{Lie}[R,R])^*$ and $ (\mathrm{Lie} T)^* = (\mathrm{Lie} \l(\GG_m))^* \oplus (\mathrm{Lie} T^{\perp})^* $ of their duals, and of the spaces of rational characters of $R$ and $T$, such that 
$$ \chi|_{T} /c  = (\weight_{\min} + \epsilon , 0_{T^{\perp}})$$
where $0_{T^{\perp}}$ denotes the trivial character on $T^{\perp}$.

 By the Hilbert--Mumford criteria (Theorem \ref{HilbMum}),  $x \in X$  is (semi)stable for the action of $R$ on $X$  if and only if (after twisting the linearisation by $\chi/c$) every $rx$  in the $R$-orbit of $x$ is (semi)stable for the action of $T$, and this happens if and only if $ (\weight_{\min} + \epsilon , 0_{T^{\perp}})$
  is contained in (the interior of) the convex hull $C_R(rx)$ of the weights with which $T$ acts on the fibres of $L^*$ over the $T$-fixed points in the closure $ \overline{Trx}$ in $X$ of the $T$-orbit $Trx$ of $rx$ for every $r \in R$. 
  
The composition of the inclusion $T^{\perp} \hookrightarrow T$ with the quotient map $T \to T/\l(\GG_m)$ is an isomorphism, so  we can regard characters of $T/\l(\GG_m)$ as characters of $T^{\perp}$.  By the Hilbert--Mumford criteria,  $z \in Z_{\min}(X)$ is (semi)stable for the action  of $R/\l(\GG_m)$ on $Z_{\min}(X)$ if and only if every $rz$  in the $R$-orbit of $z$ is (semi)stable for the action of $T/\l(\GG_m)$ (with respect to the linearisation for this action induced from the original linearisation for the $R$-action after twisting by a rational multiple of $\chi$ to make it borderline adapted). 
  Moreover this happens if and only if $(\weight_{\min},0_{T^{\perp}})$ is contained in the (interior of the) convex hull  $C_{R/\l(\GG_m)}(rz)$ of the weights with which $T%/\l(\GG_m)
  $ acts on the fibres of $L^*$ over the $T$-fixed points in the closure $ \overline{Trz}$ in $Z_{\min}(X)$ of the $T$-orbit $Trz$  of $rz$ for every $r \in R$, where these weights are all of the form $(\weight_{\min}, \eta)$ for some rational character $\eta$ of $T^{\perp}$. %, while $\chi$ is by assumption a character of $T$ which does not restrict to 0 on $\l(\GG_m)$.
  
 By the definition of $X^0_{\min}$, if $x \not\in X^0_{\min}$ and $r \in R$ then all the $T$-fixed points in $\overline{Trx}$ have the form $(\weight_j, \eta)$ where $\weight_j \ge \weight_1 > \weight_0 = \weight_{\min}$, so if $\epsilon < \weight_1 - \weight_0$ then $ (\weight_{\min} + \epsilon , 0_{T^{\perp}})$ cannot lie in their convex hull, and so $x$ is not semistable for the action of $R$. 
   By the definitions of $p:X^0_{\min} \to Z_{\min}(X)$ and the linearisation  for the action of $R/\l(\GG_m)$ on $Z_{\min}(X)$, if $x \in X^0_{\min}$ %p^{-1}(Z_{\min}(X)^{s,R/\l(\GG_m)})$ % \subseteq X^0_{\min}$ 
 then the set of $T$-fixed points in $\overline{Trx}$ is the union of $\overline{TP(rx)}$ (which is the closure of the $T/\l(\GG_m)$-orbit of $p(rx) = rp(x)$ in $Z_{\min}(X)$) and other fixed points with weights of  the form $(\weight_j, \eta)$ where $\weight_j \ge \weight_1 > \weight_0 = \weight_{\min}$. Moreover 
 $\overline{Trx}$ contains a $T$-fixed point with weight of  the form $(\weight_j, \eta)$ where $\weight_j \ge \weight_1 > \weight_0 = \weight_{\min}$
 if and only if $rx \not\in Z_{\min}(X)$ or equivalently $x \not\in Z_{\min}(X)$.
 By the strong $ss=s$ assumption, the convex hull $C_{R/\l(\GG_m)}(rp(x))$ either contains $(\weight_{\min},0_{T^{\perp}})$ in its interior or does not contain $(\weight_{\min},0_{T^{\perp}})$ at all. So since there are only finitely many weights to consider, if $\epsilon$ is sufficiently small then we have
 \begin{itemize}
 \item $ (\weight_{\min} + \epsilon , 0_{T^{\perp}})$
  is contained in the convex hull $C_R(rx)$ if and only if $ (\weight_{\min} + \epsilon , 0_{T^{\perp}})$
  is contained in its interior;  and
  \item $ (\weight_{\min} + \epsilon , 0_{T^{\perp}})$
  is contained in the (interior of the) convex hull $C_R(rx)$ if and only if $ (\weight_{\min} , 0_{T^{\perp}})$
  is contained in the (interior of the) convex hull $C_{R/\l(\GG_m)}(rp(x))$ and $x \not\in Z_{\min}(X)$.
 \end{itemize}
 Thus  
 $$ X^{s,R} = X^{ss,R} = p^{-1}(Z_{\min}(X)^{s,R/\l(\GG_m)}) \setminus Z_{\min}(X).$$
 Now we can adapt the proof of 
 Theorem \ref{mainthm} in a straightforward way. We have
 
 ($\a$) $X^{nss,H}$ is $U$-invariant and contained in $X^{nss,R} = X^{ss,R} = X^{s,R}$, so that 
$$ X^{nss,H} \,\, \subseteq \,\, \bigcap_{u \in U} u X^{s,R} = p^{-1}(Z_{\min}(X)^{s,R/\l(\GG_m)}) \setminus UZ_{\min}(X);$$

($\b$) $X^0_{\min}\setminus UZ_{\min}(X) \subseteq X^{\rms,U}$ so $X^0_{\min}\setminus UZ_{\min}(X)$ has a quasi-projective geometric quotient $(X^0_{\min}\setminus UZ_{\min}(X))/U$ with an induced linear action of $R$;

($\gamma$) $p: X^0_{\min} \to Z_{\min}(X)$ induces an $R/\l(\GG_m)$-equivariant morphism $$p_U: (p^{-1}(Z_{\min}(X)^{s,R/\l(\GG_m)})\setminus UZ_{\min}(X))/\hU \to Z_{\min}(X)^{s,R/\l(\GG_m)}$$ with fibres which are weighted projective spaces;

($\delta$) $R/\l(\GG_m)$ acts on $Z_{\min}(X)^{s,R/\l(\GG_m)}$ with finite stabilisers and a geometric quotient $$Z_{\min}(X)^{s,R/\l(\GG_m)}/R =
Z_{\min}(X)\env (R/\l(\GG_m))$$ which is a projective variety.

From these it follows that $p^{-1}(Z_{\min}(X)^{s,R/\l(\GG_m)}) \setminus UZ_{\min}(X)$ is a dense open subset of $X^{s,H}$ with a geometric quotient $(p^{-1}(Z_{\min}(X)^{s,R/\l(\GG_m)}) \setminus UZ_{\min}(X))/H$ and that $p$ induces a morphism 
$$p_{H}: (p^{-1}(Z_{\min}(X)^{s,R/\l(\GG_m)}) \setminus UZ_{\min}(X))/H \to Z_{\min}(X)^{s,R/\l(\GG_m)}/R =
Z_{\min}(X)\env (R/\l(\GG_m))$$ whose fibres are quotients of weighted projective spaces by finite groups. Hence $(p^{-1}(Z_{\min}(X)^{s,R/\l(\GG_m)}) \setminus UZ_{\min}(X))/H$ is projective and is a dense open subset of the enveloping quotient $X\env H$, so $(p^{-1}(Z_{\min}(X)^{s,R/\l(\GG_m)}) \setminus UZ_{\min}(X))/H = X\env H$. Thus 
$$ X^{nss,H} = X^{ss,fg,H} = X^{\rms,H} = X^{s,H}_{\min+} =X^{ss,H}_{\min+} = p^{-1}(Z_{\min}(X)^{s,R/\l(\GG_m)}) \setminus UZ_{\min}(X)$$
as required.
\qed

\subsection{Why is well-adaptedness so crucial?}\label{subsectionexample}  In this section we illustrate the importance of the well-adapted shift of linearisation in forming non-reductive quotients. We present a simple example when $\dim U = 1$ showing that if well-adaptedness fails %or semistability does not coincide with stability
 for a linear action of $\hU$ on a projective space $\PP(V)$, then the natural map $\PP(V)^{ss,\hU} \to \PP(V)\env \hU$ is in general not surjective, in contrast with the reductive situation.

Let $V$ be a finite-dimensional
representation of $\hat{U}=\hU^{[2\ell]} = U \rtimes  \CC^*$
where $U = \CC^+$ and $\CC^*$ acts on $\Lie(U)$ with weight $2\ell$ where $\ell \geq 1$, and let $X=\PP(V)$ with
$L=\OO(1)$ defining the canonical linearisation. When $\ell = 1$ the corresponding linear action of $\hU$ on $\PP(V)$ with respect to the line bundle $\mathcal{O}_{\PP(V)}(1)$ is studied in detail in \cite{BDHK} $\S$5, and the analysis extends easily to the case when $\ell >1$.
An adaptation of the proof of Jordan normal form shows that $V$ admits a
decomposition  
%%% eq:ExAd1.3 %%%
\begin{equation} \label{eq:ExAd1.3}
V\overset{\hat{U}%^{[2\ell]}
}{\cong} \bigoplus_{i=1}^q \CC^{(a_i)} \ten \sym^{l_i}
\CC^2   
\end{equation}
of $\hat{U}%^{[2\ell]}
$-modules, where 
%\begin{enumerate} \item 
$\CC^{(a_i)}$ is the one dimensional representation of $\hat{U}%^{[2\ell]}
$
  defined by the character $\hat{U}%^{[2\ell]}
   \to \CC^*$ of weight
$a_i \in \mb{Z}$, and 
 $\sym^{l_i} \CC^2$ is the standard irreducible
representation of $G=\SL(2,\CC)$ of highest weight $l_i \geq 0$, on which $\hat{U}%^{[2\ell]}
$
acts via the surjective homomorphism
\[
\rho_{\ell}:\hU = \hat{U}^{[2\ell]} \to \hat{U}^{[2]}, \quad (u;t) \mapsto (u;t^{\ell})
\]
and the identification of $\hat{U}^{[2]}$ with the Borel subgroup $B \subseteq
G$ of upper triangular matrices given by
\[
\hat{U}^{[2]}= \CC^+  \rtimes \CC^*  \to B, \quad (u;t) \mapsto
\left( \begin{smallmatrix} t & u \\ 0 & t^{-1} \end{smallmatrix} 
\right).   
\]
$\hU$ embeds in $G \times \CC^*$ via $\rho_{\ell}$ and the character of weight one, and 
the linear action of $\hU$ on $V$ extends to a linear action of $G \times \CC^*$ where $G=\SL(2,\CC)$
%by demanding that $G$ 
acts on
$\sym^{l_i}\CC^2$ in the usual manner and 
trivially on $\CC^{(a_i)}$ for each $i$. There is therefore
an isomorphism of  $G \times \CC^*$-spaces %(where $\CC^* = \hat{U}^{[2\ell]}/U$) 
\[
(G\times \CC^*) \times_{\hU} \PP(V) \cong (G \times \CC^*)/\hU) \times \PP(V) \cong (\CC^2 \setminus \{0\}) \times \PP(V) 
\]
giving a $G\times\CC^*$-equivariant embedding of $(G\times \CC^*) \times_{\hU} \PP(V)$ in $\PP^2 \times \PP(V)$.
%lifts to an isomorphism of linearisations. 
Here the
action of $G \times \CC^*$ on $\PP^2=\PP(\CC^3)$
is defined by the representation given in block form by
\[
(g,t) \mapsto \left(\begin{array}{c|c} 1 & 0 \\ \hline 0 &
    g\left(\begin{smallmatrix}t^{-1} & 0 \\ 0 &
        t^{-1} \end{smallmatrix}\right) \end{array}\right)  
\in \GL(3;\CC), \quad g \in G, \ t \in \CC^*,   
\]
where $\GL(3;\CC)$ acts on $\kk^3$ by left multiplication. For any integer
$N>0$, this representation 
 determines a $G \times \CC^*$-linearisation $\mc{P}_{N} $ of the action on
$\OO_{\PP(V)} \otimes \OO_{\PP^2}(N) \to \PP(V) \times \PP^2$. 
As in \cite{BDHK} $\S$5, since the complement in the affine variety $\CC^2$ of $G/U  \cong \CC^2 \setminus \{ 0\} \cong (G \times \CC^*)/\hU$ has codimension two, if $N>0$ is sufficiently large then
\[
\PP(V) \env U \cong (\PP^2 \times \PP(V))
\env_{\mc{P}%^{\prime}
_{N}} G \, \mbox{ and } \, \PP(V)
\env \hat{U}%^{[2\ell]} 
 \cong (\PP^2 \times \PP(V)) 
\env_{\mc{P}%^{\prime}
_{N}} (G \times \CC^*), 
\]
and in addition the algebras of $U$-invariants and $\hat{U}$-invariants on $\PP(V)$ are finitely generated $\CC$-algebras.
%are projective varieties.
Moreover %By Theorem 1.13 of \cite{BDHK2}
 the stable loci $\PP(V)^{\rms,U}$ and
$\PP(V)^{\rms,\hU}$ and 
semistable loci $\PP(V)^{nss,U} = \PP(V)^{ss,fg,U}$ and $\PP(V)^{nss,\hU} = \PP(V)^{ss,fg,\hU}$  may be computed as the
intersections with $\PP(V)$ (embedded in $\PP^2 \times \PP(V)$ as $\{[1:1:0]\} \times \PP(V)$) of the stable and semistable loci for the linear actions of the reductive groups $G$  and $G \times \CC^*$ on $ \PP^2 \times \PP(V)$.
% completely (semi)stable loci  for the $G$ and  $G \times\CC^*$-linearisation $\mc{P}^{\prime}_{N}$. 
For the $\hU$-quotient this is done by applying the Hilbert-Mumford criteria as given in Theorem 
\ref{HilbMum}, using the maximal torus $T_1 \times T_2 \subseteq G
\times \CC^*$ where $T_1$ is the subgroup of diagonal matrices
in $G$ and $T_2=\CC^*$. 

The geometry of the corresponding weight diagram is described in \cite{BDHK} $\S$5, where the group of characters of $T_1 \times T_2$ is identified with $\mb{Z} \times \mb{Z}$ in the natural
way. 
 For sufficiently large $N>0$, when the linearisation is twisted by a rational character so that it becomes well-adapted, the weights for the rational
$T_1 \times T_2$-linearisation $\mc{P}%^{\prime}
_N \to \PP^2 \times \PP(V)$ are
arranged in the fashion of 
Figure \ref{fig:ExAd1}. Here the three clusters of weights correspond to the fixed points for the $T_1 \times T_2$-action on $\PP(V)$ paired with each of the three fixed points $[1:0:0],[0:1:0],[0:0:1]$ in $\PP^2$, and they are translations of each other. Twisting by a rational character corresponds to vertical translation of the diagram. Well-adaptedness means that the top cluster of weights  (which is symmetrically distributed about the vertical axis)
is contained in a cone which is symmetric about the vertical axis, pointing upwards with its vertex lying just below the origin on the vertical axis, and with some weights lying on the boundary of this cone. The additional condition that semistability coincides with stability for the linear action of $\hU$ on $\PP(V)$ means that the vertex of this cone is not itself a weight.
%%% tab:ExAd1 %%% 

\begin{figure}[ht]
\begin{tikzpicture}[scale=1] \draw[->] (0,-8) -- (0,4) node[anchor=south]{$\mb{Q} \cong \Hom(T_2,\CC^*)\ten_{\mb{Z}} \mb{Q}$}; \draw[->] (-7,0) -- (7,0) node[anchor=south east] {$\mb{Q} \cong \Hom(T_1,\CC^*)\ten_{\mb{Z}} \mb{Q}$}; \draw[->] (0,0) -- (1.5,3) node[anchor=west]{$\left(\begin{smallmatrix} 1 \\ \ell \end{smallmatrix} \right)$}; \draw[->] (0,0) -- (-1.5,3) node[anchor=east]{$\left(\begin{smallmatrix} -1 \\ \ell \end{smallmatrix} \right)$}; \foreach \x in {-0.45,-0.15,0.15,0.45}{  -0.15+0.3n for n=-1,0,1,2 \fill (\x,0.55) circle (1.5pt); %% radius in parentheses 
};
 \foreach \x in {-0.15,0.15}{ \fill (\x,1.05) circle (1.5pt); }; \foreach \x in {-0.9,-0.6,...,0.9}{  0.3n for n=-3,-2,...,3  \fill (\x,1.55) circle (1.5pt); };  \fill (0,2.25) circle (1.5pt); %% radius in parentheses 
 \foreach \x in {-1.05,-0.75,-0.45,-0.15,0.15,0.45,0.75,1.05}{ -0.15+0.3n for n=-2,-1,...,3  \fill (\x,2.55) circle (1.5pt); %% radius in parentheses 
 };  
 \begin{scope}[shift={(-3,-6)}] \draw[->] (0,0) -- (1.5,3); \draw[->] (0,0) -- (-1.5,3); \node at (0,0) [anchor=north]{$(-N,-\ell N)$}; \foreach \x in {-0.45,-0.15,0.15,0.45}{  -0.15+0.3n for n=-1,0,1,2 \fill (\x,0.55) circle (1.5pt); };  \foreach \x in {-0.15,0.15}{  \fill (\x,1.05) circle (1.5pt); %% radius in parentheses 
 }; 
\foreach \x in {-0.9,-0.6,...,0.9}{  0.3n for n=-3,-2,...,3 \fill (\x,1.55) circle (1.5pt); };  \fill (0,2.25) circle (1.5pt); 
%% radius in parentheses 
\foreach \x in {-1.05,-0.75,-0.45,-0.15,0.15,0.45,0.75,1.05}{ -0.15+0.3n for n=-2,-1,...,3  \fill (\x,2.55) circle (1.5pt); %% radius in parentheses 
}; \end{scope} 

bottom right chamber (translate top chamber by (3,-6))

\begin{scope}[shift={(3,-6)}]
%% chamber
\draw[->] (0,0) -- (1.5,3); \draw[->] (0,0) -- (-1.5,3); \node at (0,0) [anchor=north]{$(N,-\ell N)$}; 
%% weights
\foreach \x in {-0.45,-0.15,0.15,0.45}{  -0.15+0.3n for n=-1,0,1,2 \fill (\x,0.55) circle (1.5pt); 
%% radius in parentheses 
};  
\foreach \x in {-0.15,0.15}{  \fill (\x,1.05) circle (1.5pt); %% radius in parentheses
}; 
\foreach \x in {-0.9,-0.6,...,0.9}{  0.3n for n=-3,-2,...,3  \fill (\x,1.55) circle (1.5pt); %% radius in parentheses 
};  \fill (0,2.25) circle (1.5pt); %% radius in parentheses

\foreach \x in {-1.05,-0.75,-0.45,-0.15,0.15,0.45,0.75,1.05}{ -0.15+0.3n for n=-2,-1,...,3  \fill (\x,2.55) circle (1.5pt); %% radius in parentheses
}; \end{scope} \end{tikzpicture}
\caption{An example of the distribution of (rational) weights for $T_1 \times T_2 \act \mc{P}%^{\prime}
_{N} \to
  \PP^2 \times 
\PP(V)$.}
\label{fig:ExAd1}
\end{figure}

In these circumstances the weight diagram combined with the Hilbert--Mumford criteria tell us that 
\begin{itemize}
\item the convex hull of any set of weights 
 contains the origin
precisely when its interior %$\Delta_p^{\circ}$
 does so,  and thus semistability and
stability coincide for the action of $T_1 \times T_2$, and hence of $G \times \CC^*$, on $\PP^2 \times \PP(V)$; 
\item if a point $p \in \PP^2 \times \PP(V)$ is stable for the linear action of $G \times \CC^*$ then $p \in (\CC^2 \setminus \{0\})
\times \PP(V)$, and 
\item if $p=([1:1:0],[v])$  then $p$ is stable if and only if  $[v] \in \PP(V)^0_{\min} \setminus U\PP(V_{\min}) = \PP(V)^0_{\min} \setminus UZ_{\min}(X)$.
\end{itemize}
This means that $\PP(V)^{\rms,\hU} = \PP(V)^{ss,fg,\hU}=\PP(V)^{nss,\hU}= \PP(V)^0_{\min} \setminus UZ_{\min}(X),$  that the composition of the inclusion of $\PP(V)^{\rms,\hU}$ in $(\PP^2 \times \PP(V))^{ss,G \times \CC^*}$ with the natural map $$(\PP^2 \times \PP(V))^{ss,G \times \CC^*} \to (\PP^2 \times \PP(V))\env (G \times \CC^*) \cong \PP(V)\env \hU$$ is surjective, and that  $ \PP(V)\env \hU$ is a geometric quotient of $\PP(V)^0_{\min} \setminus UZ_{\min}(X)$ by $\hU$, as predicted by Theorems \ref{mainthm}  and \ref{mainthmextended}.
Note however that if  well-adaptedness fails, % or  semistability does not coincide with stability for the linear action of $\hU$ on $\PP(V)$,
then the natural map $\PP(V)^{ss,fg,\hU} \to \PP(V)\env \hU$ is in general not surjective, in contrast with the reductive situation and with the situation of Theorem \ref{mainthm}.

\subsection{Blow-ups to ensure semistability coincides with stability}

In the set-up of the Green--Griffiths--Lang conjecture, initially the condition (\ref{star}) that semistability coincides with stability  in Definition \ref{def:s=ss} is not satisfied; however we can blow $X$ up to obtain a situation in which semistability does coincide with stability and then apply Theorems \ref{thm:mainA}-\ref{thm:mainintegration}. This blow-up construction works much more generally, in the reductive setting and in the non-reductive setting, as we will now explain briefly. Thus the requirement in Theorems \ref{thm:mainA}-\ref{thm:mainintegration} that semistability should coincide with stability is less restrictive than appears at first sight. The requirement for well-adaptedness is crucial, but can always be attained by twisting the linearisation by a suitable rational character without altering the action on $X$.

Suppose that a reductive group $G$ acts linearly on a projective variety $X$ which has some stable points, but also has semistable
points which are not stable. 
In \cite{K2} it is described how one can blow $X$ up along a sequence of 
  $G$-invariant subvarieties to obtain a $G$-invariant morphism 
$\tilde{X} \to X$ where $\tilde{X}$ is a projective variety acted on 
linearly by $G$ such that $\tilde{X}^{ss,G} = \tilde{X}^{s,G}$. The induced birational 
morphism $\tilde{X}/\!/G \to X/\!/G$ of the geometric invariant theoretic quotients 
is then an algorithmically determined partial desingularisation of $X/\!/G$, in the sense that if $X$ is nonsingular then the centres of the blow-ups can be taken to be nonsingular, and $\tilde{X}/\!/G$ 
has only orbifold singularities %(it is locally isomorphic to the quotient of a nonsingular variety by a finite group action)
 whereas the singularities of $X/\!/G$ 
are in general much worse. Even when $X$ is singular, we can regard the birational 
morphism $\tilde{X}/\!/G \to X/\!/G$ as resolving (most of) the contribution to the singularities of $X/\!/G$ coming from the group action.

The set $\tilde{X}^{ss,G}$ can be obtained from $X^{ss,G}$ as follows. There exist 
semistable points of $X$ which are not stable if and only if there exists a 
non-trivial connected reductive subgroup of $G$ fixing a semistable point. Let 
$r>0$ be the maximal dimension of a reductive subgroup of $G$ 
fixing a point of $X^{ss,G}$ and let $\calr(r)$ be a set of representatives of conjugacy 
classes of all connected reductive subgroups $R$ of 
dimension $r$ in $G$ such that 
$$ Z^{ss}_{R} := \{ x \in X^{ss,G} :  \mbox{$R$ fixes $x$}\} $$
is non-empty. Then
$$
\bigcup_{R \in \calr(r)} GZ^{ss}_{R}
$$
is a disjoint union of nonsingular closed subvarieties $G Z_R^{ss}$ of $X^{ss,G}$ such that $G Z_R^{ss}$ is isomorphic to the quotient of $G \times Z_R^{ss}$ by the diagonal action of the normaliser $N_G(R)$ of $R$ in $G$, which is itself reductive. The action of 
$G$ on $X$ lifts to an action on the blow-up $\psi: X_{(1)}\to X$ of 
$X$ along the closure of $\bigcup_{R \in \calr(r)} GZ_R^{ss}$ which can be linearised so that the complement 
of $X_{(1)}^{ss,G}$ in $ \psi^{-1}(X^{ss,G})$ is the proper transform of the 
subset $\phi^{-1}(\phi(GZ_R^{ss}))$ of $X^{ss,G}$ where $\phi:X^{ss,G} \to X/\!/G$ is the quotient 
map (see \cite{K2} 7.17). 
Here we use the linearisation with respect to (a tensor power of) the pullback of the ample line bundle $L$ on $X$ perturbed by a sufficiently small multiple of the exceptional divisor $E_{(1)}$. This will give us an ample line bundle on the blow-up  $\psi:X_{(1)} \to X$ , and if the perturbation is sufficiently small it will have the property that 
$$   \psi^{-1}(X^{s,G}) \subseteq  X_{(1)}^{s,G} \subseteq  X_{(1)}^{ss,G} \subseteq \psi^{-1}(X^{ss,G}) \subseteq X_{(1)}.$$
Moreover no point of $X_{(1)}^{ss,G}$ is fixed by a 
reductive subgroup of $G$ of dimension at least $r$, and a point in $ X_{(1)}^{ss,G} $ 
is fixed by a reductive subgroup $R$ of 
dimension less than $r$ in $G$ if and only if it belongs to the proper transform of the 
subvariety $Z_R^{ss}$ of $X^{ss,G}$. Repeating this construction at most $r$ times, under the assumption that $X^{s,G} \neq \emptyset$, gives us a $G$-invariant morphism 
$\tilde{X} \to X$ where $\tilde{X}$ is a projective variety acted on 
linearly by $G$ such that $\tilde{X}^{ss,G} = \tilde{X}^{s,G}$, with an induced birational 
morphism $\tilde{X}/\!/G \to X/\!/G$, as required.

\begin{remark} \label{rem:problemcases}
If $X^{s,G} = \emptyset$, then this process might terminate with $X^{ss,G}_{(k)} = G Z_R^{ss}$ for some positive dimensional reductive subgroup $R$ of $G$. Then, 
since $G Z_R^{ss}$ is isomorphic to the quotient of $G \times Z_R^{ss}$ by the diagonal action of the normaliser $N_G(R)$ of $R$ in $G$,
quotienting $X^{ss,G}_{(k)}$ by $G$ is equivalent to quotienting $Z^{ss}_R$ by the action of $N_G(R)$, which is induced by an action of the reductive group $N_R(G)/R$. 

Another possibility, if $X^{s,G} = \emptyset$, is that $X^{ss,G}_{(k)} = \emptyset$ for some $k$. In this case $X_{(k)}$ has a nonempty open subset which is isomorphic to the quotient of $G \times Y$ by the diagonal action of a parabolic subgroup $P$ of $G,$ where $Y$ is a $P$-invariant subvariety of $X$. Then finding a geometric quotient by $G$ of a nonempty open subset of $X$ corresponds to finding a geometric quotient of a nonempty open subset of $Y$ by the parabolic subgroup $P$, which is non-reductive but with internally graded unipotent radical.
\end{remark}

The blow-up construction in the reductive case can be modified in suitable non-reductive settings as described in \cite{BDHK0,yikun}.
Suppose that $H=U\rtimes R \geqslant \hat{U}$ has internally graded unipotent radical $U$ and that we have a well-adapted linear action of $H$ on $X$ with respect to $L$ as at Definition \ref{def:welladaptedaction}. Suppose also that there is some $z \in Z_{\min}(X)$ which has trivial $U$-stabiliser and is stable for the action of $R/\lambda(\GG_m)$ on $Z_{\min}(X)$; this condition is not necessary for the blow-up construction but simplifies it considerably (cf. Remark \ref{rem:problemcases}). 
Then there is a sequence of blow-ups of $X$ along $H$-invariant projective subvarieties 
 resulting in a projective variety $\hat{X}$ with a well adapted linear action of $H$ (with respect to a power of an ample line bundle given by tensoring the pullback of $L$ with small multiples of the exceptional divisors for the blow-ups) which satisfies the  condition \eqref{star}. Thus Theorem \ref{mainthm} applies, giving us a projective geometric quotient
$$\hat{X} \env \hat{U} = \hat{X}^{s,\hU}_{\min+}/\hU$$
and its (reductive) GIT quotient $\hat{X} \env H = (\hat{X} \env \hat{U} ) \env \bar{R}  = (\hat{X} \env \hat{U} ) /\!/ \bar{R}$ where $\bar{R}\cong H/\hU \cong R/\l(\GG_m)$.
Moreover there is a sequence of further blow-ups along $H$-invariant projective subvarieties, 
resulting in a projective variety $\tilde{X}$ satisfying the same conditions as $\hat{X}$ and in addition the strong version of the condition $H$-stability=$H$-semistability (see Definition \ref{def:s=ss}). 

Whenever we have such a sequence of blow-ups, Theorems \ref{thm:mainA}-\ref{thm:mainintegration} apply to the quotient $\tilde{X} \env H = \tilde{X}^{s,H}_{\min+}/H$. 
This is a projective completion of the geometric quotient by $H$ of an $H$-invariant open subset of $X$ which 
 is isomorphic (via the blow-down map) to
 the complement in $\tilde{X}^{s,H}_{\min+}$ of the exceptional divisors resulting from the blow-ups in the construction of $\tilde{X}$.

\section{$\Omega$-moment maps for complex group actions on K\"ahler manifolds}\label{sec:momentmaps}

\subsection{$\Omega$-moment maps for complex reductive actions}\label{subsec:setup}

Suppose now that $Y$ is a compact K\"ahler manifold with a holomorphic action of a complex reductive Lie group $G$; then $G$ is the complexification of any maximal compact subgroup $K$ of $G$. 
Recall that the dual $\mathfrak{k}^* = \Hom_{\RR}(\mathfrak{k},\RR)$ of the Lie algebra $\mathfrak{k}$ of $K$ embeds naturally in the complex dual $\mathfrak{g}^* = \Hom_{\CC}(\mathfrak{k},\CC)$ of the 
Lie algebra $\mathfrak{g} = \mathfrak{k} \otimes \CC$ of $G$, as 
\begin{equation} \label{eqn:dual} \mathfrak{k}^* = \{\xi \in \mathfrak{g}^*: \xi(\mathfrak{k}) \subseteq \RR\}. \end{equation}

\begin{remark} For a real or complex vector space $V$ the notation $V^*_{(\RR)}=\Hom_{\RR}(V,\RR)$ can be used to denote the real dual space, and for a complex vector space $V$ the notation $V^*_{(\CC)}=\Hom_{\CC}(V,\CC)$ can be used for the complex dual space. We will often abuse notation by writing simply $V^*$ in each case, in the hope that it will be clear from the context whether we mean the real or complex dual.  
\end{remark}

From the viewpoint of symplectic geometry it is customary to fix a maximal compact subgroup $K$ of $G$ and to assume that (by averaging over $K$ if necessary) we have chosen a K\"ahler form $\omega$ on $Y$ which is $K$-invariant; then we can ask for a moment map $\mu:Y \to \mathfrak{k}^*$ for the $K$-action as in $\S$\ref{sec:redGIT}. 
 Any other maximal compact subgroup of $G$ is given by $K' = g^{-1}Kg$ for some $g \in G$;
%and $K$ is its own normaliser in $G$
then $\omega_{K'} = g^*\omega$ is $K'$-invariant and if $\mu:Y \to \mathfrak{k}^*$ exists then $\mu_{K'} = Ad^*(g^{-1}) \circ \mu_K \circ g$ is a $K'$-moment map with respect to $\omega_{K'}$, where $Ad^*$ denotes the co-adjoint action of $G$ on the complex dual $\mathfrak{g}^*$ of its Lie algebra and $\mathfrak{k}^*$ is embedded in $\mathfrak{g}^*$ as above. 
So to define a \lq moment map' for the $G$-action on $Y$, instead of fixing a K\"ahler form  $\omega$ it is natural to ask for a $G$-orbit $\Omega$ in   %equivariant map 
$$\{ (K,\omega) \in \mathcal{K}_G \times \mbox{K\"ahler}(Y) : \omega \mbox{ is $K$-invariant} \}$$
where $ \mbox{K\"ahler}(Y)$ is the space of 
 K\"ahler forms on the complex manifold $Y$ and 
 $$\mathcal{K}_G = \{K \, |\,  K \mbox{ is a maximal compact subgroup of } G\}.$$
 We will give the $G$-orbit $\Omega$ of $(K,\omega) \in \mathcal{K}_G \times \mbox{K\"ahler}(Y)$ the structure of a smooth manifold obtained by identifying it with the quotient of $G$ by the stabiliser in $G$ of $(K,\omega)$, and thus avoid the necessity to work with the infinite-dimensional space $ \mbox{K\"ahler}(Y)$.
 
 \begin{definition} \label{defn:omega}
We will call a $G$-orbit $\Omega$ in   %equivariant map 
$\{ (K,\omega) \in \mathcal{K}_G \times \mbox{K\"ahler}(Y):\omega \mbox{ is $K$-invariant} \}$
 a \emph{$G$-equivariant K\"ahler structure} on $Y$. 
 Given a $G$-equivariant K\"ahler structure $\Omega$ on $Y$,
 we define an \emph{$\Omega$-moment map} for the $G$-action on $Y$    %,\Omega)$ 
 to be a smooth $G$-equivariant map 
$$\newmm: \Omega \times Y \to \mathfrak{g}^*$$
such that $\newmm(K,\omega, x) = \mu_{(K,\omega)}(x)$ for each $(K,\omega) \in \Omega$ %mathcal{K}_G$
 and $x \in Y$, where
%$$\mathcal{K}_G = \{K | K \mbox{ is a maximal compact subgroup of } G\}.$$
$\mu_{(K,\omega)}:Y \to \mathfrak{g}^*$ is the composition of a moment map for the $K$-action on $(Y,\omega)$ with the canonical embedding \eqref{eqn:dual} in  $\mathfrak{g}^*$  of the dual %$\lieks$
 of the Lie algebra of~$K$. 
\end{definition}

\begin{remark} \label{rem:symmetric} The stabiliser in $G$ of $(K,\omega)$ in the orbit $\Omega$ is the intersection of the normaliser $N_G(K)$ of $K$ in $G$ with the stabiliser $\{ g \in G:g^*\omega = \omega\}$ of $\omega$; this intersection must contain $K$ so we can identify $\Omega$ with a quotient of the symmetric space $G/K$. 
By \cite{HHSWZ}  Theorem A the quotient of the normaliser $N_G(K')$ of any compact subgroup $K'$ of $G$ by right multiplication by the product $K'_0\, Z_G(K')$ of its connected component $K'_0$ and its centraliser $Z_G(K')$ in $G$ is finite. %Let us assume for simplicity that $G$ is connected. Then $K$ is connected, and its 
The centraliser $Z_G(K)$ of the maximal compact subgroup $K$ in $G$ is the centre $Z(G)$ of $G$. Moreover the connected component $G_0$ of the identity in $G$ 
 is the product of its commutator subgroup $[G_0,G_0]$ and its centre $Z(G_0)$, which intersect in a finite central subgroup $F$ of $G_0$.  Here $[G_0,G_0]$ is semisimple while $Z(G_0)$ is a complex torus. 

When $G=Z(G_0)$ is a complex torus then it has a unique maximal compact subgroup $K$ whose normaliser is $G$. On the other hand it follows from Theorem 2 of \cite{atiyah} that in this situation, if the action of $G$ on $Y$ is effective (at least up to a finite subgroup), then the stabiliser $\{ g \in G:g^*\omega = \omega\}$ of any $K$-invariant K\"ahler form $\omega$ for which there exists a moment map is $K$. This tells us that in general when $G_0 = [G_0,G_0]\, Z(G_0)$, if the action of $Z(G_0)$ on $M$ is effective up to a finite subgroup, then the stabiliser $\{ g \in G_0:g^*\omega = \omega\}$ in $G_0$ of any $K$-invariant K\"ahler form $\omega$ for which there exists a moment map is contained in  $[G_0,G_0] \, Z(K_0)$, where $K_0$ is the connected component of the identity in $K$ and thus a maximal compact subgroup of $G_0$. 
Since $([G_0,G_0] \, Z(K_0) \, \cap \, K_0Z(G_0))/K_0$ is finite, combining this with \cite{HHSWZ}  Theorem A as above shows that
 the quotient by $K$ of the stabiliser in $G$ of $(K,\omega)$ is finite. Thus this stabiliser is a compact subgroup of $G$ which contains the maximal compact subgroup $K$, and so equals $K$.
This implies that  the orbit $\Omega$ is the symmetric space $G/K$, provided that the action of the centre $Z(G)$ of $G$ on $Y$ is effective up to a finite subgroup; if this condition is not met then we can replace the $G$-action with an action of a quotient $\tilde{G}$ of $G$ by a central subgroup such that $Z(\tilde{G})$ does act effectively on $Y$.

\end{remark}

\begin{definition} \label{defn:mommapeqn} We will say that a smooth $G$-equivariant map 
$\newmm: \Omega \times Y \to \mathfrak{g}^*$ \emph{satisfies the $\Omega$-moment map differential equation} for a $G$-action on $Y$ with respect to a $G$-equivariant K\"ahler structure $\Omega$ on $Y$ if  the following  condition on its derivative holds:

If $a \in \mathfrak{g}$  and $\xi \in T_xY$  then the derivative at $(K,\omega,x) \in \Omega \times Y$ of the complex-valued function on $\Omega \times Y$ given by evaluating $\newmm: \Omega \times Y \to \mathfrak{g}^*$ at $a$ takes $(0,  %b_{(K,\omega)},
\xi) \in T_{(K,\omega)}\Omega \times T_xY$ to 
\begin{equation} \label{one} (  %\overline{\eta_{\omega,x}(a_x,\xi)} 
\eta_{\omega,x}(\xi,a_x) - \eta_{\omega,x}(\iota_K(a)_x, \xi))/2i,\end{equation}
where $\eta_{\omega,x}$ is the value at $x$ of the Hermitian metric on $Y$ whose imaginary part is the $K$-invariant K\"ahler form $\omega$, while $\iota_K$ is the $K$-invariant anti-complex involution on $\mathfrak{g}$ with fixed point set $\mathfrak{k}$, and $x \mapsto a_x$ is the holomorphic vector field on $Y$ determined by the infinitesimal action of $G$. 

\end{definition}

\begin{lemma} \label{lem:mommapeqn} A smooth $G$-equivariant map 
$\newmm: \Omega \times Y \to \mathfrak{g}^*$ is an $\Omega$-moment map for a $G$-action on $Y$ with respect to a $G$-equivariant K\"ahler structure $\Omega$ on $Y$ if and only if it satisfies the $\Omega$-moment map differential equation.  \end{lemma}
\proof 
We need to show that $\mu_{(K,\omega)}:Y \to \mathfrak{g}^*$ defined for each $(K,\omega) \in \Omega$ %mathcal{K}_G$
 and $x \in Y$
 by
$\mu_{(K,\omega)}(x) = \newmm(K,\omega, x) $ is the composition of a moment map for the action of $K$ on $(Y,\omega)$ with the embedding of $\mathfrak{k}^*$  in  $\mathfrak{g}^*$
as $\{ \zeta \in \mathfrak{g}^*: \zeta(\liek) \subseteq \RR\}$. Assuming the convention that the sesquilinear form $\eta_{\omega,x}$ is complex linear in its second variable,
this follows as

(i) $(  %\overline{\eta_{\omega,x}(a_x,\xi)}
\eta_{\omega,x}(\xi,a_x) - \eta_{\omega,x}(\iota_K(a)_x, \xi))/2i $ is a complex linear function of $a \in \mathfrak{g}$, and

(ii) if $a \in \mathfrak{k}$ then $\iota_K(a) = a$ so $(   %\overline{\eta_{\omega,x}(a_x,\xi)} 
\eta_{\omega,x}(\xi,a_x) - \eta_{\omega,x}(\iota_K(a)_x, \xi))/2i $ is real and equal to $\omega_x(\xi,a_x)$.

\qed

\begin{remark} Given a $G$-equivariant K\"ahler structure $\Omega$ on $Y$ and a choice of maximal compact subgroup $K$ of $G$ with $(K,\omega) \in \Omega$, the existence and choice of an $\Omega$-moment map 
$\newmm: \Omega \times Y \to \mathfrak{g}^*$ for the $G$-action on $Y$ is equivalent to the existence and choice of a moment map $\mu:Y \to \lieks$ in the traditional sense for the $K$-action on the K\"ahler manifold $(Y,\omega)$.
However this point of view with emphasis on the complex group action rather than the compact one is a more natural one to take when extending the notion of a moment map to non-reductive linear algebraic groups, in particular if we wish to construct quotients of open subsets of $Y$ by holomorphic actions of non-reductive linear algebraic groups with internally graded unipotent radicals. 
\end{remark}

\begin{example} \label{example:FS}
When a reductive group $G$ acts linearly on a complex projective variety $Y \subseteq \PP^n$ via a representation $\rho:G \to \GL(n+1)$, then as in \cite{francesthesis} (see \eqref{mommap} above) we can choose a $G$-equivariant K\"ahler structure $\Omega \subseteq \mathcal{K}_G \times \mbox{K\"ahler}(Y)$ in the sense of Definition \ref{defn:omega}, such that \\
(i) for each maximal compact subgroup $K$ of $G$ and K\"ahler form $\omega$ on $Y$ such that 
$(K,\omega) \in \Omega$ there is a $K$-invariant hermitian inner product $\langle  \, ,\rangle_K$ on $\CC^{n+1}$ with respect to which $\rho(K)$ acts unitarily, and $\omega$ is the associated Fubini--Study form; \\
(ii) there are moment maps $\mu_{(K,\omega)}:Y \to \lieks$   %\mathrm{Lie}(K)^*$   
 for each $(K,\omega) \in \Omega$ giving an $\Omega$-moment map $\newmm: \Omega \times Y \to \mathfrak{g}^*$ with
$$\newmm(K,\omega, [y]).a = \mu_{(K,\omega)}([y]).a = \frac{\langle y, \rho_*(a) y \rangle_K}{2 \pi i || y ||^2} =\frac{ \bar{y}^T \rho_*(a)y}{2\pi i ||y||^2} $$
for $a \in \mathfrak{g}$ and $[y] \in Y \subseteq \PP^n$;\\
(iii) the $G$-(semi)stable loci of $Y$ can be described in terms of the $\Omega$-moment map $\newmm:\Omega \times Y \to \mathfrak{g}^*$ as
$$ Y^{s,G} = \{x \in Y : 0 \in \newmm( \Omega \times \{x\}) \mbox{ and } \dim \mathrm{Stab}_G(x) = 0 \}$$
and 
$$ Y^{ss,G} = \{x \in Y : 0 \in \newmm( \Omega \times \overline{Gx})  \};$$
(iv) we can define $S$-equivalence as usual on $Y^{ss,G}$ as $x \sim y$ if and only if the closures of $Gx$ and $Gy$ meet in $Y^{ss,G}$, and then for any $(K,\omega) \in \Omega$ the inclusion of $\mu_{(K,\omega)}^{-1}(0)$ in $Y^{ss}$ %and of $\{x \in \mu_{(K,\omega)}^{-1}(0) : \dim \mathrm{Stab}_G(x) = 0 \}$ in $Y^s$
 induces an identification
$$\mu_{(K,\omega)}^{-1}(0)/K \, \cong \, Y^{ss,G}/\! \sim \,\,\, \cong \,\,Y/\!/G.$$
\end{example}

\subsection{$\Omega$-moment maps for non-reductive actions} 
\label{subsec:Hactions}

%----------------------------

Suppose now that $Y$ is a compact K\"ahler manifold with a holomorphic action of a complex reductive group $G = K_\CC$ with a maximal compact subgroup $K$, and that $G$ has a subgroup $\hat{U}=U\rtimes \l(\CC^*)$ with internally graded unipotent radical $U$, such that  
 $K \cap \hat{U} = \l(S^1)$ is the unique maximal compact subgroup of the one-parameter subgroup $\l(\CC^*)$ of $\hU$ (and is therefore a maximal compact subgroup of $\hU$). Then the Lie algebra of $\hU$ decomposes as a real vector space as
\begin{equation}\label{decomp}
\hat{\lieu}=\RR \oplus i\RR \oplus \lieu
\end{equation}
where $\mathrm{Lie}(K \cap \hU)=i\RR$ and $\lieu$ is the Lie algebra of the (complex) unipotent group $U$. 
Let $\Omega$ be a $G$-equivariant K\"ahler structure on $Y$ represented by $(K,\omega)$ where $\omega$ is a $K$-invariant K\"ahler form on $Y$.

Suppose that $X\subseteq Y$ is a compact complex submanifold of $Y$ which is invariant under the $\hU$ action, though not necessarily invariant under the $G$-action. 
Then  the circle subgroup $\l(S^1)=K\cap \hU$ preserves the K\"ahler structure on $X$ given by the restriction of $\omega$. 
The K\"ahler form $\omega$ makes $Y$ and $X$ into symplectic manifolds acted on by $K$ and $S^1$ respectively and in addition gives $Y$ a $K$-invariant and $X$ an induced $S^1$-invariant Riemannian metric. Assume that a moment map 
exists for the action of $K$ on $Y$, whose composition 
 with the embedding $\liek^* %= \liek^*_{(\RR)}
  \to  %\lieg^*_{(\CC)} = 
  \lieg^*$ defines for  us an $\Omega$-moment map $\newmm: \Omega \times Y \to \lieg^*$ % = \lieg^*_{(\CC)}$ 
  for the action of $G$ on $Y$ in the sense of Definition \ref{defn:omega}. Composing this with the inclusion of $X$ in $Y$ and the restriction map $\lieg^*  \to %\hat{\lieu}^*_{(\CC)} =
  \hat{\lieu}^*= \CC \oplus \lieu^*=\RR \oplus i\RR \oplus \lieu^*$, where  $\lieu^* = \lieu^*_{(\CC)}$ and $ \hat{\lieu}^*= \hat{\lieu}^*_{(\CC)}$, 
 we get a %$\Omega$-moment
  map $\newmmxu:X \to \hat{\lieu}^*$ fitting into a  diagram 
\begin{equation}\label{uhatmomentmap}
\xymatrix{\Omega \times X \ar[r]^-{\newmm|_{\Omega \times X}} \ar[rd]^{\newmmxu} &  \lieg^*=\liek^* \oplus i\liek^* \ar[d] %^{p^*}
\\
& \hat{\lieu}^*=\RR \oplus i\RR \oplus \lieu^*} .
\end{equation}
We can allow the addition of a rational character (that is, a rational multiple of the derivative of the group homomorphism $\hU \to \CC^*$ with kernel $U$) to this \lq $\Omega$-moment map' $\newmmxu:X \to \hat{\lieu}^*$. Indeed, in the K\"ahler setting there is no reason not to allow real multiples, rather than only rational ones, here.

\begin{remark} \label{GrebMiebach}
This construction is related to the work of Greb and Miebach in
 \cite{grebmiebach}. Motivated partly by the non-reductive version of GIT for unipotent actions using reductive envelopes explored in \cite{dorankirwan}, they study holomorphic actions of a 
unipotent subgroup $U \leqslant G=K_\CC$ of a simply-connected semisimple complex Lie group $G$
 on a compact K\"{a}hler manifold $X$ with K\"ahler form $\omega$. They look for\\
 (i) analogues in this setting of the notions of a {linear action} and 
of a {reductive envelope} $\overline{G \times_U X}$;\\
(ii) constraints on $\omega$ to allow it to be extended to a $K$-invariant K\"{a}hler form on a reductive envelope $\overline{G \times_U X}$, and \\
(ii) methods involving moment maps for the $K$-action on $\overline{G \times_U X}$ for constructing and studying quotients for the $U$-action on $X$ in the K\"ahler context. 

They prove that the following conditions are equivalent:\\
(a) $G \times_U X$ is K\"ahler;\\
(b) the $U$-action on $X$ is meromorphic in a suitable sense; \\   
(c) there exists a \lq Hamiltonian $G$-extension': a compact K\"{a}hler manifold $(Z,\omega_Z)$ with a Hamiltonian $K$-action and a $U$-equivariant embedding $X \hookrightarrow Z$ such that the de Rham cohomology class $[\omega_Z]$ restricts on $X$ to $[\omega]$.

If these conditions are satisfied there are embeddings ${X \hookrightarrow G \times_U X \hookrightarrow G \times_U Z \cong G/U \times Z \hookrightarrow V \times Z}$ where 
$G/U$ is embedded as a $G$-orbit in a representation $V$ of $G$ with flat $K$-invariant K\"ahler structure. They define
$${X^{ss,U}[\omega] = X \cap \{ y \in G/U \times Z: \mu^{-1}(0) \cap \overline{Gy} \neq \emptyset\}}$$
where $\mu = \mu_Z + \mu_V$ for moment maps $\mu_Z: Z \to \lieks$ and $\mu_V: V \to \lieks$, and show that   
 there is a  stratified K\"ahler %structure on a compact complex
 space $\overline{Q}$ with a nonsingular open subset $Q$ on which the stratified K\"ahler structure restricts to a smooth K\"ahler form $\omega_Q$, and a 
geometric quotient $\pi: X^{ss,U}[\omega] \to X^{ss,U}[\omega]/U = Q$ 
extending to a meromorphic map from  $X$ to $\overline{Q}$ with $[\pi^*\omega_Q] = [\omega]$.
 
Suppose now that $X$ is invariant under the action of a subgroup $\hU$ of $G$ which is the semidirect product of the unipotent subgroup $U$ and a multiplicative group which grades it. Then $G \times_U X$ has an induced action of $G\times \CC^*$, and we can ask for this action to extend to $\overline{G \times_U X}$, and for $\omega$ to extend to a $K \times S^1$-invariant K\"{a}hler form on $\overline{G \times_U X}$. Then we can take $Y=\overline{G \times_U X}$ in our construction above of an $\Omega$-moment map.

While \cite{grebmiebach} is motivated partly by \cite{dorankirwan} which studies unipotent actions and their enveloping quotients in a GIT context, here we are
motivated by the results of \cite{BDHK2} which show that under suitable conditions and for suitable choices of linearisations there exist non-reductive GIT quotients  $X\env \hU$ which are better behaved (in particular having less dependence on choices) than is generally the case for unipotent actions. One can think of $X\env \hU$ as a reductive GIT quotient  of $X\env U$ by the multiplicative group $\CC^*$, where $X\env U$ corresponds to the space $\overline{Q}$ in \cite{grebmiebach} and  the analogue of the singular boundary in $\overline{Q}$ is unstable for the action of the multiplicative group, so does not appear in the quotient $X\env \hU$.

\end{remark}

Now suppose slightly more generally that $Y$ is a compact K\"ahler manifold with a holomorphic action of a complex reductive Lie group $G=K_\CC$, and that  %Therefore if $\liek$ and $\lieg$ are the Lie algebras of $K$ and $G$ then $\lieg=\liek \oplus i\liek$. 
 $H=U \rtimes R$ is a linear algebraic subgroup of  $G$ with internally graded unipotent radical $U$. As before suppose that $\l: \CC^* \to R$ is a central one-parameter subgroup of the Levi subgroup $R$ of $H$ which grades $U$, with $\hat{U}=U\rtimes \l(\CC^*) \leqslant H$.  
Suppose also that $K \cap H=Q \in \mathcal{K}_R \subseteq  \mathcal{K}_H$ is a maximal compact subgroup of $H$ which is contained in the Levi subgroup $R=Q_\CC$. %\end{itemize}
Then as before the Lie algebra of $\hU$ decomposes as a real vector space as
$\hat{\lieu}=\RR \oplus i\RR \oplus \lieu
$ 
where $\mathrm{Lie}(K \cap \hU)=i\RR$ and $\lieu$ is the Lie algebra of $U$, while the Lie algebra of $H$ decomposes as 
\begin{equation}  \label{decomposes} \Lie(H)=\Lie(Q) \oplus i \Lie(Q) \oplus \lieu=(\Lie(\l(S^1)) \oplus \Lie(\bar{Q})) \oplus i(\Lie(\l(S^1)) \oplus \Lie(\bar{Q})) \oplus \lieu,\end{equation}
where $\bar{Q}$ is a subgroup of $Q$ such that (at least on the level of Lie algebras) $Q$ is the product of $\bar{Q}$ and $\l(S^1)$.

 \begin{definition} \label{defn:omegaH}
Suppose that $X$ is a compact K\"ahler manifold with a holomorphic action of a complex linear algebraic group $H=U\rtimes R$ with internally graded unipotent radical $U$ as above. An $H$-equivariant K\"ahler structure on $X$ is an $H$-orbit $\Omega$ in %the image under restriction of K\"ahler forms from $Y$ to $X$ of
$$\{ (K,\omega) \in \Omega_{G,Y}:  K\cap H \in \mathcal{K}_H \}, $$
where $Y$ is a compact K\"ahler manifold with a holomorphic action of a complex reductive Lie group $G=K_\CC$ and   a $G$-equivariant K\"ahler structure  $\Omega_{G,Y}$  such that \\
(i) $H$ is a subgroup of $G$,\\
(ii) $X$ is a complex submanifold of $Y$, and \\
(iii)  the action of $G$ on $Y$ restricts to the action of $H$ on $X$. %, and \\
 Given an $H$-equivariant K\"ahler structure $\Omega$ on $X$,  we define an \emph{$\Omega$-moment map} for the $H$-action on $X$    %,\Omega)$ 
 to be a smooth $H$-equivariant map 
$$\newmmxh: \Omega \times X \to \mathfrak{h}^*$$
such that 
if $a \in \mathfrak{h}$  and $\xi \in T_x X$  then the derivative at $(K,\omega,x) \in \Omega \times X$ of the complex-valued function on $\Omega \times X$ given by evaluating $\newmmxh: \Omega \times X \to \mathfrak{h}^*$ at $a$ takes $(0,  %b_{(K,\omega)},
\xi) \in T_{(K,\omega)}\Omega \times T_x X$ to 
\begin{equation} \label{one2} (  %\overline{\eta_{\omega,x}(a_x,\xi)} 
\eta_{\omega,x}(\xi,a_x) - \eta_{\omega,x}(\iota_K(a)_x, \xi))/2i,\end{equation}
where $\eta_{\omega,x}$ is the value at $x$ of the Hermitian metric on $Y$ whose imaginary part is the $K$-invariant K\"ahler form $\omega$, while $\iota_K$ is the $K$-invariant anti-complex involution on $\mathfrak{g}$ with fixed point set $\mathfrak{k}$, and $x \mapsto b_x$ denotes the infinitesimal action of $b \in \lieg$ (cf. Lemma  \ref{lem:mommapeqn}).

\end{definition}

If an $\Omega$-moment map exists for the $H$-action on $X$, then it is uniquely determined up to the addition of an $H$-equivariant map $\Omega \times X \to \mathfrak{h}^*$ which takes a constant (and $K\cap H$-invariant) value on $\{(K,\omega)\}  \times X$ for $(K,\omega) \in \Omega$. 

By Lemma  \ref{lem:mommapeqn}, if $\newmm: \Omega_{G,Y} \times Y \to \lieg^*$ is an $\Omega_{G,Y}$-moment map for the $G$-action on $Y$ and $\newmmxg:\Omega \times X \to \lieg^*$ is its restriction to $\Omega \times X$, we have an $\Omega$-moment map $\Omega \times X \to \lieh^*$ for the $H$-action on $X$ fitting into  the following extended diagram 
\begin{equation}\label{uhatmomentmapH}
\xymatrix{\Omega \times X \ar[r]^-{\newmmxg} \ar[rd]^{\newmmxh}  &  \lieg^*=\liek^* \oplus i\liek^* \ar[d]^{p_1^*} & \\
& \mathfrak{h}^* %\Lie(H)^*=\hat{\lieu}^* \oplus \Lie(R/\CC^*)
 \ar[d]^{p_2^*} \ar[rd]^{p_3^*} & \\
& \hat{\lieu}^* & \Lie({R})^*=\Lie({Q})^* \oplus i\Lie({Q})}^*\\ 
\end{equation}
where %$\bar{R}$ is the complexification of $\bar{Q}$ and
 $R$ is the complexification of $Q=K\cap H$ for any maximal compact subgroup $K$ of $G$ which contains a maximal compact subgroup of $H$. We obtain another $\Omega$-moment map by adding to $\newmmxh$ any $H$-equivariant map $\Omega \times X \to \mathfrak{h}^*$ which is constant and $K\cap H$-invariant on $\{(K,\omega)\}  \times X$ for every $(K,\omega)$ in the $H$-orbit $\Omega$. Here  $\newmmxu=p_2^* \circ \newmmxh$, and
 the composition $p_3^* \circ \newmmxh$ restricts to a $Q$-moment map for the symplectic structure $\omega$ on $X$ identified with $ \{(K,\omega)\} \times X \subseteq \Omega \times X$, for any $(K,\omega) \in \Omega$ such that $K \cap H = Q$.

\begin{remark} \label{shift} As in Remark \ref{rem:symmetric} we can assume that the stabiliser in $G$ of $(K,\omega) \in \Omega_{G,Y}$ is the maximal subgroup $K$ of $G$, and so the stabiliser in $H$ of $(K,\omega) \in \Omega$ is the maximal compact subgroup $Q=K \cap H$ of $H$. Thus $\Omega \cong H/Q$. By $H$-equivariance, an $\Omega$-moment map for the action of $H$ on $X$ is determined by its restriction to $ \{(K,\omega)\} \times X$ for any $(K,\omega) \in \Omega$, and this restriction is determined by the condition on its derivative up to a constant in $\lieh^*$ which must be fixed by the co-adjoint action of $K\cap H$.
Moreover
$\alpha \in \lieh^*$ is invariant under the co-adjoint action of $K\cap H$ if and only if it annihilates $[\lieh,\lieh \cap \liek] = \{[\zeta,\xi]: \zeta \in \lieh, \xi \in \lieh \cap \liek\}$, and the existence of a central one-parameter of the complexification $R$ of $K\cap H$ grading the unipotent radical $U$ of $H$ implies that $[\lieh,\lieh \cap \liek]$ contains $\lieu$ and equals $[\lieh,\lieh] = \lieu \oplus [\mathfrak{r},\mathfrak{r}]$ where $\mathfrak{r} = \mathrm{Lie}(R)$. Thus
the annihilator $[\lieh,\lieh \cap \liek]^\circ$ is isomorphic to the dual $\mathfrak{z}(\mathfrak{r})^*$ of the Lie algebra of the centre $Z(R)$ of $R$.

Note that only 0 in $\lieu^*$ is fixed by the co-adjoint action of $\l(S^1)$ and therefore by the action of $K \cap H$.
\end{remark}

\begin{example} \label{implosion}
Suppose that $U$ is a maximal unipotent subgroup of a semisimple complex Lie group $G$, and that $\l:\CC^* \to G$ is a one-parameter subgroup of $G$ which grades $U$. Then $U$ is the unipotent radical of a Borel subgroup $B$ of $G$ which contains $\l(\CC^*)$, and $B=T_{\CC} U$ for a maximal torus $T_\CC$ of $G$ which is the centraliser of $\l(\CC^*)$ in $G$. Let $K$ be a maximal compact subgroup of $G$ which contains the maximal compact subgroup $\l(S^1)$ of $\hU = U \rtimes \l(\CC^*)$. Then there is a maximal compact torus $T\leqslant K$ containing $\l(S^1)$, which centralises $\l(S^1)$ and so must be the unique maximal compact torus in $T_\CC$ (with complexification $T_\CC$ as the notation suggests). This means that $K \cap B = T$ and $KB=G$. Moreover if we choose a $K$-invariant inner product on $\liek$ then the composition of the inclusion \eqref{eqn:dual} of $\liek^*$ in $\lieg^*$ with the restriction map $\lieg^* \to \lieu^*$ can be identified with the projection of $\liek^*\cong \liek$ onto the annihilator $\liet^{\perp}$ of $\liet$ in $\liek^*$, while its composition with the restriction map $\lieg^* \to \hat{\lieu}^*$ can be identified with the projection onto $\liet^{\perp} \oplus \mathrm{Lie}(\l(S^1))^*$. 

In this situation we can relate $\Omega$-moment maps to the viewpoint of symplectic implosion
\cite{implosion}. Here we assume that a holomorphic action of $\hU$ on $X$ extends to a holomorphic action of $G$ on $X$ and that we have a $G$-equivariant K\"ahler structure on $X$ containing $(K,\omega)$ with a moment map $\mu:X \to \liek^*$ for the $K$-action on $(X,\omega)$.  Hence we can take $Y=X$ in Definition \ref{defn:omegaH}, obtaining a $\hU$-equivariant K\"ahler structure $\Omega$ on $X$ and $\Omega$-moment map $\newmmxu: \Omega \times X \to \hat{\lieu}^*$. 

The symplectic implosion (or imploded cross-section) of $X$ associated to the Hamiltonian action of $K$ on $X$ is defined as follows.
The Weyl group $W = N_G(T)/T$ acts on $\liet \cong \liet^*$ which decomposes into
Weyl chambers; let 
$\liet_{+}^*$ be the positive Weyl chamber in $\liet^*$ corresponding to the Borel subgroup $B$. % $\cong  \liets/W \cong \lieks/K$.
For $\zeta \in \liet^*$ let $K_\zeta = \{ k \in K|(Ad^*k)\zeta = \zeta \}$, with commutator
subgroup $[K_\zeta,K_\zeta]$. Then the imploded cross-section  of $X$ is
$${X_{impl} = \mu^{-1}(\liet^*_+)/\sim}$$
where $x \sim y$ if and only if   $x=ky$ for some $k \in [K_\zeta,K_\zeta]$
with 
$\zeta = \mu(x) = \mu(y) \in \liet^*_+.$
Over the interior points of $\liet^*_+$ no collapsing occurs since
$[K_\zeta,K_\zeta]=[T,T]$ is trivial. 
$X_{impl}$ inherits a {stratified symplectic structure and   $T$-action} with  moment map 
$X_{impl} \to \liet^*_+ \subseteq \liet^*$   induced by the restriction of $\mu$.
We have
$${X_{impl} \cong (X \times (T^*K)_{impl})/\!/K}$$
where $K$ acts on itself by left translation and on the symplectic manifold $T^*K \cong \lieks \times K \cong G$ with moment map 
$$\mu(p,q).a = p\cdot a_q \,\,\,\, \forall a \in \liek, \, q \in K, \, p\in T^*_q K = \liek^*.$$ In fact the \lq universal imploded cross-section'
$(T^*K)_{impl}$  is an affine algebraic variety over $\CC$. Indeed the geometric quotient
$G/U$ of the affine variety $G$ by right multiplication by $U$ is a quasi-affine variety whose algebra of regular functions
$\mathcal{O}(G/U) = \mathcal{O}(G)^U$ is finitely generated, so that
$G/U$ has a canonical affine completion
$$ %\overline{G/U}^{aff} \textcolor{blue} {
{G\env U = \mathrm{Spec}(\mathcal{O}(G)^U).}$$
Guillemin, Jeffrey and Sjamaar showed in \cite{implosion} that  $G\env U$ has a $K$-invariant stratified K\"{a}hler structure which is symplectically isomorphic to the universal implosion $(T^*K)_{impl}$.
It follows that if 
 $X$ is an affine or projective variety acted on linearly by $G$ then the corresponding algebra of $U$-invariants is finitely generated with associated non-reductive GIT quotient $X \env U$ and 
$$X_{impl} \cong (X \times (G\env U))/\!/G \cong X \env U. $$
Thus in this situation the non-reductive GIT quotient $X \env U$ has a description in terms of moment maps, but this description is quite complicated. The non-reductive GIT quotient $X \env \hU = (X \env U)/\!/\CC^*$, where we allow the linearisation of the $\hU$-action to be twisted by a rational character $\eta$ or equivalently subtract $\eta$ from the moment map $\mu_{\l(S^1)}$ given by composing $\mu$ with restriction $\liek^* \to \mathrm{Lie}(\l(S^1))^*$, has an induced description as 
$$((\mu^{-1}(\liet^*_+) \cap \mu_{\l(S^1)}^{-1}(\eta))/\sim)/\l(S^1).$$ 
If $\eta$ is such that $\mu^{-1}(\liet^*) \cap \mu_{\l(S^1)}^{-1}(\eta)$ is contained in the pre-image under $\mu$ of the interior of  $\liet^*_+$  then the equivalence relation $\sim$ has no  effect, and we obtain a much simpler description  
$$ X \env \hU \cong \mu_{\hU}^{-1}(\eta)/\l(S^1)$$
of $X\env \hU$ in terms of  the \lq moment map' $\mu_{\hU}: X \to \hat{\lieu}^*$ for the $\hU$-action on $X$ obtained from the $\Omega$-moment map $\newmmxu: \Omega \times X \to \hat{\lieu}^*$ by identifying $X$ with $\{(K,\omega)\} \times X \subseteq \Omega \times X$.

\end{example}

\begin{example} \label{remarkFS}
We are particularly interested in $\Omega$-moment maps  
$\newmmxh: \Omega \times X \to {\mathfrak{h}}^*$   
 in the special case when $X$ is a nonsingular projective subvariety of $Y=\PP^n$ with a linear $H$-action on $X$ via a homomorphism $\rho: H \hookrightarrow  \GL(n+1, \CC)$, 
 as in Example \ref{example:FS}. % under the the linear acton of $\GL(n+1,\CC)$ on the ambient projective space. Then 
Here $K$ is the unitary group $U(n+1)$ and its Lie algebra $\liek=u(n+1)$ consists of matrices $A$ such that $A=-\bar{A}^T$. % and  $p^*$ is the dual of the embedding \[p: \hat{\lieu} \hookrightarrow \liek,\ A \mapsto A-\bar{A}^T.\] 
We can assume that a maximal compact subgroup $Q$  of $H$ containing the circle $\l(S^1) \leqslant \hU$ preserves the standard
$U(n+1)$-invariant hermitian inner product $\langle  \, ,\rangle$ on $\CC^{n+1}$.
We can use the corresponding Fubini-Study form $\omega$ on $\PP^n$ to define a $G$-equivariant K\"ahler structure $\Omega_{G,\PP^n}$ on $\PP^n$ containing $(K,\omega)$ and an $\Omega$-moment map  $\newmmxh: \Omega \times X \to \hat{\lieh}^*$ whose restriction $\mu^H_{(K,\omega)}: X \to \lieh^*$ to $X$ identified with $\{(K,\omega)\} \times X \subseteq \Omega \times X$ is given, up to the addition of a $K \cap H$-invariant constant, by
\[\mu^H_{(K,\omega)}([x]) \cdot a =\frac{1}{2\pi i ||x||^2} \bar{x}^T \rho_*(a)x \in \CC \text{ for all } a \in H\]
where $\cdot$ denotes the natural pairing between the complex dual of the Lie algebra $\hat{\mathfrak{u}}$ and the Lie algebra itself (cf. Example \ref{example:FS}).
The real and imaginary parts can be written as
\[\mathrm{Re} (\mu^H_{(K,\omega)}([x]) \cdot a) =\frac{1}{4\pi i ||x||^2}(\bar{x}^T (\rho_*(a)-\overline{\rho_*(a)}^T )x\ ,\ \   \mathrm{Im} (\mu_{\hU}([x]) \cdot a)=\frac{1}{4\pi i ||x||^2}\bar{x}^T (\rho_*(a)+\overline{\rho_*(a)}^T )x,\]
where $ \overline{\rho_*(a)}^T = - \rho_*(\iota_K(a))$,
and we can check directly that \eqref{one}   %uhatmomentmapH}
 is satisfied.

Let $\xi \in T_x S^{2n+1}$ be a tangent vector to the unit sphere $S^{2n+1} \subseteq \CC^{n+1}$ at $x\in S^{2n+1}$ representing the tangent vector $[\xi]$ at $[x] \in \PP^n$. Then 
\begin{equation}\label{complexmoment} 
2\pi i \, d\mu^H_{(K,\omega)}([x])([\xi])\cdot a=\bar{x}^T \rho_*(a)\xi+\bar{\xi}^T \rho_*(a)x=\langle \overline{\rho_*(a)}^T x,\xi \rangle + \langle \xi, \rho_*(a)x \rangle
\end{equation}
where $\langle \, , \rangle$  denotes the standard Hermitian inner product %on $\PP^n$
 satisfying $\langle \xi,\chi \rangle=\bar{\xi}^T \chi=\overline{\langle \chi,\xi \rangle}$. 
\end{example}

Motivated by this example, we will give a series of definitions, modelled on the algebraic situation, for holomorphic actions on compact K\"ahler manifolds  of  complex linear algebraic groups $H=U\rtimes R$ with internally graded unipotent radical $U$ and grading one-parameter subgroup $\l:\CC^* \to Z(R)$, when there is  an $H$-equivariant K\"ahler structure and  an {$\Omega$-moment map}.

\begin{definition} \label{cond starK} Suppose that $X$ is a compact K\"ahler manifold with a holomorphic action of a complex linear algebraic group $H=U\rtimes R$ with internally graded unipotent radical $U$ and grading one-parameter subgroup $\l:\CC^* \to Z(R)$. 
Let $\Omega \subseteq \Omega_{G,Y}$ be an $H$-equivariant K\"ahler structure  on $X$ with respect to an inclusion $X \subseteq Y$ where the $H$-action on $X$ extends to an action of a reductive group $G$ on $Y$. Let $(K,\omega) \in \Omega$ and let 
$\newmmxh: \Omega \times X \to \mathfrak{h}^*$ be
 an {$\Omega$-moment map} for the $H$-action on $X$ so that 
 $\mu^{\l(S^1)}_{(K,\omega)}(x) = \newmmxh(K,\omega,x).\l_*(1)$ defines a moment map for the action of $\l(S^1)$ on $X$.  Then define
\[
Z_{\min}(X)=\left\{
\begin{array}{c|c}
\multirow{2}{*}{$x \in X$} & \text{$x$ is fixed by $\l(\GG_m)$ and} \\ 
 & \text{$\mu^{\l(S^1)}_{(K,\omega)}(x)$ is the minimal value taken by $\mu_{\l(S^1)}$ on $X$    }
\end{array}
\right\}
\]
and
\[
X_{\min}^0:=\{x\in X \mid     %p(x)  \in Z_{\min}(X)\}  \quad \mbox{ where } \quad  p(x) =  \lim_{\substack{ t \to 0\\ t \in \GG_m }} t \cdot x \quad \mbox{ for } x \in X.
\mathcal{P}(x) \mbox{ has a limit point in } Z_{\min}(X) \}  \quad \mbox{ where } \quad  \mathcal{P}(x) = \{t \cdot x\mid t\in \CC^*\}  \quad \mbox{ for } x \in X.
% \{ t \cdot x 
\]   
We will say that \emph{semistability coincides with stability} for the action of $\hU$  on $X$, 
with respect to the equivariant K\"ahler structure
$\Omega \subseteq \Omega_{G,Y}$  on $X$ and the restriction to $\hU$ of the {$\Omega$-moment map}
$\newmmxh: \Omega \times X \to \mathfrak{h}^*$, if $Z_{\min}(X) \subseteq Z_{\min}(Y)$ and 
\begin{equation}\label{starK}
 \stab_{U}(z) = \{ e \}  \textrm{ for every } z \in Z_{\min}(X). %\tag{$*$}
\end{equation}
\end{definition}

\begin{remark}\label{remss=s} 
\begin{enumerate}[(i)]
\item Note that the grading condition implies that $X^0_{\min}$ is $U$-invariant.
\item In the algebraic situation of Example \ref{example:FS}  the limit $p(x) =  \lim_{\substack{ t \to 0\\ t \in \GG_m }} t \cdot x$ exists for any $x \in X$, and then
$$X_{\min}^0 =\{x\in X \mid     p(x)  \in Z_{\min}(X)\}. %\quad \mbox{ where } \quad  p(x) =  \lim_{\substack{ t \to 0\\ t \in \GG_m }} t \cdot x \quad \mbox{ for } x \in X.
$$
The same is in fact true more generally, by \cite{atiyah} (cf. Remark \ref{morsebott}). This is related to the fact (see \cite{fujiki}) that $\CC^*$-actions on compact K\"ahler manifolds which restrict to Hamiltonian $S^1$-actions are meromorphic in the sense of Remark \ref{GrebMiebach}. The results of \cite{atiyah} also show that the complement of $X_{\min}^0$ in $X$ has real codimension at least two, and so $X$ is connected if and only if $X_{\min}^0$ is connected, in which case $Z_{\min}(X)$ is connected.

As in the algebraic situation of 
Definition \ref{cond star} and Example \ref{remarkFS}, the condition \eqref{starK} holds if and only if we have $\stab_{U}(x) = \{e\}$ for all $x \in X^0_{\min}$, because $\stab_{U}(x) = \{e\}$ is an open condition on $x\in X$ and is invariant under the action of $\l(\CC^*)$ as $\l(\CC^*)$ normalises $U$. Similarly if $H'$ is a subgroup of $H$ containing $U$ then $\l(\CC^*)$ normalises $H'$ and $\stab_{H'}(x) = \{e\}$ if $\stab_{H'}(p(x))= \{e\}$ for any $x \in X^0_{\min}$.

\item In the algebraic situation of Example \ref{remarkFS}, the condition $Z_{\min}(X) \subseteq Z_{\min}(Y)$ will always be satisfied when $Y=\PP^n$ is the projective space associated to the space of sections of a very ample line bundle on $X$ defining a linearisation of the $H$-action.
\end{enumerate}
\end{remark}

\begin{example}
Recall from $\S$\ref{subsectionexample} the case of a linear  action of $\hat{U}=\hU^{[2\ell]}= U \rtimes  \CC^*$ on $X=\PP(V)$ %where  $V$ is a finite-dimensional representation of $\hU$\hat{U} = U \rtimes  \CC^*$
where $U = \CC^+$ and $\CC^*$ acts on $\Lie(U)$ with weight $2\ell$. 
Here we saw that the representation of $\hU$ on $V$ extends to a representation of $\SL(2,\CC) \times \CC^*$, so we can construct a $\hU$-equivariant K\"ahler structure
$\Omega \subseteq \Omega_{G,Y}$  on $X$ and an {$\Omega$-moment map}
$\newmmxu: \Omega \times X \to \hat{\lieu}^*$ by setting $Y=X$ and $G=\SL(2,\CC) \times \CC^*$, and using a Fubini--Study K\"ahler form on $\PP(V)$ which is $\SU(2) \times S^1$-invariant. Here the quotient $X\env \hU$ described from an algebraic viewpoint in $\S$\ref{subsectionexample} can be described in terms of symplectic implosion and thus in terms of 
the {$\Omega$-moment map}
$\newmmxu: \Omega \times X \to \hat{\lieu}^*$   as in Example \ref{implosion}.
In this case the requirement that $Z_{\min}(X) \subseteq Z_{\min}(Y)$ in Definition \ref{cond starK} is of course satisfied.

We could also choose $Y=P(W)$ where $W$ is a finite-dimensional representation of $\hU$ with $V$ as a $\hU$-submodule. Then the  representation of $\hU$ on $W$ extends to a representation of $G= \SL(2,\CC) \times \CC^*$, and we can construct another $\hU$-equivariant K\"ahler structure
$\Omega \subseteq \Omega_{G,Y}$  on $X$ and an {$\Omega$-moment map}
$\newmmxu: \Omega \times X \to \hat{\lieu}^*$ using this choice of $Y$ and $G= \SL(2,\CC) \times \CC^*$.
However in general $V$ fails to be a $G$-submodule of $W$ and the requirement that $Z_{\min}(X) \subseteq Z_{\min}(Y)$ in Definition \ref{cond starK} is not satisfied (for example if $V \cong \sym^k(\CC^2)$ for some $k\geqslant 1$ then we could take $W=\sym^m(\CC^2)$ for $m>k$). 
This means that well-adaptedness for $Y$ is not necessarily the same as well-adaptedness for $X$, and the analysis in $\S$\ref{subsectionexample} (or consideration of symplectic implosion as in Example \ref{implosion})
 shows that $(\{[1:1:0]\} \times \PP(V)) \cap
(\PP^2 \times \PP(W))^{ss,G \times \CC^*}$  does not necessarily coincide with  $\PP(V)^{\rms,\hU}$; in particular this intersection may be empty when $\PP(V)^{\rms,\hU}$ is not.

\end{example}

\begin{definition}\label{def:welladaptedactionK} 

A {\it Hamiltonian action} of a linear algebraic group $H$ on a compact K\"ahler manifold $X$ is given by
a holomorphic action of $H$ on $X$ with an $H$-equivariant K\"ahler structure $\Omega$ and an $\Omega$-moment map $\newmmxh: \Omega \times X \to \lieh^*$. 
If 
$(K,\omega) \in \Omega$ let
$$\weight_{\min} = \weight_0 < \weight_1 < \ldots $$
be the values taken on the fixed point set $X^{\l(S^1)}$ for the $\l(S^1)$-action on $X$ by the composition 
$$\mu^{\l(S^1)}_{(K,\omega)}: X \to \mathrm{Lie}(\l(S^1))^* \cong \RR$$ of $\lieh^* \to \mathrm{Lie}(\l(S^1))^*$  with the 
 restriction $\mu^H_{(K,\omega)}: X \to \lieh^*$ of $\newmmxh:\Omega \times X \to \lieh^*$ to $X$ identified with $\{(K,\omega)\} \times X$. We will say that the Hamiltonian action is \emph{borderline-adapted} if $\weight_{\min} = 0$ and \emph{well-adapted} if $\weight_{\min} < 0 < \weight_{\min} + \epsilon$ for sufficiently small $ \epsilon> 0$; how small $\epsilon$ should be will be discussed in $\S$5.4. 

\end{definition}

\begin{remark} \label{shift2} 
Suppose we have a {Hamiltonian action} of a linear algebraic group $H$ on a compact K\"ahler manifold $X$  given by
a holomorphic action of $H$ on $X$ with an $H$-equivariant K\"ahler structure $\Omega$ and an $\Omega$-moment map $\newmmxh: \Omega \times X \to \lieh^*$ as in Definition \ref{def:welladaptedactionK}.

Note that $\mu^{\l(S^1)}_{(K,\omega)}: X \to \mathrm{Lie}(\l(S^1))^*$
 takes the constant value $\weight_{\min}$ on $Z_{\min}(X)$. 
 Note also that by Remark \ref{shift}, since $\l(S^1)$ is central in $K \cap H$, we can always adjust the $\Omega$-moment map $\newmmxh$ to add any fixed constant to  the values 
$\weight_{\min} = \weight_0 < \weight_1 < \ldots $
taken on the fixed point set $X^{\l(S^1)}$ for the $\l(S^1)$-action on $X$ by the composition 
$\mu^{\l(S^1)}_{(K,\omega)}: X \to \mathrm{Lie}(\l(S^1))^*$ of $\lieh^* \to \mathrm{Lie}(\l(S^1))^*$  with the 
 restriction $\mu^H_{(K,\omega)}: X \to \lieh^*$ of $\newmmxh:\Omega \times X \to \lieh^*$ to $X$ identified with $\{(K,\omega)\} \times X$, and thus ensure that $\weight_{\min} <0 < \weight_1 $, or indeed that $0 \in (\weight_{\min}, \weight_1)$ is arbitrarily close to $\weight_{\min}$.
 \end{remark}
 
 Let 
\begin{equation} \label{mmaprmodl}
\mu_{(K,\omega)}^{R/\l(\CC^*)} : Z_{\min}(X) \to \Lie(R/\l(\CC^*))^*
\end{equation}
be given by composing the restriction of $\mu^H_{(K,\omega)}$ to $Z_{\min}(X)$ with the restriction map $ \lieh^* \to \Lie(R)^*$ and identifying $\Lie(R/\l(\CC^*))^*$ first with the kernel of $ \Lie(R)^* \to \Lie(\l(\CC^*))^*$ and then by translation with the pre-image of $\weight_{\min}$ in $ \Lie(R)^*$. Then $\mu_{(K,\omega)}^{R/\l(\CC^*)}$ takes values in the dual of the Lie algebra of the maximal compact subgroup $Q/\l(S^1)$ of $R/\l(\CC^*)$, and can thus be identified with a moment map $$\mu_{(K,\omega)}^{Q/\l(S^1)} : Z_{\min}(X) \to \Lie(Q/\l(S^1))^*$$ %defined at \eqref{mmaprmodl}
 for the induced action of $Q/\l(S^1)$ on $Z_{\min}(X)$ with the K\"ahler form $\omega$. .

\begin{definition} \label{cond starKH}[$H$-stability=$H$-semistability]
As in Definition \ref{def:s=ss} there is a strong version and a weak version of the condition that $H$-stability coincides with $H$-semistability in this setting.
We will say that \emph{ semistability coincides with stability in the strong sense} for a well-adapted Hamiltonian action of $H$  on $X$ if 
semistability coincides with stability for the action of $\hU$  on $X$ in the sense of Definition \ref{cond starK}, and
in addition 0 is a regular value for the moment map 
\[\mu_{(K,\omega)}^{Q/\l(S^1)} : Z_{\min}(X) \to \Lie(Q/\l(S^1))^*\]
 for the induced action of $Q/\l(S^1)$ on $Z_{\min}(X)$ with the K\"ahler form $\omega$. 
%\end{definition}

%\begin{definition}
 \label{weaks=ss} %[Weak $H$-stability=$H$-semistability]
We will say that \emph{semistability coincides with stability in the weak sense} for a well-adapted Hamiltonian action of $H$  on $X$ if 
semistability coincides with stability for the action of $\hU$ on $X$ in the sense of Definition \ref{cond starK}, and
in addition $0$ is a regular value for the moment map 
\[\mu_{(K,\omega)}^{Q/\l(S^1)} : (\mu_{(K,\omega)}^{\hU})^{-1}(0) \to \Lie(Q/\l(S^1))^*\] 
 for the induced action of $Q/\l(S^1)$ on $(\mu_{(K,\omega)}^{\hU})^{-1}(0)$ with the K\"ahler form $\omega$.
\end{definition}

\begin{remark} \label{remark:ss=s} As in the algebraic situation of Definition \ref{def:s=ss}, semistability coinciding with stability in the strong sense implies that semistability coincides with stability in the weak sense, and in this paper we will always assume the strong version of this condition. In the setting of Example \ref{remarkFS} we know that 0 is a regular value for the moment map 
$\mu_{(K,\omega)}^{Q/\l(S^1)}$ for the induced action of the compact group $Q/\l(S^1)$ on $Z_{\min}(X)$ if and only if the stabiliser in $Q/\lambda(S^1)$ of every point in $(\mu_{(K,\omega)}^{Q/\l(S^1)})^{-1}(0)$ is finite, and this happens if and only if 
semistability coincides with stability for the induced linear action of its complexification, the reductive group $R/\l(\CC^*)$ (see \cite{francesthesis} $\S$5.5, Thm 7.4 and Remark 8.14, \cite{Ness} Thm 2.2). Thus in this setting, semistability coincides with stability for the action of $H$ in the strong sense of Definition \ref{cond starKH} if and only if semistability coincides with stability for the action of $H$ in the strong sense of Definition \ref{def:s=ss}.
 \end{remark}

\begin{remark} \label{cond starK2} 

Note that if $X^0_{\min} = UZ_{\min}(X)$ then the fibres of the $\hU$-invariant map $p :X^0_{\min} \to Z_{\min}(X)$ are single $\hU$-orbits and indeed single $U$-orbits. So in this situation we can regard $Z_{\min}(X)$ as a quotient of $X^0_{\min}$ by $U$ and by $\hU$, and we can regard $Z_{\min}(X)\env R$ as a quotient of $p^{-1}(Z_{\min}(X)^{ss,R})$ by $H$. 
%In this situation stable points will not exist, but otherwise for a well-adapted linearisation stable points will exist.
\end{remark}

\subsection{Zero loci for moment maps} %Proof of Theorem \ref{thm:mainA} for $\hU$ actions}

In this section we suppose that $X$ is a compact K\"ahler manifold with a holomorphic action of a complex linear algebraic group $H=U\rtimes R$ whose Levi subgroup $R$ has a central one-parameter subgroup $\l:\CC^* \to Z(R)$
which grades the 
 unipotent radical $U$ of $H$. 
Let $\Omega$ be an $H$-equivariant K\"ahler structure  on $X$, let $(K,\omega) \in \Omega$ and let $\newmmxh: \Omega \times X \to \mathfrak{h}^*$ be an {$\Omega$-moment map} for the $H$-action on $X$ in the sense of Definition \ref{defn:omegaH}. Let $\mu^{H}_{(K,\omega)}: X \to \lieh^*$ be the   restriction %$\mu^H_{(K,\omega)}: X \to \lieh^*$
 of $\newmmxh:\Omega \times X \to \lieh^*$ to $X$ identified with $\{(K,\omega)\} \times X$, 
and let $\mu^{\hU}_{(K,\omega)}:X \to \hat{\lieu}^*$  and $\mu^{U}_{(K,\omega)}:X \to \lieu^*$ be the compositions of $\mu^{H}_{(K,\omega)}$ with the restriction maps $\lieh^* \to \hat{\lieu}^*$
and $\lieh^* \to \lieu^*$.

\begin{lemma} \label{zminumoment}
Suppose that semistability coincides with stability for the action of $H$, in the sense of Definitions \ref{def:welladaptedactionK}  and \ref{cond starKH}, and  that $(K,\omega) \in \Omega$. Then
$\mu^{U}_{(K,\omega)}:X \to \lieu^*$ is constant on $Z_{\min}(X)$ with value 0.
\end{lemma}
\proof
Since $\omega$ is $\l(S^1)$-invariant and $\l(S^1) \leqslant K$ fixes each $y \in Z_{\min}(X)$, by Definition \ref{cond starK} $\mu^{U}_{(K,\omega)}(y)$ is fixed by the coadjoint action of $\l(S^1)$ on $\lieu^*$. The grading condition implies that if $\zeta \in \lieu^*$ is fixed by  the coadjoint action of $\l(S^1)$ then $\zeta$ is 0, so $\mu^{U}_{(K,\omega)}(y)=0$ as required.
\qed

\begin{remark} \label{rem:zminumoment} We can see that $\mu^{U}_{(K,\omega)}:X \to \lieu^*$ is constant on $Z_{\min}(X)$ from the differential equation \eqref{one2} without using equivariance, as follows.
The condition that $\l(\CC^*)$ acts on $\lieu$ with only strictly positive weights %and that $Z_{\min}(X) \subseteq Z_{\min}(X)(Y)$
 ensures that it acts on $\iota_K(\lieu)$ with only strictly negative weights. Thus 
  if $y \in Z_{\min}(X)$ and $a \in \lieu$ then $\iota_K(a)_y$ belongs to the sum of the weight spaces in $T_y Z_{\min}(X)$ with strictly negative weights, and this is 0 by the definition of $Z_{\min}(X)$.
  Hence
$\eta_{\omega,y}(\iota_K(a)_y, \xi) = 0$  for all $\xi \in T_yX$, and so by \eqref{one2}
 the derivative at  $y$  of the complex conjugate of the complex-valued function on $Z_{\min}(X)$ given by evaluating $2i\mu^{U}_{(K,\omega)}$ at $a$ takes $(0,  %b_{(K,\omega)},
\xi)$ to 
$
\eta_{\omega,x}(a_x,\xi)$, which is complex linear in $\xi$. Thus this function is holomorphic on the compact K\"ahler manifold $Z_{\min}(X)$, and so is constant, with zero derivative.
\end{remark}

\begin{remark} \label{rem:constant}
The vector field $x \mapsto \l_*(s)_x$ on $X$ induced by any $\l_*(s) \in \l_*(\mathrm{Lie}(S^1))$ restricts to zero on $Z_{\min}(X)$.
 The grading condition implies that $\lieu = [\l_*(\RR),\lieu]$ where we identify $\RR$ %and $\CC$
  with the Lie algebra of $S^1$. % and $\CC^*$. 
  As $\omega$ is $\l(S^1)$-invariant, 
  this gives us another way to see that 
   ${\eta_{\omega,y}(b_y,\xi)}= 0$ if $y \in Z_{\min}(X)$ and $\xi \in T_y Z_{\min}(X)$ and $b \in \lieu$.
  \end{remark}
  
 \begin{lemma} \label{borel}
  If $b \in \hat{\lieu} \setminus \lieu$ then $b \in \mathrm{Ad}(U)(\lambda_*(\mathrm{Lie}(\CC^*)))$.
  If $g \in U \rtimes \l(\RR^{>0}) \setminus U$ then $g = u\l(t) u^{-1}$ for some $u \in U$ and $t \in \RR^{>0}$.
  \end{lemma}
  \proof
  Suppose that $c,d \in \lieu$ and $z \in \CC$ satisfy $[\lambda_*(z) + c,d]=0$. Then $[\lambda_*(z),d] = [d,c]$, and so it follows from the grading condition on $\lambda: \CC^* \to \hU$ that if $d \neq 0$ then $z=0$ and $[d,c]=0$.
  
  Now if $b \in \hat{\lieu}$ then by Jordan decomposition $b = b_s + b_n$ where $b_s \in \hat{\lieu}$ is semisimple, $b_n \in \hat{\lieu}$ is nilpotent and $[b_s,b_n]=0$. Since $U$ is the unique maximal unipotent subgroup of $\hU$, we have $b_n \in \lieu$ and $b_s = \lambda_*(z) + c$ where $c \in \lieu$ and $z \in \CC$ and $[\lambda_*(z) + c,b_n]=0$. As above this implies that either $z=0$, so that $b \in \lieu$, or $b_n=0$ so that $b$ is semisimple. 
  
  Thus if $b \in \hat{\lieu} \setminus \lieu$ then $b$ is semisimple, and so $b$ lies in the adjoint orbit of an element of the Lie algebra of the maximal torus $\lambda(\CC^*)$ of $\hU$.
  
  Similarly if $g \in \hU_{\RR} = U \rtimes \l(\RR^{>0})$ then by the multiplicative Jordan decomposition $g = g_s g_u$ where $g_s$ is semisimple, $g_u$ is unipotent and $g_s g_u = g_u g_s$. As in the Lie algebra setting, $g_u \in U$ and $g_s = \l(t)u$ for some $t \in \RR^{>0}$ and $u \in U$. Then $g_u \l(t) u = \l(t) u g_u$ so $g_u = \l(t) u g_u u^{-1} \l(t)^{-1}$. By the grading condition this implies that either $t^r = 1$ for some positive integer $r$, and hence $t=1$, so that $g \in U$, or $g_u = 1$ so that $g$ is semisimple and hence conjugate to $\l(t)$ for some $t \in \RR^{>0}$.
  \qed

 \begin{lemma} \label{lem:x0min}
Suppose that semistability coincides with stability for the action of $H$, in the sense of Definitions \ref{def:welladaptedactionK}  and \ref{cond starKH}, and  that $(K,\omega) \in \Omega$. Then
\[  X^0_{\min} \setminus U Z_{\min}(X)  =   %X^{s,\hU}_{\min+}= %X^{ss,\hU}_{\min+}=
\{x\in X^0_{\min}: \dim \Stab_{\hU}(x)=0\}.\]
\end{lemma}
\proof
%By definition $X^{s,\hU}_{\min+}= %X^{ss,\hU}_{\min+}= X^0_{\min} \setminus UZ_{\min}(X)$.
Recall from 
Remark \ref{remss=s}(ii) that $\stab_{U}(x) = \{e\}$ for all $x \in X^0_{\min}$. By Lemma \ref{borel}, this implies that $\dim \Stab_{\hU}(x) > 0$ for $x \in X^0_{\min}$ if and only if $x$ is fixed by a conjugate of $\l(\CC^*)$ in $\hU$. By the definition of $X^0_{\min}$, the only points in $X^0_{\min}$ fixed by $\l(\CC^*)$ are those in $Z_{\min}(X)$. 
\qed

\medskip

Crucial ingredients in the proof of Theorem \ref{thm:mainA} 
are the slice theorems (Propositions \ref{sliceU}, \ref{sliceUhat} and \ref{sliceH} below), %We formulate and prove these first for $U$-actions, followed by $\hU$-actions and finally in full generality for $H$-actions. Each proof heavily relies on the previous ones, but needs new ingredients.
and the following simple geometric statement is used in the proof of these. % and their precursor Proposition \ref{slice}.

\begin{lemma}[The gradient trick]\label{gradient}  Let $X$ be a smooth manifold acted on by a Lie group $H$, let $a\in \Lie(H)$ and let $Y \subseteq f^{-1}(0) \subset X$ be a union of finitely many compact connected components of  the vanishing set of a smooth map $f: X \to \RR$. \\
(i) Assume  that $df(a_x)>0$ holds for all $x\in Y$. If $y\in Y$ and $\exp(ta)y \in Y$ for some $t \in \RR$, then $t=0$. \\
(ii) Now assume that $g:X \to \RR$ is a smooth map with minimum value $\b_{\min}$ on $X$, and that %$a\in \Lie(H)$ and
 $\b>\b_{\min}$ is such that $df(a_x)>0$ for all $x\in Y \cap g^{-1}([\b_{\min},\b])$ and $dg(a_x)>0$ for all $x\in g^{-1}((\b_{\min},\b])$. If  $y\in Y \cap g^{-1}((\b_{\min},\b])$ and $\exp(ta)y \in Y \cap g^{-1}((\b_{\min},\b])$ for some $t \in \RR$, then $t=0$. 
\end{lemma} 
\proof
(i) By reversing the roles of $y$ and $\exp(ta)y$ if necessary, it is enough to prove that if  $y\in Y$ and $\exp(ta)y \in Y$ for some $t \geq 0$, then $t=0$. 

So suppose $y \in Y$. By continuity there is an open neighbourhood $V$ of $Y$ in $X$ such that $df(a_x) > 0$ %\geqslant \delta$
 for all $x \in V$.   %\overline{V}$. We may assume that $V$ is a connected component of $f^{-1}(-\epsilon,\epsilon)$ for some $\epsilon > 0$ where $\pm \epsilon$ are regular values of $f$. 
%If $y\in Y$ then
Then
\begin{equation} \frac{d}{dt} (f(\exp(ta) y) = df(a_{\exp(ta)y})  > 0 %\geqslant \delta 
\label{Vbar}
 \end{equation}
whenever $\exp(ta)y \in V$. %\overline{V}$.
Thus $f(\exp(ta)y)$ takes the value 0 when $t=0$ and is a strictly increasing function of $t$ near 0. %so long as $\exp(ta)y \in \overline{V}$.
 So if $\exp(ta)y \in Y$ for some $t>0$ then, %\not\in \overline{V}$ for some $t_0 > 0$ then $\exp(ta)y \not\in \overline{V}$ for all $t \geqslant t_0$, since otherwise
 by taking $t_1$ to be the infimum of such $t$, we find that $t_1  > 0$ and  $f(\exp(sa)y) > 0$ for all $s\in (0,t_1)$, but $\exp(t_1 a)y \in Y$ and so $f(\exp(t_1 a)y) = 0$, contradicting \eqref{Vbar}. Hence if $t>0$ then %$f(\exp(ta)y) >0$ and so
  $\exp(ta)y \not\in Y$, as required.

(ii) The proof of (i) can be modified slightly as follows. 
Again by reversing the roles of $y$ and $\exp(ta)y$ if necessary, it is enough to prove that if  $y\in Y \cap g^{-1}((\b_{\min},\b])$ and $\exp(ta)y \in Y \cap g^{-1}((\b_{\min},\b])$ for some $t \geq 0$, then $t=0$. 

If $g(\exp(t_0 a)y) =\b$ for some $t_0 \geq 0$ then $g(\exp(s a)y) > \b$ for all $s>t_0$ near $t_0$, since $dg(a_x)>0$ for all $x\in g^{-1}((\b_{\min},\b])$ and hence for $x= \exp(t_0 a)y$. So if there is some $t_1 > t_0$ such that $g(\exp(t_1 a)y) =\b$, then we can choose $t_1$ to be minimal and hence assume that $g(\exp(s a)y) > \b$ for all $s \in (t_0,t_1)$, contradicting the assumption that $dg(a_x)>0$ for $x= \exp(t_1 a)y$. So if $g(\exp(t_0 a)y) = \b$ for some $t_0\geq 0$ and $t >t_0$, then $g(\exp(t a)y) > \b$ and hence $\exp(ta)y \not\in Y \cap g^{-1}((\b_{\min},\b])$. Thus it is enough to prove that if  $y\in Y \cap g^{-1}((\b_{\min},\b])$ and $\exp(ta)y \in Y \cap g^{-1}((\b_{\min},\b])$ and $g(\exp(sa)y) < \b$ for all $s\in [0,t]$ where $t\geq 0$, then $t=0$. 

So let us assume for a contradiction that $t>0$ and $y\in Y \cap g^{-1}((\b_{\min},\b])$ and $\exp(ta)y \in Y \cap g^{-1}((\b_{\min},\b])$ and $g(\exp(sa)y) < \b$ for all $s\in [0,t]$.
%By compactness there is some $\delta >0$ and
By continuity there is an open 
neighbourhood $V$ of $Y  \cap g^{-1}([\b_{\min},\b])$ in $g^{-1}[\b_{min},\b]$ such that $df(a_x) %\geqslant \delta
 >0$ for all $x \in \overline{V}$. 
 By compactness we can choose $V$ to be a union of finitely many connected components of $(f,g)^{-1}((-\epsilon,\epsilon) \times [\b_{\min},\b])$ for some $\epsilon > 0$, and by Sard's theorem we can assume that  $\pm \epsilon$ are regular values of  the restriction of $f$ to $g^{-1}([\b_{\min},\b))$. Since $\overline{V} \subseteq g^{-1}[\b_{min},\b]$ and $y\in Y \cap g^{-1}((\b_{\min},\b])$, both  $f(\exp(sa)y)$ and $g(\exp(sa)y)$ are strictly increasing functions of $s$ with strictly positive derivatives so long as $s \geq 0$ and  $\exp(sa)y \in \overline{V}$. But by assumption $\exp(ta)y \in Y$ so $f(\exp(ta)y) = 0 = f(y)$.
 So we must have $\exp(t_0a)y) \not\in \overline{V}$ for some $t_0 \in (0,t)$, while $y \in Y \subseteq V$ and $\exp(ta)y \in Y \subseteq V$. 
 
 Thus %,  since $g(\exp(sa)y) < \b$ for all $s\in [0,t]$,
  there exist $t_1 \in (0,t_0)$ and 
 $t_2 \in (t_0,t)$ such that 
 $\exp(t_1 a)y \in \partial V$ and $\exp(t_2 a)y \in \partial V$. Since $\b_{\min} < g(\exp(sa)y) < \b$ for all $s\in [0,t]$, this means that $f(\exp(t_1 a)y) = \pm \epsilon$ and $f(\exp(t_2 a)y) = \pm \epsilon$.
 Also  the function $s \mapsto f(\exp(sa)y)$ has strictly positive derivative when $s \in \overline{V}$ (in particular when $s = t_1$ and when $s=t_2$), and takes the value 0 when $s=0$. Hence by choosing $t_1$ minimal such that $t_1 > 0$ and  $\exp(t_1 a)y \in \partial V$, we can assume that $f(\exp(t_1 a)y) =  \epsilon$ and that $f(\exp(s a)y) > \epsilon$ when $s>t_1$ is near $t_1$. Now we can choose $t_2$ minimal such that $t_2 > t_1$ and  $\exp(t_2 a)y \in \partial V$, and then by the intermediate value theorem $f(\exp(t_2 a)y) =  \epsilon$, while $f(\exp(s a)y) > \epsilon$ when $s \in (t_1,t_2)$. This gives us a contradiction because  the function $s \mapsto f(\exp(sa)y)$ has strictly positive derivative when  $s=t_2$.
 
\qed

\medskip

\begin{remark} \label{lem:delta}
If $\delta>0$ %and $\alpha \geq 0$
 then the open subset $(\mu^{\l(S^1)}_{(K,\omega)})^{-1}(-\infty, \weight_{\min} + \delta)$ of $X$ is a neighbourhood of $Z_{\min}(X)$, and   
\[Z_{\min}(X)= \bigcap_{\delta > 0} \,\, (\mu^{\l(S^1)}_{(K,\omega)})^{-1}(-\infty, \weight_{\min} + \delta). \]
Since $\mu^{\l(S^1)}_{(K,\omega)}$ is a Morse--Bott function \cite{atiyah} with gradient flow along the orbits of $\l(\RR^*)$  (see \eqref{doubledagger}) and minimum $Z_{\min}(X)$, if $\delta_0 >0$ is sufficiently small then  $(\mu^{\l(S^1)}_{(K,\omega)})^{-1}(\weight_{\min} + \delta_0)$ is a closed $\l(S^1)$-invariant submanifold of $X$, and the restriction
$$(0,\infty)  \times (\mu^{\l(S^1)}_{(K,\omega)})^{-1}(\weight_{\min} + \delta_0) \to (\mu^{\l(S^1)}_{(K,\omega)})^{-1}(\weight_{\min}, \weight_{\min} + \delta_0) = (\mu^{\l(S^1)}_{(K,\omega)})^{-1}(-\infty, \weight_{\min} + \delta_0) \setminus Z_{\min}(X)$$
of $m:(0,\infty) \times X \to X$ defined by $(t,x) \mapsto \l(e^{-2\pi t})x$ is a diffeomorphism, and for $x \in (\mu^{\l(S^1)}_{(K,\omega)})^{-1}(\weight_{\min} + \delta_0)$ the map
$(0,\infty) \to (\weight_{\min}, \weight_{\min} + \delta_0)$ defined by
$$t \mapsto \mu^{\l(S^1)}_{(K,\omega)}(\l(e^{-2\pi t})x)$$
is an orientation-reversing diffeomorphism (so it has strictly negative derivative everywhere).
\end{remark}

\begin{lemma} \label{claim} Suppose that semistability coincides with stability for the action of $\hat{U}$, in the sense of Definition %s \ref{def:welladaptedaction}  and
\ref{cond star}, and  that $(K,\omega) \in \Omega$. 
Fix a norm $|| \cdot ||$ on the Lie algebra of $H = U \rtimes R$ which is invariant under the adjoint action of $\l(S^1)$.
Given any sufficiently small $\delta_0 >0$, there exist $E_0 \ge \varepsilon_0 >0$ and $E_1 \ge \varepsilon_1 >0$ and $E_2 <1$
 such that if $0<\delta \leq \delta_0$ and $a \in \mathrm{Lie} R$ and  $b \in \mathrm{Lie} U$ then
%$$ d\mu^{\l(S^1)}_{(K,\omega)}(y)\cdot b_y \geq \varepsilon_0 \sqrt{\delta} ||b||$$
$$
\varepsilon_0||b||\sqrt{\delta} \le d\mu^{\l(S^1)}_{(K,\omega)}(y)(b_y)\le E_0||b||\sqrt{\delta} $$ and $$\varepsilon_1 {\delta} \le \eta_{\omega,y}(\l_*(1)_y,\l_*(1)_y) = ||\l_*(1)_y||_\eta^2  \le E_1 {\delta}$$
and     $$|\eta_{(K,\omega)}(\l_*(1)_y,a_y)|  \le E_2 || \l_*(1)_y||_\eta  ||a_y||_\eta $$ 
whenever $y \in (\mu^{\l(S^1)}_{(\omega,y)})^{-1}(\weight_{\min} + \delta)$, where $||\xi||_\eta = \sqrt{\eta_{\omega,y}(\xi,\xi)}$ for any $\xi \in T_yX$. %, and there exist $E_1 \ge \varepsilon_1 >0$ such that $$\varepsilon_1\sqrt{\delta} \le ||\l_*(1)_y||_\eta \le E_1\sqrt{\delta} \text{ when } \mu^{\l(S^1)}_{(K,\omega)}(y)=\omega_{\min}+\delta \text{ and } 0<\delta<\delta_0.$$
\end{lemma}
\proof For $x \in X_{\min}^0$ and $z \in \CC$ we let 
\[\varphi(x,z)=\begin{cases} \l(z)x & z\neq 0 \\ \lim_{t\to 0}\l(t)x & z=0 \end{cases}\]
As in Remark \ref{lem:delta}, $\mu^{\l(S^1)}_{(K,\omega)}$ is a Morse--Bott function  with gradient flow along the orbits of $\l(\RR^*)$ and $Z_{\min}(X)$ is its minimum. %and $Z_{\min}(X)$ is its minimum and its gradient flow is given by orbits of $\RR^{>0}$,
So the closure in $X^0_{\min}$ of the downwards gradient flowline for  $\mu^{\l(S^1)}_{(K,\omega)}$ from $x \in X^0_{\min} \setminus Z_{\min}(X)$ is 
$$ \{ \phi(x,t): t \in [0,1]\}.$$
Moreover
$\varphi: X_{\min}^0 \times \CC \to X_{\min}^0$ is smooth, $\CC^*$-equivariant, and
\[\mu^{\l(S^1)}_{(K,\omega)}(\varphi(x,t))=\omega_{\min}+\psi(x,t)^2 \text{ for } t \in \RR, x\in X_{\min}^0 \setminus Z_{\min}(X),\]
where $\psi: (X_{\min}^0 \setminus Z_{\min}(X)) \times \RR \to \RR$ is smooth and $t \mapsto \psi(x,t)$ is an orientation preserving diffeomorphism of $\RR$ for all $x \in X_{\min}^0 \setminus Z_{\min}(X)$ so \[\frac{\partial \psi}{\partial t}(x,t)>0 \text{ for all } t \in \RR, x \in X_{\min}^0 \setminus Z_{\min}(X),\]
and $$\psi(x,-t)=-\psi(x,t) \text{ for all } t \in \RR, x \in X_{\min}^0 \setminus Z_{\min}(X).$$ Then since $\mu^{\l(S^1)}_{(K,\omega)}$ is $\l(S^1)$-invariant, we have 
\begin{equation}\label{morsebott}
\mu^{\l(S^1)}_{(K,\omega)}(\varphi(x,z))=\omega_{\min}+\psi(x,|z|)^2 \text{ for } z \in \CC, x\in X_{\min}^0 \setminus Z_{\min}(X).
\end{equation}
Fix $\delta_0>0$ such that $\omega_{\min}+\delta_0<\omega_1$. Then for all $x \in X_{\min}^0 \setminus Z_{\min}(X)$ there is a unique $\l(S^1)$-orbit in $\l(\CC^*)x$ on which $\mu^{\l(S^1)}_{(K,\omega)}$ takes the value $\delta_0$, and a diffeomorphism 
\[(\mu^{\l(S^1)}_{(K,\omega)})^{-1}(\omega_{\min}+\delta_0) \times (0,\delta_0) \to (\mu^{\l(S^1)}_{(K,\omega)})^{-1}(\omega_{\min}, \omega_{\min}+\delta_0),\] 
extending continuously to 
$(\mu^{\l(S^1)}_{(K,\omega)})^{-1}(\omega_{\min}+\delta_0) \times [0,\delta_0] \to (\mu^{\l(S^1)}_{(K,\omega)})^{-1}[\omega_{\min}, \omega_{\min}+\delta_0]$,
defined as $(x,\delta) \mapsto \l(\tau_{\delta_0}(x,\delta))x = \phi(x, \tau_{\delta_0}(x,\delta))$ where 
\[\tau_{\delta_0}: (\mu^{\l(S^1)}_{(K,\omega)})^{-1}(\omega_{\min}+\delta_0) \times (0,\delta_0) \to (0,1)\]
is such that $\tau_{\delta_0}(x,\delta)$ is the unique $t \in (0,1)$ satisfying %the following equivalent equations:
\[\mu^{\l(S^1)}_{(K,\omega)}(\l(t)x)=\omega_{\min}+\delta. \] 
Thus $\tau_{\delta_0}(x,\delta) \in (0,1)$ is uniquely determined by the equivalent conditions $ \mu^{\l(S^1)}_{(K,\omega)}(\varphi(x,\tau_{\delta_0}(x,\delta)))=\omega_{\min}+\delta$ and  $\psi(x,\tau_{\delta_0}(x,\delta))^2=\delta$.
Since we chose $\psi$ to be orientation-preserving, we have $\psi(x,\tau_{\delta_0}(x,\delta))=\sqrt{\delta}\ge 0$ for all $x \in X_{\min}^0 \setminus Z_{\min}(X)$ and $\delta \in [0,\delta_0]$. 

It follows from Lemma \ref{lem:x0min}  that $b_y \neq 0$ if $y \in Z_{\min}(X)$ and $b\neq 0$, and then by Remark \ref{rem:constant} $b_y \notin T_yZ_{\min}(X)$. Hence by the compactness of $Z_{\min}(X)$ and of $\{b \in \lieu: ||b|| = 1\}$, % and of $(\mu^{\l(S^1)}_{(K,\omega)})^{-1}(-\infty, \weight_{\min} + \delta_0]$,
for any sufficiently small $\delta_0>0$ %\ll 1$ there is 
we can find $E_0\ge \varepsilon_0>0$ independent of $b$ and $y$ such that 
$$
\varepsilon_0||b||\sqrt{\delta} \le d\mu^{\l(S^1)}_{(K,\omega)}(y)(b_y)\le E_0||b||\sqrt{\delta} \text{ holds if } y=\varphi(x,\tau_{\delta_0}(x,\delta)) \text{ and } \mu^{\l(S^1)}_{(K,\omega)}(x)=\omega_{\min}+\delta_0 \text{ for some } 0<\delta\leq \delta_0,
$$ or equivalently
\begin{equation}\label{epsilon0}
\varepsilon_0||b||\sqrt{\delta} \le d\mu^{\l(S^1)}_{(K,\omega)}(y)(b_y)\le E_0||b||\sqrt{\delta} \text{ holds if } 
\mu^{\l(S^1)}_{(K,\omega)}(y)=\omega_{\min}+\delta \text{ and } 0<\delta\leq \delta_0.
\end{equation} 
Furthermore since $\mu^{\l(S^1)}_{(K,\omega)}(y)$ is a Morse--Bott function, near its minimum $Z_{\min}(X)$ the ratio of $\mu^{\l(S^1)}_{(K,\omega)}(y) - \weight_{\min}$ to the normsquare of its gradient $\l_*(1)_y$ at $y$ is bounded above and below by strictly positive constants, so we can assume that $\varepsilon_1 {\delta} \le \eta_{\omega,y}(\l_*(1)_y,\l_*(1)_y) = ||\l_*(1)_y||_\eta^2  \le E_1 {\delta}$ for  $y \in (\mu^{\l(S^1)}_{(\omega,y)})^{-1}(\weight_{\min} + \delta)$ as required.

Finally if $a \in \mathrm{Lie} R$ then $[a,\l_*(1)]  = 0$ and so $\exp(\RR a)$ preserves $Z_{\min}(X)$. Thus $a_y \in T_y Z_{\min}(X)$ and $i a_y \in T_y Z_{\min}(X)$
 if $y \in Z_{\min}(X)$, while the closure in $X^0_{\min}$ of any downwards gradient flowline for  $\mu^{\l(S^1)}_{(K,\omega)}$ from a point in $X^0_{\min} \setminus Z_{\min}(X)$ meets $Z_{\min}(X)$ transversally. 
 Since $\mu^{\l(S^1)}_{(K,\omega)}$ is a Morse--Bott function, the space of these flowlines is homeomorphic to a sphere bundle over $Z_{\min}(X)$ and hence is compact. The unit sphere in $\mathrm{Lie} R$ is also compact, so
  it follows that there is some $E_2 \in [0,1)$ such that if $y \in Z_{\min}(X)$ is the limit of a downwards gradient flowline for  $\mu^{\l(S^1)}_{(K,\omega)}$ and  $a_y \neq 0$
 then the cosine of the angle between this flowline and  $ a_y$ lies in the interval $[-E_2,E_2]$. By continuity we can assume that this is also true for  $y \in (\mu^{\l(S^1)}_{(\omega,y)})^{-1}(\weight_{\min} + \delta)$
 if $0 < \delta < \delta_0$, provided that $\delta_0$ is chosen sufficiently small. Since the gradient of $\mu^{\l(S^1)}_{(K,\omega)}$ at $y$ is $ \l_*(1)_y$, the result follows.
%d\mu^{\l(S^1)}_{(K,\omega)}(y)\cdot 
\qed

 \begin{definition}
 For  $b \in \lieh = \mathrm{Lie}(H)$ define $f_b: X \to \RR$ by $f_b(x) = \mathrm{Im} \, \mu^{H}_{(K,\omega)}(x).b$.
 Note that  $f_b(x) = \mathrm{Im} \, \mu^{U}_{(K,\omega)}(x).b$ if $b \in \lieu = \mathrm{Lie}(U)$.
\end{definition}

\begin{lemma}\label{claimU}  
Suppose that semistability coincides with stability for the action of $\hat{U}$, in the sense of Definition %s \ref{def:welladaptedaction}  and
 \ref{cond star}, and  that $(K,\omega) \in \Omega$. 
Given any sufficiently small $\delta_0 >0$, there exists $\varepsilon >0$ such that if $0<\delta \leq \delta_0$ then
\[ \mathrm{d}f_b(y)(b_y)                                                     %\mu^{U}_{(K,\omega)}(\exp(b)x)\cdot a_y 
\ge \varepsilon ||b||^2\]
whenever $b \in \mathrm{Lie}U$ and $y \in (\mu^{\l(S^1)}_{(K,\omega)})^{-1}(\weight_{\min} + \delta)$.
\end{lemma}

\proof 
We can assume, for simplicity of exposition and without essential loss of generality, that $X=Y$; this is where we use the requirement in Definitions \ref{cond starK} and
\ref{cond starKH} that $Z_{\min}(X) \subseteq Z_{\min}(Y)$. 

By \eqref{one2} in Definition \ref{defn:omegaH} the derivative of $f_b$ at $y$ is given for $b \in \lieu$ and $\xi \in T_y X$ by $\mathrm{Re} \, i \mathrm{d}\mu_{(K,\omega)}^{U}(y)(\xi)$ where
%\[\frac{d h}{dt }|_{(b,x)}(a,\xi) =
\[  \mathrm{d}\mu_{(K,\omega)}^{U}(y)(\xi)\cdot b
=  ({\eta_{\omega,y}(\xi,b_y)} - \eta_{\omega,y}(\iota_K(b)_y, \xi))/2i.
\] 
Here $\eta_{\omega,y}$ is the value at $y$ of the Hermitian metric determined by the $K$-invariant K\"ahler form $\omega$, while $\iota_K$ is the $K$-invariant involution on $\mathfrak{g}$ with fixed point set $\mathfrak{k}$, and $x \mapsto b_x$ is the holomorphic vector field on $X$ determined by the infinitesimal action of $b$. 
%We choose $a$ to be $a=ib$, so that 
Setting $\xi = b_y$ gives us  
\[  \mathrm{d}f_b(y)(b_y)                                %2i\, \mathrm{d} h |_{(0,x)}(b,0)\cdot a
 =  (\eta_{\omega,y}(b_y,b_y) - \mathrm{Re} \, \eta_{\omega,y}(\iota_K(b)_y, b_y))/2.
\]
By Lemma \ref{lem:x0min} $b \mapsto b_y$ is injective on $\lieu$ for $y$ sufficiently close to $Z_{\min}(X)$, and so the sesquilinear form defined on $\lieu$ by $(a,b) \mapsto {\eta_{\omega,y}(a_y,b_y)}$ is positive definite. By continuity and compactness there is some $\varepsilon >0$ such that $\eta_{\omega,y}(b_y,b_y) \geq 4\varepsilon $ for all $b \in \lieu$ with $||b|| = 1$ and all $y$ in some compact neighbourhood of $Z_{\min}(X)$; equivalently 
\[\eta_{\omega,y}(b_y,b_y)/2 \geq 2\varepsilon ||b||^2\]
for all $b \in \lieu$  and all $y$ in this compact neighbourhood of $Z_{\min}(X)$.
On the other hand, by Remark \ref{rem:zminumoment}
%The condition that $\l(\CC^*)$ acts on $\lieu$ with only strictly positive weights %and that $Z_{\min}(X) \subseteq Z_{\min}(X)(Y)$
% ensures that it acts on $\iota_K(\lieu)$ with only strictly negative weights. Thus 
%  if $y \in Z_{\min}(X)$ and $a \in \lieu$ then $\iota_K(a)_y$ belongs to the sum of the weight spaces in $T_y Z_{\min}(X)$ with strictly negative weights, and this is 0 by the definition of $Z_{\min}(X)$.
%  Hence
  \begin{equation}\label{iota}
\eta_{\omega,y}(\iota_K(a)_y, b_y) = 0 \text{ when } y \in Z_{\min}(X), a \in \lieu
\end{equation} %which implies that $\eta_{\omega,y}(\iota_K(a)_y, b_y) = 0$,
and by continuity and compactness again we have 
\begin{equation}\label{closetozmin}
|\eta_{\omega,y}(\iota_K(a)_y, b_y)|/2 \leq \varepsilon  ||a||\, ||b|| \text{ for all } a,b \in \lieu \text{ and $y$ sufficiently close to } Z_{\min}(X).
\end{equation}
Taking $a=b$, the result follows by the triangle inequality.
\qed

\begin{proposition} {[\textbf{Local slice theorem for $U$}]} \label{sliceU} Suppose that semistability coincides with stability for the action of $\hat{U}$, in the sense of Definition %s \ref{def:welladaptedaction}  and
 \ref{cond star}, and  that $(K,\omega) \in \Omega$. 
If $\delta_0 > 0$ is small enough then
\begin{enumerate}[(i)]
\item  the intersections with $(\mu^{\l(S^1)}_{(K,\omega)})^{-1}(-\infty, \weight_{\min} + \delta_0)$ of $(\mu^{U}_{(K,\omega)})^{-1}(0)$ and $U Z_{\min}(X)$ % and $U(\mu^{U}_{(K,\omega)})^{-1}(0)$
 are closed submanifolds of the open neighbourhood $(\mu^{\l(S^1)}_{(K,\omega)})^{-1}(-\infty, \weight_{\min} + \delta_0)$ of $Z_{\min}(X)$ in $X$;
\item  %the intersection with $(\mu^{\l(S^1)}_{(K,\omega)})^{-1}(-\infty, \weight_{\min} + \delta_0)$ of %$(\mu^{U}_{(K,\omega)})^{-1}(0)$ and $U Z_{\min}(X)$ and 
$U(\mu^{U}_{(K,\omega)})^{-1}(0)$ %are closed submanifolds of
contains the open neighbourhood $(\mu^{\l(S^1)}_{(K,\omega)})^{-1}(-\infty, \weight_{\min} + \delta_0)$ of $Z_{\min}(X)$ in $X$;
\item %\[(\mu^{U}_{(K,\omega)})^{-1}(0) \cap U Z_{\min}(X) \cap (\mu^{\l(S^1)}_{(K,\omega)})^{-1}(-\infty, \weight_{\min} + \delta_0) = Z_{\min}(X),\]
if 
$x \in (\mu^{U}_{(K,\omega)})^{-1}(0) \cap (\mu^{\l(S^1)}_{(K,\omega)})^{-1}(-\infty,\omega_{\min}+\delta_0) \text { and } ux \in (\mu^{U}_{(K,\omega)})^{-1}(0) \cap (\mu^{\l(S^1)}_{(K,\omega)})^{-1}(-\infty,\omega_{\min}+\delta_0)$ for some $u \in U$ then $u=1$, and 
the  map  
$U \times (\mu^{U}_{(K,\omega)})^{-1}(0) \to U (\mu^{U}_{(K,\omega)})^{-1}(0)$
defined by $(u,x) \mapsto ux$ is a local diffeomorphism near $\{1\} \times Z_{\min}(X)$; %when restricted to the pre-image of  $ %U(\mu^{U}_{(K,\omega)})^{-1}(0) \cap (\mu^{\l(S^1)}_{(K,\omega)})^{-1}(-\infty, \weight_{\min} + \delta_0)$;
\item 
$UZ_{\min}(X) \cap (\mu^{U}_{(K,\omega)})^{-1}(0)\cap (\mu^{\l(S^1)}_{(K,\omega)})^{-1}(-\infty, \omega_{\min}+\delta_0)=Z_{\min}(X).$
\end{enumerate}
\end{proposition}

\proof
As in the proof of Lemma \ref{claimU} we can assume that $X=Y$. 

For (i) and (ii), note that if $y \in Z_{\min}(X)$ then the tangent space of the orbit $T_y(Uy)$ is isomorphic to the Lie algebra $\lieu$ of $U$ and
\begin{equation}\label{isom}
d\mu^{U}_{(K,\omega)}(y)|_{T_y(Uy)}: T_y(Uy) \stackrel{\simeq}{\to} \lieu^*
\end{equation}
is an isomorphism defined by a positive definite pairing on $\lieu$ by Lemma \ref{claimU}. So if $\delta_0>0$ is sufficiently small then the same is true for
$y \in (\mu^{\l(S^1)}_{(K,\omega)})^{-1}(-\infty,\omega_{\min}+\delta_0)$. Therefore
$(\mu^{U}_{(K,\omega)})^{-1}(0) \cap (\mu^{\l(S^1)}_{(K,\omega)})^{-1}(-\infty,\omega_{\min}+\delta_0)$ is a closed submanifold of $(\mu^{\l(S^1)}_{(K,\omega)})^{-1}(-\infty,\omega_{\min}+\delta_0)$, with 
\[\{0\}=T_y(Uy) \cap T_y((\mu^{U}_{(K,\omega)})^{-1}(0)) \supseteq T_y(Uy) \cap T_yZ_{\min}(X)\]
by Lemma \ref{zminumoment}, and
the restriction $U \times (\mu^{U}_{(K,\omega)})^{-1}(0) \to X$ of the action of $U$ on $X$ has bijective derivative near $\{1\} \times Z_{\min}(X)$. Now (ii) follows since 
$(\mu^{U}_{(K,\omega)})^{-1}(0) \cap (\mu^{\l(S^1)}_{(K,\omega)})^{-1}(-\infty,\omega_{\min}+\delta_0)$ contains $Z_{\min}(X)$.

Since $\exp: \lieu \to U$ is surjective, for (iii) it remains to show that if 
\[x \in (\mu^{U}_{(K,\omega)})^{-1}(0) \cap (\mu^{\l(S^1)}_{(K,\omega)})^{-1}(-\infty,\omega_{\min}+\delta_0) \text { and } ux \in (\mu^{U}_{(K,\omega)})^{-1}(0) \cap (\mu^{\l(S^1)}_{(K,\omega)})^{-1}(-\infty,\omega_{\min}+\delta_0)\]
where %$0<\delta,\delta' \le \delta_0$ and
 $u=\exp(sb)$ for some $b\in \lieu$ with $||b||=1$ and $s \in \RR$, then $s=0$.  By Lemmas \ref{claim} and \ref{claimU} the functions 
\[g(s)=\mu^{\l(S^1)}_{(K,\omega)}(\exp(sb)x)(b) \text{ and } f(s)=\mu^{U}_{(K,\omega)}(\exp(sb)x)(b)\]
satisfy the conditions of the second part of the gradient trick (Lemma \ref{gradient})  with 
\[[\b_{\min},\b]=[\omega_{\min},\omega_{\min}+\delta_0], \,\,\, Y=(\mu^{U}_{(K,\omega)})^{-1}(0), \]
and so $s=0$ by Lemma \ref{gradient} (ii). This proves (iii) and completes the proof of (i) by showing that  the intersection with $(\mu^{\l(S^1)}_{(K,\omega)})^{-1}(-\infty, \weight_{\min} + \delta_0)$ of  $U Z_{\min}(X)$ % and $U(\mu^{U}_{(K,\omega)})^{-1}(0)$
 is a closed submanifold of  $(\mu^{\l(S^1)}_{(K,\omega)})^{-1}(-\infty, \weight_{\min} + \delta_0)$.

Finally it follows from Lemma \ref{zminumoment} that $UZ_{\min}(X) \cap (\mu^{U}_{(K,\omega)})^{-1}(0)\cap (\mu^{\l(S^1)}_{(K,\omega)})^{-1}(-\infty, \omega_{\min}+\delta_0) \supseteq Z_{\min}(X)$. The converse inclusion also holds, since if $u\in U$, $x \in Z_{\min}(X)$ and $y=ux \in (\mu^{U}_{(K,\omega)})^{-1}(0) \cap (\mu^{\l(S^1)}_{(K,\omega)})^{-1}(-\infty,\omega_{\min}+\delta_0)$, then by (iii) $u=1$ and hence $y=x\in Z_{\min}(X)$. 
\qed

\begin{corollary} \label{cor2}
Suppose that semistability coincides with stability for the action of $\hat{U}$, in the sense of Definition %s \ref{def:welladaptedaction}  and
 \ref{cond star}, and  that $(K,\omega) \in \Omega$. 
 Then the $\hU$-stabiliser $\mathrm{Stab}_{\hU}(x)$ of any $x$ in $(\mu^{U}_{(K,\omega)})^{-1}(0)\cap (\mu^{\l(S^1)}_{(K,\omega)})^{-1}(\omega_{\min}, \omega_{\min}+\delta_0)$ is finite. %, and so $x \in X^0_{\min} \setminus UZ_{\min}(X)$. 
\end{corollary}
\proof
Since $\hU$ has finite stabilizers on $(\mu^{\l(S^1)}_{(K,\omega)})^{-1}(-\infty, \omega_{\min}+\delta_0)\setminus UZ_{\min}(X)$, this follows from Proposition \ref{sliceU} (iv)  and Lemma \ref{lem:x0min}. %by the $U$-slice theorem and Lemma \ref{lem:x0min}.
\qed

\begin{definition} For $x \in \RR$ we introduce the shorthand notation 
\[(\mu^{\hU}_{(K,\omega)})^{-1}(x)=(\mu^{U}_{(K,\omega)})^{-1}(0) \cap (\mu^{\l(S^1)}_{(K,\omega)})^{-1}(x);\]
equivalently the direct sum decomposition $\mathrm{Lie}\hU = \mathrm{Lie}U \oplus \mathrm{Lie}\l(\CC^*)$ allows us to identify $(\mathrm{Lie}\hU)^*$ with  $(\mathrm{Lie}U)^* \oplus (\mathrm{Lie}\l(\CC^*))^*$, and hence to identify $\RR$ with $(\mathrm{Lie}\l(S^1))^* \le (\mathrm{Lie}\l(\CC^*))^* \le (\mathrm{Lie}\hU)^*$.

Similarly given a direct sum decomposition $\mathrm{Lie}R = \mathrm{Lie}\bar{R} \oplus \mathrm{Lie}\l(\CC^*)$ for a subgroup $\bar{R}$ of $R$, we can identify $(\mathrm{Lie}H)^*$ with  $(\mathrm{Lie}U\rtimes \bar{R})^* \oplus (\mathrm{Lie}\l(\CC^*))^*$, and hence identify $\RR$ with $(\mathrm{Lie}\l(S^1))^* \le (\mathrm{Lie}\l(\CC^*))^* \le (\mathrm{Lie}H)^*$, and thus interpret
\[(\mu^{H}_{(K,\omega)})^{-1}(x)=(\mu^{U \rtimes \bar{R}}_{(K,\omega)})^{-1}(0) \cap (\mu^{\l(S^1)}_{(K,\omega)})^{-1}(x).\]
\end{definition}

\begin{proposition} {[\textbf{Slice theorem for $\hU$}]} \label{sliceUhat} 
Suppose that semistability coincides with stability for the action of $\hat{U}$ on $X$, in the sense of Definition %s \ref{def:welladaptedaction}  and
 \ref{cond star}, and  that $(K,\omega) \in \Omega$. 
If $\delta_0 > 0$ is small enough then for $0<\delta <\delta_0$ we have 
\begin{enumerate}[(i)]
\item if $\hat{u} \in \hU$ and $x\in (\mu^{\hU}_{(K,\omega)})^{-1}(\omega_{\min}+\delta)$ then 
\[\hat{u}x \in (\mu^{\hU}_{(K,\omega)})^{-1}(\omega_{\min}+\delta) \text{ if and only if } \hat{u}\in \l(S^1);\]
\item $\omega_{\min}+\delta$ is a regular value of $\mu^{\hU}_{(K,\omega)}$,  and 
$\hU(\mu^{\hU}_{(K,\omega)})^{-1}(\omega_{\min}+\delta)$ coincides with the open subset $X^0_{\min} \setminus UZ_{\min}(X)$ of~$X$; 
\item $\hU(\mu^{\hU}_{(K,\omega)})^{-1}(\omega_{\min}+\delta)$ is diffeomorphic to the quotient of the product $\hU \times (\mu^{\hU}_{(K,\omega)})^{-1}(\omega_{\min}+\delta)$ by the diagonal action of $\l(S^1)$ acting by right multiplication on $\hU$ and by the given left action on $(\mu^{\hU}_{(K,\omega)})^{-1}(\omega_{\min}+\delta).$
\end{enumerate}
\end{proposition}

\proof 
As before let $\hU_{\RR} = U \rtimes \l(\RR^{>0})$ with Lie algebra $\hat{\lieu}_\RR$. Then $\hU = \l(S^1) \hU_{\RR}$ and $\l(S^1) \cap  \hU_{\RR} = \{1\}$. Also $(\mu^{\hU}_{(K,\omega)})^{-1}(\omega_{\min}+\delta)$ is invariant under the action of $\l(S^1)$. So for (i) it suffices to show that if $\hat{u} \in \hU_{\RR}$ and $x\in (\mu^{\hU}_{(K,\omega)})^{-1}(\omega_{\min}+\delta)$ and 
$\hat{u}x \in (\mu^{\hU}_{(K,\omega)})^{-1}(\omega_{\min}+\delta)$ then $\hat{u} = 1$. If $\hat{u} \in U$ then this follows from Proposition \ref{sliceU} so we can assume that $\hat{u} \in \hU \setminus U$.
Then by Lemma \ref{borel} $\hat{u}$ is in the image of the exponential map $\exp: \hat{\lieu}_\RR \to \hU_{\RR}$, so we can assume that $\hat{u} = \exp( sb+\l_*(t))$ for some $b\in \lieu$  with $||b||=1$ and $s,t \in \RR$. By exchanging the roles of $x$ and $\hat{u}x$ and replacing $\hat{u}$ with its inverse, if necessary, we can assume that $t \ge 0$. Then by replacing $s$ and $b$ with $-s$ and $-b$, if necessary, we can assume that $s\ge 0$.

Thus let $\hat{b} = sb + \l_*(t) \in \hat{\lieu}_{\RR}$  with $||b||=1$ and $s,t \in [0,\infty)$, and define $f_{\hat{b}}: X \to \RR$ by $f_{\hat{b}}(x) = \mathrm{Im} \, \mu^{\hU}_{(K,\omega)}(x).\hat{b}$. We want to show that there exists $\delta_0 > 0$, independent of the choice of $\hat{b}$, such that if $0<\delta \le \delta_0$ then
\[ \mathrm{d}f_{\hat{b}}(y)(\hat{b}_y)                                                     %\mu^{U}_{(K,\omega)}(\exp(b)x)\cdot a_y 
>0\]
when $y \in (\mu^{\hU}_{(K,\omega)})^{-1}(\weight_{\min} + \delta)$ and $\hat{b} \neq 0$. Then (i) will follow from 
 the first part of the gradient trick (Lemma \ref{gradient}).       
So suppose $b\in \lieu$  with $||b||=1$ and $s,t \in [0,\infty)$. Since $i_K(\l_*(1))_y=-\l_*(1)_y$ we have that $\mathrm{d}f_{\hat{b}}(y)(\hat{b}_y) $ is given (up to a factor of 2) by
\begin{multline}\label{expanded}
 \mathrm{Re} (2i \, \mathrm{d}\mu^{\hU}_{(K,\omega)}(y)(sb_y+\l_*(t)_y)\cdot (sb_y+\l_*(t)_y))=||sb_y+t\l_*(1)_y||_\eta^2-\mathrm{Re} \, \eta_{\omega,y}(i_K(sb+t\l_*(1))_y,sb_y+t\l_*(1)_y)=\\
=s^2||b_y||^2_\eta+2t^2||\l_*(1)_y||_\eta^2+st\mathrm{Re}(3\eta_{\omega,y}(\l_*(1)_y,b_y)-\eta_{\omega,y}(i_K(b)_y,\l_*(1)_y))-s^2 \mathrm{Re} \, \eta_{\omega,y}(i_K(b)_y,b_y), 
\end{multline}
where for any $\xi \in T_yX$ we use $||\xi||_\eta$ as shorthand for the non-negative square root of $\eta_{\omega,y}(\xi,\xi)$.
By compactness, since $b_y \neq 0$ if $y\in Z_{\min}(X)$ we can find $0<\varepsilon_3\le E_3$ independent of the choice of $b$ with $||b||=1$ such that 
\begin{equation}\label{epsilon1} 
\varepsilon_3 \le ||b_y||_\eta \le E_3 \text{ when } \mu^{\l(S^1)}_{(K,\omega)}(y)=\omega_{\min}+\delta \text{ and } 0\le \delta \le \delta_0.
\end{equation}
By Lemma \ref{claim}, %\eqref{epsilon0},
 since $||b||=1$, we can assume $\delta_0 >0$ is such that
\begin{equation}\label{epsilon0modified}
\varepsilon_0 \sqrt{\delta} \le \mathrm{Re} \, \eta_{\omega,y}(\l_*(1)_y,b_y) \le E_0 \sqrt{\delta} \text{ when } \mu^{\l(S^1)}_{(K,\omega)}(y)=\omega_{\min}+\delta \text{ and } 0 \le \delta \le \delta_0
\end{equation} and 
\begin{equation}\label{epsilon2}
\sqrt{\varepsilon_1} \sqrt{\delta} \le
 ||\l_*(1)_y||_\eta \le \sqrt{E_1} \sqrt{\delta} \text{ when } \mu^{\l(S^1)}_{(K,\omega)}(y)=\omega_{\min}+\delta \text{ and } 0\le \delta \le\delta_0.
\end{equation}
Finally, since $i_K(b)_y=0$ but $b_y \neq 0$ for $y \in Z_{\min}(X)$ (see Remark \ref{rem:zminumoment}), given any $\varepsilon>0$ we can find $\delta_0=\delta_0(\varepsilon)$ so that %the following two inequalities hold:
\begin{equation}\label{ikbepsilon}
%\begin{cases} ||i_K(b)_y||_\eta < \varepsilon ||b|| \\ 
||i_K(b)_y||_\eta < \varepsilon ||b_y|| %\end{cases} 
\text{ when } \mu^{\l(S^1)}_{(K,\omega)}(y)=\omega_{\min}+\delta \text{ and } 0\le \delta \le \delta_0
\end{equation}
and hence by Cauchy-Schwarz %we arrive at 
\begin{equation}\label{s2term}
|\mathrm{Re} \, \eta_{\omega,y}(i_K(b)_y,b_y)| \le \varepsilon ||b_y||_\eta^2   
\text{ when } \mu^{\l(S^1)}_{(K,\omega)}(y)=\omega_{\min}+\delta \text{ and } 0\le \delta \le \delta_0
\end{equation}
and
\begin{equation}\label{stterm}
|\mathrm{Re} \, \eta_{\omega,y}(i_K(b)_y,\l_*(1)_y)| \le \varepsilon ||b_y||_\eta ||\l_*(1)_y||_\eta    
\text{ when } \mu^{\l(S^1)}_{(K,\omega)}(y)=\omega_{\min}+\delta \text{ and } 0\le \delta \le \delta_0. 
\end{equation}
We want to show that if $\mu^{\l(S^1)}_{(K,\omega)}(y)=\omega_{\min}+\delta \text{ and } 0 < \delta \le \delta_0$ then
$$s^2||b_y||^2_\eta+2t^2||\l_*(1)_y||_\eta^2+st\mathrm{Re}(3\eta_{\omega,y}(\l_*(1)_y,b_y)-\eta_{\omega,y}(i_K(b)_y,\l_*(1)_y))-s^2 \mathrm{Re} \, \eta_{\omega,y}(i_K(b)_y,b_y)$$ is strictly positive when $s,t \in [0,\infty)$ are not both zero.
It suffices to have 
$$ ||b_y||^2_\eta > \mathrm{Re} \, \eta_{\omega,y}(i_K(b)_y,b_y)$$
and 
$$\mathrm{Re}(3\eta_{\omega,y}(\l_*(1)_y,b_y)-\eta_{\omega,y}(i_K(b)_y,\l_*(1)_y))>0$$
when $\mu^{\l(S^1)}_{(K,\omega)}(y)=\omega_{\min}+\delta \text{ and } 0 < \delta \le \delta_0$. By the inequalities (\ref{epsilon1} - \ref{stterm}), we can satisfy this condition by taking $\varepsilon \in (0,1)$ such that
\[\varepsilon<\frac{3 \varepsilon_0}{E_3 \sqrt{E_1}}\]
and letting $\delta_0=\delta_0(\varepsilon)>0$ chosen as above. 
We have in fact shown that  there exist $\delta_0 > 0$ and $\a_0,\b_0,\gamma_0>0$, independent of the choice of $\hat{b} = sb + \l_*(t) \in \hat{\lieu}_{\RR}$  with $||b||=1$ and $s,t \in [0,\infty)$, such that if $0 \le \delta \leq \delta_0$ then 
\[ \mathrm{d}f_{\hat{b}}(y)(\hat{b}_y)                                                     %\mu^{U}_{(K,\omega)}(\exp(b)x)\cdot a_y 
\ge   \a_0 s^2   +  \b_0 st \sqrt{\delta}  + \gamma_0 t^2 \delta  \]
whenever $y \in (\mu^{\hU}_{(K,\omega)})^{-1}(\weight_{\min} + \delta)$. This completes the proof of (i).

%Moreover (ii) and (iii) will then follow (as in the proof of Proposition \ref{sliceU}) using
It now follows from Corollary \ref{cor2} that if $0<\delta  \le \delta_0$ and $y \in (\mu^{\hU}_{(K,\omega)})^{-1}(\omega_{\min}+\delta)$ then $T_y(\hU_\RR y) \simeq \hat{\lieu}_\RR$ and 
\[\mathrm{Im} \, \mathrm{d}\mu^{\hU}_{(K,\omega)}(y)|_{T_y(\hU_\RR y)}: T_y(\hU_\RR y) \to \hat{\lieu}_\RR^*\]
is an isomorphism. % defined by a positive definite pairing on $\hat{\lieu}_\RR$.
 This means (as in the proof of Proposition \ref{sliceU}) that $(\mu^{\hU}_{(K,\omega)})^{-1}(\omega_{\min}+\delta)$ is a closed submanifold of $X$, that  
$\hU(\mu^{\hU}_{(K,\omega)})^{-1}(\omega_{\min}+\delta)$ is an open subset of $X$ and that (iii) holds. Moreover 
for any $y \in (\mu^{\hU}_{(K,\omega)})^{-1}(\omega_{\min},\omega_{\min}+\delta_0)$ there is a unique $\xi_y \in T_y\hat{U}_{\RR}y$ (varying smoothly with $y$) such that 
$$\mathrm{d}\mu^{\hU}_{(K,\omega)}(y)(\xi_y).b = 0 \text{ for all } b \in \lieu \text{ and } \mathrm{d}\mu^{\l(S^1)}_{(K,\omega)}(y)(\xi_y) = 1.$$
%By the compactness of $(\mu^{\hU}_{(K,\omega)})^{-1}(\omega_{\min}+\delta)$ for $0<\delta < \delta_0$ and the connectedness of $(0,\delta_0)$, i
Integrating the vector field $\xi$ provides diffeomorphisms $\Xi_{\delta,\delta'}: (\mu^{\hU}_{(K,\omega)})^{-1}(\omega_{\min}+\delta) \to (\mu^{\hU}_{(K,\omega)})^{-1}(\omega_{\min}+\delta')$ for $\delta,\delta' \in (0,\delta_0)$ such that $\Xi_{\delta,\delta'}(y) \in \hat{U}_{\RR}y$ for all $y \in (\mu^{\hU}_{(K,\omega)})^{-1}(\omega_{\min}+\delta)$.
Thus $\hat{U} (\mu^{\hU}_{(K,\omega)})^{-1}(\omega_{\min}+\delta)$ is independent of $\delta \in (0,\delta_0)$. 

By Proposition \ref{sliceU} (iii) there is some $\delta_1 \in (0,\delta_0)$ such that 
$$(\mu^{\l(S^1)}_{(K,\omega)})^{-1}(-\infty,\omega_{\min}+\delta_1)  \subseteq U((\mu^{U}_{(K,\omega)})^{-1}(0) \cap (\mu^{\l(S^1)}_{(K,\omega)})^{-1}(-\infty,\omega_{\min}+\delta_0)    ). $$
By this and the observations made in the previous paragraph, since $X^0_{\min} \subseteq \l(\CC^*) (\mu^{\l(S^1)}_{(K,\omega)})^{-1}(-\infty,\omega_{\min}+\delta_1)$, it follows that for any $\delta \in (0,\delta_0)$
$$ X^0_{\min} \subseteq \bigcup_{0 \le \delta' < \delta_0} \,\, \hat{U} (\mu^{\hU}_{(K,\omega)})^{-1}(\omega_{\min}+\delta') = UZ_{\min}(X)\,  \cup \, \hat{U} (\mu^{\hU}_{(K,\omega)})^{-1}(\omega_{\min}+\delta),$$
and so 
$$ X^0_{\min} \setminus UZ_{\min}(X) \subseteq \hU (\mu^{\hU}_{(K,\omega)})^{-1}(\omega_{\min} + \delta) .$$
Conversely 
since $X^0_{\min}$ contains $(\mu^{\hU}_{(K,\omega)})^{-1}(\omega_{\min}+\delta)$ and is $\hU$-invariant, it follows from (i) and Lemma \ref{lem:x0min} that %by Corollary \ref{cor2} 
$$\hU(\mu^{\hU}_{(K,\omega)})^{-1}(\omega_{\min}+\delta) \subseteq X^0_{\min} \setminus UZ_{\min}(X)$$
 and the proof is complete.
\qed

\medskip

A modification of the proof of the $\hU$-slice theorem gives us a slice theorem for $H$.

\begin{proposition} \label{sliceH} [\textbf{Slice theorem for $H$}]\label{sliceH} 
Suppose that the Hamiltonian action of $H = U \rtimes R$ on $X$  with respect to the $H$-equivariant K\"ahler structure $\Omega$ on $X$ (with $(K,\omega) \in \Omega$) satisfies (strong) stability=semistability in the sense of Definition \ref{weaks=ss}. Then there exists $\delta_0 > 0$ such that %if $\weight_{\min} < 0 < \weight_{\min} + \delta$ then
%If $\delta_0 > 0$ is small enough then
 for $0<\delta <\delta_0$ we have 
\begin{enumerate}[(i)]
\item if $h \in H$ and $x\in (\mu^{H}_{(K,\omega)})^{-1}(\omega_{\min}+\delta)$ then 
\[hx \in (\mu^{H}_{(K,\omega)})^{-1}(\omega_{\min}+\delta) \text{ if and only if } h \in Q;\]
\item $\omega_{\min}+\delta$ is a regular value of $\mu^{H}_{(K,\omega)}$, and $H(\mu^{H}_{(K,\omega)})^{-1}(\omega_{\min}+\delta)$ is an open subset of $X^0_{\min} \setminus UZ_{\min}(X)$; 
\item $H(\mu^{H}_{(K,\omega)})^{-1}(\omega_{\min}+\delta)$ is diffeomorphic to the quotient of the product $H \times (\mu^{H}_{(K,\omega)})^{-1}(\omega_{\min}+\delta)$ by the diagonal action of $Q$ acting by right multiplication on $H$ and by the given left action on $(\mu^{H}_{(K,\omega)})^{-1}(\omega_{\min}+\delta).$
\end{enumerate}
\end{proposition}

\proof 
Recall that $R = Q_{\CC} = Q \exp(i\lieq)$ where $\lieq$ is the Lie algebra of $Q$ and $\l(\RR^{>0}) = R \cap \exp(i \lieq)$ and $\l(S^1)  = Q \cap \l(\CC^*)$, and the multiplication map $Q \times \exp(i\lieq) \to R$ is a diffeomorphism. Note that $\exp(i\lieq)$ is not in general a subgroup of $R$.

Let  $\lieh_\RR = \lieu \oplus i \lieq \subseteq \lieh$ and $H_{\RR} =  U \exp(i \lieq ) \subseteq H$. Then $H = Q H_{\RR}$ and $Q \cap  H_{\RR} = \{1\}$. Also $(\mu^{H}_{(K,\omega)})^{-1}(\omega_{\min}+\delta)$ is invariant under the action of $Q$. So for (i) it suffices to show that if $h \in H_{\RR}$ and $x\in (\mu^{H}_{(K,\omega)})^{-1}(\omega_{\min}+\delta)$ and 
$hx \in (\mu^{H}_{(K,\omega)})^{-1}(\omega_{\min}+\delta)$ then $h = 1$. If $h \in U$ then $h=1$ by Proposition \ref{sliceU}, so we can assume that $h \in H_{\RR} \setminus U$.
Moreover since every element of $\lieq$ lies in the Lie algebra of a maximal torus of $Q$, and $\l(S^1)$ is contained in every maximal torus of $Q$, it is enough to prove (i) when $Q$ is a two-dimensional compact torus $\l(S^1) \times \l'(S^1)$ so $H \cong U \rtimes (\CC^*)^2$.
By a slight abuse of notation, we can allow $\l'$ to be a rational 1-parameter subgroup of $H$ represented by an element of $\liek$ of norm 1.

Then by a modification of the proof of Lemma \ref{borel}, $h$ lies in the image of the exponential map $\exp: \lieh_\RR \to H_{\RR}$, so we can assume that $h = \exp( sb+\l_*(t) + \l'_*(t'))$ for some $b\in \lieu$ and $s,t,t' \in \RR$ where $b$ and $\l'$ have norm 1. By exchanging the roles of $x$ and $\hat{u}x$ and replacing $\hat{u}$ with its inverse, if necessary, we can assume that $t \ge 0$. Then by replacing $s$ and $b$ with $-s$ and $-b$, if necessary, we can assume that $s\ge 0$.

Thus let $\hat{b} = sb + \l_*(t) + \l'_*(t') \in {\lieh}_{\RR}$  with $||b||=1$ and $s,t \in [0,\infty)$ and $t' \in \RR$, and define $f_{\hat{b}}: X \to \RR$ by $f_{\hat{b}}(x) = \mathrm{Im} \, \mu^{H}_{(K,\omega)}(x).\hat{b}$. We want to show that there exists $\delta_0 > 0$, independent of the choice of $\hat{b}$, such that if $0<\delta \le \delta_0$ then
\[ \mathrm{d}f_{\hat{b}}(y)(\hat{b}_y)                                                     %\mu^{U}_{(K,\omega)}(\exp(b)x)\cdot a_y 
>0\]
when $y \in (\mu^{H}_{(K,\omega)})^{-1}(\weight_{\min} + \delta)$ and $\hat{b} \neq 0$. Then (i) will follow from 
%By the $U$-slice theorem (Proposition \ref{sliceU}) and
 the first part of the gradient trick (Lemma \ref{gradient}).       %, it will be enough to show that there exists $\delta_0 > 0$ such that if $0<\delta <\delta_0$ and

So suppose $b\in \lieu$  with $||b||=1$ and $s,t \in [0,\infty)$ and $t' \in \RR$. Since $i_K(\l_*(1))_y=-\l_*(1)_y$ and $i_K(\l'_*(1))_y=-\l'_*(1)_y$
we have that $\mathrm{d}f_{\hat{b}}(y)(\hat{b}_y) $ is given (up to a factor of 2) by
\begin{multline} %\label{expanded}
 \mathrm{Re} (2i \, \mathrm{d}\mu^{\hU}_{(K,\omega)}(y)(sb_y+\l_*(t)_y + \l'_*(t')_y)\cdot (sb_y+\l_*(t)_y+ \l'_*(t')_y))\\= ||sb_y+t\l_*(1)_y + t' \l'_*(1)_y||_\eta^2-\mathrm{Re} \, \eta_{\omega,y}(i_K(sb+t\l_*(1) + \l'_*(t'))_y,sb_y+t\l_*(1)_y +t' \l'_*(1)_y)\\
=s^2||b_y||^2_\eta+2||t \l_*(1)_y +t' \l'_*(1)_y||_\eta^2+s\mathrm{Re}(3\eta_{\omega,y}(t\l_*(1)_y +t' \l'_*(1)_y,b_y)-\eta_{\omega,y}(i_K(b)_y,t\l_*(1)_y +t' \l'_*(1)_y))-s^2 \mathrm{Re} \, \eta_{\omega,y}(i_K(b)_y,b_y).
\end{multline}
In the proof of Proposition \ref{sliceUhat} it was  shown that, for $\delta_0 > 0$ sufficiently small, there exist $\a_0,\b_0,\gamma_0>0$, independent of the choice of $sb + \l_*(t) \in \hat{\lieu}_{\RR}$  with $||b||=1$ and $s,t \in [0,\infty)$, such that if $0 \le \delta \leq \delta_0$ then 
\[ s^2||b_y||^2_\eta+s\mathrm{Re}(3\eta_{\omega,y}(t\l_*(1)_y,b_y)-\eta_{\omega,y}(i_K(b)_y,t\l_*(1)_y))-s^2 \mathrm{Re} \, \eta_{\omega,y}(i_K(b)_y,b_y)
%\mathrm{d}f_{\hat{b}}(y)(\hat{b}_y)                                                     %\mu^{U}_{(K,\omega)}(\exp(b)x)\cdot a_y 
\ge   \a_0 s^2   +  \b_0 st \sqrt{\delta}  \]
and $2||t \l_*(1)_y||_\eta^2 \ge   \gamma_0 t^2 \delta $
whenever $y \in (\mu^{\hU}_{(K,\omega)})^{-1}(\weight_{\min} + \delta)$. 
By Lemma \ref{claim} we can also assume that 
$$|\eta_{(K,\omega)}(\l_*(1)_y,\l'_*(1)_y)|  \le E_2 || \l_*(1)_y||_\eta  ||\l'_*(1)_y||_\eta $$
where $E_2 \in [0,1)$ is independent of the choice of $sb + \l_*(t) \in \hat{\lieu}_{\RR}$  with $||b||=1$ and $s,t \in [0,\infty)$.
%Thus  to show that $\mathrm{d}f_{\hat{b}}(y)(\hat{b}_y)                                                     >0$ it suffices to 
Moreover the strong version of semistability=stability tells us that $ \l'_*(1)_y \neq 0$ when $y \in (\mu^{H}_{(K,\omega)})^{-1}(\omega_{\min})$, so if $\delta_0>0$ is sufficiently small then  $|| \l'_*(1)_y ||_\eta$ is bounded away from zero on $(\mu^{H}_{(K,\omega)})^{-1}(-\infty,\omega_{\min}+\delta_0)$, say 
\[|| \l'_*(1)_y ||_\eta>\zeta > 0 \text{ when } y\in (\mu^{H}_{(K,\omega)})^{-1}(-\infty,\omega_{\min}+\delta_0).\] 

Since $\l(S^1)$ grades $U$, we can write $b = [\l_*(1),b']$ for some $b' \in \lieu$. As $\eta_{\omega,y}$ is $\l(S^1)$-invariant and $\l(S^1)$  commutes with $\l'(\CC^*)$, it follows that  when $y \in Z_{\min}(X)$
we have  
$$\eta_{\omega,y}( \l'_*(1)_y,b_y) =    \eta_{\omega,y}( \l'_*(1)_y,[\l_*(1)_y,b'_y])   = \eta_{\omega,y}( [\l'_*(1)_y, \l_*(1)_y],b'_y)        = 0.$$
Thus, by compactness again, given $\epsilon>0$ we can choose $\delta_0>0$ small enough so that 
$$ |  \mathrm{Re} \, \eta_{\omega,y} (\l'_*(1)_y,b_y) | \le \epsilon $$
when $y \in (\mu^{H}_{(K,\omega)})^{-1}(-\infty,\omega_{\min} + \delta_0)$. 
Since $i_K(b)_y = 0$ when $y \in Z_{\min}(X)$ we can also assume that $$ |  \mathrm{Re} \, \eta_{\omega,y} (i_K(b)_y,\l'_*(1)_y) | \le \epsilon $$
when $y \in (\mu^{H}_{(K,\omega)})^{-1}(-\infty,\omega_{\min} + \delta_0)$. 
Therefore
\begin{multline} \mathrm{d}f_{\hat{b}}(y)(\hat{b}_y) \ge  \a_0 s^2   +  \b_0 st \sqrt{\delta} + 2||t \l_*(1)_y +t' \l'_*(1)_y||_\eta^2+st'\mathrm{Re}(3\eta_{\omega,y}( \l'_*(1)_y,b_y)-\eta_{\omega,y}(i_K(b)_y, \l'_*(1)_y))\\
\ge  \a_0 s^2   +  \b_0 st \sqrt{\delta} + 2t^2|| \l_*(1)_y||_\eta^2  + 2(t')^2|| \l'_*(1)_y||_\eta^2 - 4tt' E_2 || \l_*(1)_y||_\eta|| \l'_*(1)_y||_\eta
-4st'\epsilon \\ 
\ge  \left( \sqrt{\a_0} s  -  \frac{2\epsilon}{\sqrt{\a}_0} t' \right) ^2   +  \b_0 st \sqrt{\delta} + 2 \left(t || \l_*(1)_y||_\eta  - {E_2} t'|| \l'_*(1)_y||_\eta \right)^2 + 2(t')^2 \left(\zeta^2(1 - E_2^2) - \frac{4\epsilon^2}{\a_0} \right). 
\end{multline} 
Here $\a_0, \b_0$ and $\zeta$ are strictly positive, and so is $\zeta^2(1 - E_2^2)$ since $E_2 \in [0,1)$. Moreover we can choose $\epsilon>0$ as small as we wish, so we can assume that $\zeta^2(1 - E_2^2) > 4\epsilon^2 /\a_0$. Also $s$ and $t$ are non-negative and hence %$|| \l_*(1)_y||_\eta^2 \ge   \gamma_0 \delta/2 $ where $\gamma_0 >0$. So 
 if $\delta_0>0$ is chosen sufficiently small and 
$y \in (\mu^{H}_{(K,\omega)})^{-1}(\weight_{\min} + \delta)$ for $\delta \in (0,\delta_0)$, then $ \mathrm{d}f_{\hat{b}}(y)(\hat{b}_y)                                                     %\mu^{U}_{(K,\omega)}(\exp(b)x)\cdot a_y 
\ge 0$, and if $ \mathrm{d}f_{\hat{b}}(y)(\hat{b}_y)                                                     %\mu^{U}_{(K,\omega)}(\exp(b)x)\cdot a_y 
= 0$, % when $y \in (\mu^{H}_{(K,\omega)})^{-1}(\weight_{\min} + \delta)$ for $\delta \in (0,\delta_0)$,
 then $s=t=t'=0$
so
$\hat{b} = 0$. 
(i) now follows from the gradient trick, and then (ii) and (iii) follow as in the proof of Proposition \ref{sliceUhat}.
\qed

\subsection{Proof of Theorem \ref{thm:mainA}}\label{Hactions}
Suppose now as in Theorem \ref{thm:mainA} that a nonsingular connected projective variety $X$ has a well adapted linear action of a linear algebraic group $H=U\rtimes R$ with internally graded unipotent radical $U$, satisfying H-stability=H-semistability in the strong sense (see Definitions \ref{def:s=ss} and \ref{cond starKH} and Remark \ref{remark:ss=s}). 
As in Theorem \ref{mainthmextended}, we will assume that $U Z_{\min}(X)$ is not dense in $X$; otherwise $Z_{\min}(X)$ plays the role of the GIT quotients $X \env U$ and $X\env \hU$ via the $\hU$-invariant morphism $p: X^0_{\min} \to Z_{\min}(X)$,  and $Z_{\min}(X)\env (R/\l(\CC^*))$ plays the role of $X \env H$, so we are reduced to the familiar reductive situation.

Let $\Omega$ be an $H$-equivariant K\"ahler structure  on $X$
constructed using Fubini--Study K\"ahler forms %on $Y=\PP(\CC^{n+1})$ under the action of $G = \GL(n+1)$
 as in 
Example \ref{remarkFS}, 
let $(K,\omega) \in \Omega$ and let $\newmmxh: \Omega \times X \to {\lieh}^*$ be an {$\Omega$-moment map} for the $H$-action on $X$ in the sense of Definition \ref{defn:omegaH}. Let $\mu^{H}_{(K,\omega)}: X \to {\lieh}^*$ be the   restriction %$\mu^H_{(K,\omega)}: X \to \lieh^*$
 of $\newmmxh:\Omega \times X \to {\lieh}^*$ to $X$ identified with $\{(K,\omega)\} \times X$. 
 The well-adaptedness of the linearisation means that if $\delta_0 >0$ is chosen as in the slice theorems of the previous section (Propositions \ref{sliceU}, \ref{sliceUhat}, \ref{sliceH}) then we can assume that $$\weight_{\min} < 0 < \weight_{\min} + \delta_0.$$
% We also assume that stable points exist for the $\hU$-action in the sense of Definition \ref{cond starK2}. 
We will also assume that $(\mu^{H}_{(K,\omega)})^{-1}(0)$ is nonempty.

 We have the following version of diagram \eqref{uhatmomentmapH}:
 \begin{equation}\label{uhatmomentmapH2}
\xymatrix{\Omega \times X \ar[r]^-{\newmmxg} \ar[rd]^{\newmmxh}  &  \lieg^*=\liek^* \oplus i\liek^* \ar[d]^{p_1^*} & \\
& \mathfrak{h}^* %\Lie(H)^*
=\hat{\lieu}^* \oplus \bar{\mathfrak{r}}^* 
 \ar[d]^{p_2^*} \ar[rd]^{\bar{p}_3^*} & \\
& \hat{\lieu}^* & \bar{\mathfrak{r}}^*= \bar{\mathfrak{q}}^* \oplus i \bar{\mathfrak{q}}^*}
\end{equation}
where $\bar{\mathfrak{r}}$ is the Lie algebra of $\bar{R} = R/\l(\CC^*)$ identified with a sub-Lie algebra of the Lie algebra $\mathfrak{r}$ of $R = Q_\CC$, and $\bar{\mathfrak{q}}$ is the Lie algebra of $Q/\l(S^1)$ identified with a sub-Lie algebra of the Lie algebra $\mathfrak{q}$ of the maximal compact subgroup $Q=K\cap H$ of $R$.
Let $\mu^{\bar{R}}_{(K,\omega)}: X \to {\bar{\mathfrak{r}}}^*$ be the composition of $\mu^{H}_{(K,\omega)}: X \to {\lieh}^*$ with the restriction map from ${\lieh}^*$ to $ {\bar{\mathfrak{r}}}^*$.
Then
\[(\mu^{H}_{(K,\omega)})^{-1}(0)=(\mu^{\hU}_{(K,\omega)})^{-1}(0) \cap (\mu^{{\bar{R}}}_{(K,\omega)})^{-1}(0)\]
where both sides are $Q$-invariant, and hence 
\begin{equation}\label{momentmaps}
(\mu^{H}_{(K,\omega)})^{-1}(0)/Q=(((\mu^{\hU}_{(K,\omega)})^{-1}(0) \cap (\mu^{{\bar{R}}}_{(K,\omega)})^{-1}(0)
)/\l(S^1))/(Q/\l(S^1)).
\end{equation}
 We now have the following result  which, in combination with Proposition \ref{sliceH}, completes the proof of Theorem \ref{thm:mainA}.
 
\begin{theorem}\label{quotientshomeoH} Under the assumptions above,  $H (\mu_{(K,\omega)}^{H})^{-1}(0)$ coincides with the $H$-stable locus $X^{s,H}_{\min+}$, and  the inclusion of $(\mu_{(K,\omega)}^{H})^{-1}(0)$
    in $X^{s,H}_{\min+}$ induces a diffeomorphism of orbifolds  from $(\mu_{(K,\omega)}^{H})^{-1}(0)/Q$ %\mu_{\l(S^1)}^{-1}(-\infty,  \delta)
     to the GIT quotient $X \env H$.
\end{theorem}
\proof
Recall from Theorem \ref{mainthmextended} that 
the projective variety $X\env H$ is a geometric quotient of $X^{s,{{H}}}_{\min+}$ % =X^{ss,{{H}}}_{\min+}$
 of $X$ by the action of $H$, the stabiliser  $\Stab_H(x)$ is finite for all $x \in X^{s,{{H}}}_{\min+}$
and 
$$ X^{nss,H} = X^{ss,fg,H} = X^{\rms,H} = X^{s,H}_{\min+} =X^{ss,H}_{\min+} = p^{-1}(Z_{\min}(X)^{s,\bar{R}}) \setminus UZ_{\min}(X),$$
where $p:X^0_{\min} \to Z_{\min}(X)$ is defined as in Definition \ref{def:p} and $Z_{\min}(X)^{s,\bar{R}}$ is defined as 
 in Definition \ref{def:s=ss}. Here $p^{-1}(Z_{\min}(X)^{s,\bar{R}}) $ and $ UZ_{\min}(X)$ are $H$-invariant open subsets of $X$. By the relationship between reductive GIT and symplectic reduction (see \eqref{reductivemmap}) we have
 $$ Z_{\min}(X) \cap (\mu^{{\bar{R}}}_{(K,\omega)})^{-1}(0) \subseteq Z_{\min}(X)^{ss,\bar{R}} = Z_{\min}(X)^{s,\bar{R}} \subseteq p^{-1}(Z_{\min}(X)^{s,\bar{R}}),$$
 and hence 
 when $\delta_0 >0$ is chosen sufficiently small satisfying the conditions of Proposition \ref{sliceH} and $\weight_{\min} < 0 < \weight_{\min} + \delta_0$, then $(\mu_{(K,\omega)}^{H})^{-1}(0)$ is contained in $p^{-1}(Z_{\min}(X)^{s,\bar{R}})$. Moreover $(\mu^{H}_{(K,\omega)})^{-1}(0) \subseteq (\mu^{\hU}_{(K,\omega)})^{-1}(0)$, which is disjoint from $UZ_{\min}(X)$ by Proposition \ref{sliceUhat}. 
 Hence $(\mu_{(K,\omega)}^{H})^{-1}(0)$ is contained in $p^{-1}(Z_{\min}(X)^{s,\bar{R}})  \setminus UZ_{\min}(X),$
  which coincides  %. Thus by Proposition \ref{sliceH} $(\mu_{(K,\omega)}^{H})^{-1}(0)$ is contained in
 with the $H$-invariant open subset $X^{s,H}_{\min+}$  %=p^{-1}(Z_{\min}(X)^{s,R/\l(\GG_m)}) \setminus UZ_{\min}(X)$
  of $X$. So we have an inclusion of open subsets
 $$H (\mu_{(K,\omega)}^{H})^{-1}(0) \subseteq X^{s,H}_{\min+}$$
 of $X$, inducing an open inclusion $H (\mu_{(K,\omega)}^{H})^{-1}(0)/H \subseteq X^{s,H}_{\min+}/H$ where $X^{s,H}_{\min+}/H = X \env H$ is a projective variety.
 By Proposition \ref{sliceH} (iii)   the inclusion of 
$(\mu_{(K,\omega)}^{H})^{-1}(0)$ in $W_H = H (\mu_{(K,\omega)}^{H})^{-1}(0)$ induces a diffeomorphism of orbifolds $(\mu_{(K,\omega)}^{H})^{-1}(0)/Q %\mu_{\l(S^1)}^{-1}(-\infty,  \delta)
     \to W_H/H$, so $W_H/H$ is compact. The complement of $X^{s,H}_{\min+}$ in $X$ has real codimension at least two and so $X_{\min+}^{s,H}$ and its quotient $X^{s,H}_{\min+}/H$ are connected, and  $W_H/H$ is nonempty since by assumption $( \mu_{(K,\omega)}^{H})^{-1}(0) \neq \emptyset$. So we must have $W_H/H = X^{s,H}_{\min+}/H = X \env H$ and hence $H (\mu_{(K,\omega)}^{H})^{-1}(0) = X^{s,H}_{\min+}$, and the result follows. \qed

\section{Morse inequalities and Betti numbers}\label{sec:betti}

We recall some basic facts on equivariantly perfect  stratifications, following \cite{francesthesis}. For a topological space $Y$ endowed with a $G$-action let
\[P_t(Y)=\sum_{i\ge 0} t^i \dim H^i(Y,\QQ) \text{ and } P_t^G(Y)=\sum_{i\ge 0} t^i \dim H^i_G(Y,\QQ)\]
denote its rational Poincari\'e series and equivariant Poincar\'e series.  
\subsection{Morse inequalities}
Given a smooth stratification $X=\cup_{\beta \in B}S_\b$ of the manifold $X$ in the sense of \cite{francesthesis}, we can build up the cohomology of $X$ inductively from the cohomology of the strata. There is a Thom-Gysin sequence  relating the cohomology of the stratum $S_\beta$ and of the open subsets $\cup_{\g < \b}S_\g$ and $\cup_{\g \le \b} S_\g$ of $X$. These give us the Morse inequalities which can be expressed as follows. Assume that each component of any stratum $S_\b$ has the same codimension, say $d(\beta)$, in $X$. Then 
\[\sum_{\beta \in B}t^{d(\beta)}P_t(S_\b)-P_t(X)=(1+t)R(t)\]
where $R(t)$ is a series with non-negative integer coefficients. 
The stratification is called \emph{perfect} if the Morse inequalities are equalities, that is, if 
\[P_t(X)=\sum_{\beta \in B}t^{d(\beta)}P_t(S_\b).\]

If the space $X$ is acted on by a topological group $G$ then the equivariant cohomology $H_G^*(X,\QQ)$ is define as the ordinary cohomology of the topological quotient $EG \times_G X$ where $EG$ is a contractible space equipped with a free $G$-action (i.e. the total space of the universal $G$-bundle $EG \to BG$):
\[H_G^*(X;\QQ)=H^*(EG \times_G X;\QQ).\]
%Let $K$ denote the maximal compact subgroup of $G$. 
For any smooth $G$-invariant  stratification $X=\cup_{\beta \in B}S_\b$ we obtain equivariant Morse inequalities 
\[\sum_{\beta \in B}t^{d(\beta)}P^G_t(S_\b)-P^G_t(X)=(1+t)R_G(t)\]
where again $R_G(t)$ has non-negative coefficients. % and $P_t^K$ denotes the equivariant Poincar\'e series. 
The stratification is called \emph{equivariantly perfect} if these are equalities; that is, $R_G(t)=0$. 

In \cite{francesthesis} it is shown that  the normsquare of a moment map $f=||\mu||^2$ for a Hamiltonian action of a compact group $K$ on a compact symplectic manifold $X$ is equivariantly perfect. 
$K$ is homotopy equivalent to its complexification $G=K_\CC$ so $G$-equivariant cohomology is the same as $K$-equivariant cohomology for any $G$-space. In the GIT set-up the equivariant perfection of the normsquare of $\mu$ expresses the equivariant Betti numbers of the semistable locus $X^{ss}$ inductively in terms of the equivariant Betti numbers of the unstable strata and those of $X$ itself. When $X^s=X^{ss}$ holds, the action of $G$ on $X^{ss}$ is rationally free (finite stabilisers) and therefore 
\[H_K^*(X^{ss};\QQ)=H^*(X^{ss}/G;\QQ)=H^*(X/\!/G;\QQ)\]
holds which provides formulas for the Betti numbers of the GIT quotient. 

We will use the following criterion due to Atiyah and Bott for a stratification to be equivariantly perfect. 
\begin{lemma}[\cite{atiyahbott} 1.4, \cite{francesthesis} 2.18]\label{equivariantperfect} %Let $G$ be a complex algebraic group with maximal compact subgroup $K$. 
Let $X=\cup_{\beta\in B}S_\b$ be a smooth $G$-invariant stratification of $X$ such that for each $\b \in B$ the equivariant Euler class of the normal bundle to $S_\b$ in $X$ is not a zero divisor in $H^*_G(S_\b,\QQ)$. Then the stratification is equivariantly perfect over $\QQ$. 

\end{lemma}
   
\subsection{Betti numbers for $\hU$ actions} Let $X$ be a smooth projective variety endowed with a well-adapted action of $\hU$ for which semistability coincides with stability in the sense of Definition \ref{cond star}. We aim to prove that the simple stratification 
\[X_{\min}^{0}=X^{s,\hU} \cup \hU Z_{\min}(X)\]
is $\hU$-equivariantly perfect. We start with some technical lemmas. Recall that $p:X^0_{\min} \to Z_{\min}(X)$ sends $x \in X^0_{\min}$ to $\lim_{t\to 0} \l(t)x$.

\begin{lemma} \label{lemmaB} Suppose that $\hU$-semistability coincides with $\hU$-stability in the sense of Definition \ref{cond star} for the action of $\hU$ on $X$. Then 
\begin{enumerate}[(i)]
\item the $U$-sweep $U  Z_{\min}(X)=\hat{U} Z_{\min}(X)$ is a closed subvariety of $X_{\min}^0$; 
\item the equivariant Euler class of the normal bundle to $UZ_{\min}(X)$ in $X_{\min}^0$ is not a zero divisor in $H_{S^1}^*(UZ_{\min}(X))$. 
\end{enumerate}
\end{lemma}

\proof (i) follows from Proposition \ref{sliceU}, but we give another argument here which we also use in proving the second statement. It is well known (going back to Bialynicki-Birula \cite{BB} and beyond) that the restriction to $X_{\min}^0 \setminus  Z_{\min}(X)$ of the $\GG_m$-invariant morphism $p:  X_{\min}^0 \to  Z_{\min}(X)$ factors through a projective morphism from a geometric quotient $(X_{\min}^0 \setminus  Z_{\min}(X))/\GG_m$ to $ Z_{\min}(X)$, and the fibres of $p$ can be identified with the affine cone associated to the fibre of the morphism from this geometric quotient to $Z_{\min}(X)$; when $X$ is nonsingular then the fibres of $p$ are affine spaces and the fibres of the geometric quotient $(X_{\min}^0 \setminus  Z_{\min}(X))/\GG_m$ over $Z_{\min}(X)$ are weighted projective spaces. 

Lemma \ref{lemmaA} tells us that the fibre over $x\in  Z_{\min}(X)$ of the restriction of $p$ to $UZ_{\min}(X)$ is isomorphic to $U/\stab_U(x)=U$. This is an affine space of dimension $\dim U$ on which $\GG_m$ acts with strictly positive weights, so the induced morphism
$$ (U  Z_{\min}(X) \setminus  Z_{\min}(X))/\GG_m  \to  Z_{\min}(X)
$$
has fibres which are  weighted projective spaces. In particular it follows that the embedding $(U  Z_{\min}(X) \setminus  Z_{\min}(X))/\GG_m  \to ( X_{\min}^0 \setminus  Z_{\min}(X))/\GG_m$ is closed, and hence so is the corresponding embedding of affine cones $U  Z_{\min}(X) \to  X_{\min}^0$. 

(ii) We want to show that the equivariant Euler class of the normal bundle $N$ to $UZ_{\min}(X)$ in $X_{\min}^0$ is not a zero-divisor in $H_{S^1}^*(UZ_{\min}(X),\QQ)$. We saw in the proof of (i) that $UZ_{\min}(X)$ is an affine cone over $Z_{\min}(X)$ and therefore the latter is a deformation retract of the former, giving us the isomorphism 
\[H_{S^1}^*(UZ_{\min}(X),\QQ) \simeq H_{S^1}^*(Z_{\min}(X),\QQ).\] 
Under this isomorphism the equivariant Euler class of $N$ is identified with the equivariant Euler class of its restriction to $Z_{\min}(X)$. However, $Z_{\min}(X)$ is a union of fixed points components of $S^1$ and therefore the $S^1$-weights to the normal directions are all nonzero, proving that the $S^1$-equivariant Euler class, which is the product of these normal weights, is nonzero. 
\qed

Lemma \ref{equivariantperfect} and Lemma \ref{lemmaB} give us  

\begin{corollary}
The stratification $X_{\min}^0=X^{s,\hU} \cup UZ_{\min}(X)$ is equivariantly perfect over $\QQ$.
\end{corollary}
We are ready to prove 

\begin{theorem}\label{thm:main}  Let $X$ be a smooth projective variety endowed with a well-adapted action of $\hU$ such that semistability coincides with stability in the sense of Definition \ref{cond star} for the action of $\hU$ on $X$. Then 
\[P_t(X/\!/\hU)=P_t(Z_{\min}(X))\frac{1-t^{2d}}{1-t^2}.\]
where $d=\dim(X)-\dim(U)-\dim(Z_{\min}(X))$. 
\end{theorem}

\proof We apply the Morse equalities for the $S^1$-equivariant Morse-stratification $X_{\min}^0=X^{s,\hU} \cup UZ_{\min}(X)$, where the codimension of the closed stratum $UZ_{\min}(X)$ in $X_{\min}^0$ is $d=\dim(X)-\dim(U)-\dim(Z_{\min}(X))$. Both strata retract equivariantly onto $Z_{\min}(X)$, and hence we get
\[P_t^{S^1}(X^{s,\hU})=P_t^{S^1}(Z_{\min}(X))(1-t^{2d}).\]
Finally we use that $P_t^{S^1}(Y)=P_t(Y)P_t(BS^1)=P_t(Y)\cdot \frac{1}{1-t^2}$ holds for any $Y$ with a trivial $S^1$ action to get the final formula. 
\qed

\subsection{Stratifications and Betti numbers for $H$ actions} %Let $(X,L,H,\hat{U},\chi)$ be
Suppose that we have a well-adapted linear action of $H=U \rtimes R$ with internally graded unipotent radical $U$ and grading one-parameter subgroup $\l:\CC^* \to Z(R)$ on a nonsingular connected projective variety $X$ satisfying the strong version of $H$-stability=$H$-semistability in the sense of Definition \ref{weaks=ss}. Let $\bar{R}=R/\l(\CC^*)$ as in \S \ref{Hactions}. The GIT quotient of $X$ by this linear action of $H$ is given by an iterated quotient 
\[X\env H=(X\env \hU)/\!/\bar{R},\] 
%where $\l:\CC^* \to R$ is the central one-parameter subgroup of $R$ which grades the unipotent radical $U$ of $H$,
and hence one way to compute the Betti numbers of $X/\!/H$ is to apply the machinery of \cite{francesthesis} to $X/\!/\hU$ acted on by the reductive group $\bar{R}=R/\l(\CC^*)$. Another method is to generalise the simple Morse stratification of $X/\!/\hU$ described in the previous section to the $H$-action. Let $Z_{\min}(X)^{s,\bar{R}}$ and $Z_{\min}(X)^{ss,\bar{R}}$
denote the stable locus and the semistable locus for the residual $\bar{R}$ action on $Z_{\min}(X)$; here we assume that the linearisation of the $R$-action has been twisted by a character so that the action of $\l(\CC^*)$ on the line bundle restricted to $Z_{\min}(X)$ is trivial. 

\begin{lemma}\label{lemma:fibrationH} Assume that the strong version of $H$-stability=$H$-semistability holds in the sense of Definition \ref{def:s=ss} for the well-adapted $H$ action on $X$, so that $Z_{\min}(X)^{s,\bar{R}}=Z_{\min}(X)^{ss,\bar{R}}$. Then
\begin{enumerate}[(i)]
\item $p^{-1}(Z_{\min}(X)^{s,\bar{R}})\setminus UZ_{\min}(X)^{s,\bar{R}}$ is an open $H$-invariant subset of $X^{s,H}$; 
\item $p: p^{-1}(Z_{\min}(X)^{s,\bar{R}}) \to Z_{\min}(X)^{s,\bar{R}}$ is $R$-equivariant and $U$-invariant, and the induced map 
\[\bar{p}: (p^{-1}(Z_{\min}(X)^{s,\bar{R}})\setminus UZ_{\min}(X)^{s,\bar{R}}(\CC^*))/H \to Z_{\min}(X)/\!/\bar{R}\] is a weighted projective bundle; 
\item $X/\!/H=(p^{-1}(Z_{\min}(X)^{s,\bar{R}})\setminus UZ_{\min}(X)^{s,\bar{R}})/H$;
\item the $R$-equivariant Euler class of the normal bundle to $UZ^{s,\bar{R}}_{\min}$ in $p^{-1}(Z_{\min}(X)^{s,\bar{R}})$ is not a zero divisor in $H_{R}^*(UZ^{s,\bar{R}}_{\min})$. 
\end{enumerate}
\end{lemma}
\proof This follows from the proofs of Theorem \ref{mainthmextended} and Lemma \ref{lemmaB}. 
\qed

\begin{theorem}\label{thm:mainH}
Let $X$ be a smooth projective variety endowed with a well-adapted action of $H=U \rtimes R$ such that $H$-stability=$H$-semistability holds in the strong sense of Definition \ref{def:s=ss}, and $Z_{\min}(X)^{s,\bar{R}}=Z_{\min}(X)^{ss,\bar{R}}$. Then 
\[P_t(X\env H)=P_t(Z_{\min}(X)/\!/\bar{R})\frac{1-t^{2d}}{1-t^2}\]
where $d=\dim(X)-\dim(U)-\dim(Z_{\min}(X))$. 
\end{theorem}

\proof %Let us introduce the shorthand notation
By Theorem \ref{mainthmextended} 
 ${X}^{{s},H}=p^{-1}(Z_{\min}(X)^{s,\bar{R}})\setminus UZ_{\min}(X)^{s,\bar{R}}$. We have $R$-equivariant Morse equalities for the stratification $p^{-1}(Z_{\min}(X)^{s,\bar{R}})={X}^{{s},H} \cup UZ^{s,\bar{R}}_{\min}$, where the codimension of the closed stratum $UZ^{s,\bar{R}}_{\min}$ in $p^{-1}(Z_{\min}(X)^{s,\bar{R}})$ is $d=\dim(X)-\dim(U)-\dim(Z_{\min}(X))$. Both strata retract onto $Z^{s,\bar{R}}_{\min}$,  and so
\[P_t^{R}({X}^{{s},H})=P_t^{R}(Z^{s,\bar{R}}_{\min})(1-t^{2d}).\]
Note that the classifying space of $R$ decomposes as $BR=B\l(\CC^*) \times B(\bar{R})$, and thus $P_t^R(Y)=P_t^{\bar{R}}(Y)P(BS^1)=P_t^{\bar{R}}(Y)\frac{1}{1-t^2}$ holds for any $Y$ acted on by $\bar{R}$. Therefore 
\[P_t^{R}({X}^{{s},H})=P_t^{\bar{R}}(Z^{s,\bar{R}}_{\min})\frac{1-t^{2d}}{1-t^2}=P_t(Z_{\min}(X)/\!/(\bar{R}))\frac{1-t^{2d}}{1-t^2}.\]
On the other hand, by Theorem \ref{mainthmextended} , %Lemma \ref{lemmaH}
$X\env H$ is a geometric quotient of $X^{s,H}$ by $H$ acting with only finite stabilisers, and $H = U \rtimes R$ is homotopy equivalent to $R$, so 
\[P_t^{R}({X}^{{s},H})=P_t({X}^{s,H}/H)=P_t(X\env H).\]
\qed

\begin{remark} \label{surjectivity}
The proof of this theorem shows that the restriction map on equivariant cohomology
$$ H_R^*( X,\QQ) \to H^*_R({X}^{s,H},\QQ) \cong H^*(X\env H,\QQ)$$
is surjective.
\end{remark}

\section{Cohomology and intersection pairings on non-reductive quotients}\label{sec:martin}

This section gives a non-reductive version of  results of Martin \cite{SM}.

\subsection{Reductive abelianisation} We start with a brief summary of Martin's results \cite{SM}.
Suppose that $X$ is a compact symplectic manifold acted on by a compact Lie group $K$ with maximal torus $T$ in a Hamiltonian fashion, with moment map $\mu_K: X \to \liek^* = \mathrm{Lie}(K)^*$ for the action of $K$ and induced moment map $\mu_T: X \to \liet^*$ for the action of $T$. In the algebraic (or more generally K\"ahler) situation 
 the symplectic quotient $\mu_K^{-1}(0)/K$ can be identified with  the GIT (or K\"ahler) quotient $X/\!/G$ of $X$ by $G=K_\CC$ \cite{francesthesis}. Let $T$ be a maximal torus of $K$ and $T_\CC \subset G=K \otimes \CC$ the corresponding complex maximal torus in $G$. 

\begin{definition} Let $\a \in \liet^*$ be a weight of $T$, and let $\CC_\a$  denote the corresponding $1$-dimensional representation of $T$. Let $L_\a \to X/\!/T_\CC$ denote the associated bundle
\[L_\a:=\mu_T^{-1}(0)\times_T \CC_\a \to X/\!/T_{\CC},\]
whose Euler class is denoted by $e(\a) \in H^2(X/\!/T_\CC)\simeq H_T^2(X)$.  
The set of roots of $G$, that is, the set of nonzero weights of the adjoint action on $\lieg$, is denoted by $\Delta$. We fix a choice $\Delta^+ \subset \Delta$ of the set of positive roots, and denote by $\Delta^-$ the corresponding  set of negative roots.
\end{definition}
The following diagram of Martin \cite{SM} relates $X/\!/G$ and $X/\!/T_{\CC}$ through a fibration and an inclusion: 
\begin{equation}\label{diagrammartin}
\xymatrix{\mu_K^{-1}(0)/T \ar@{^{(}->}[r]^-{i} \ar[d]^{\pi} & \mu_T^{-1}(0)/T=X/\!/T_{\CC} \\
X/\!/G=\mu_K^{-1}(0)/K & } 
\end{equation}

Note that  $X/\!/G$ and $X/\!/T_{\CC}$ are symplectic manifolds, and hence possess compatible almost complex structures, unique up to homotopy. The following proposition is the main technical result of \cite{SM}.
\begin{proposition} %\begin{enumerate}
 Suppose that 0 is a regular value of the moment map $\mu_K: X \to \mathrm{Lie}(K)^*$.
\begin{enumerate}[(i)]
\item The vector bundle $\oplus_{\a \in \Delta^-}L_\a \to X/\!/T_\CC$ has a section $s$, which
is transverse to the zero section, and whose zero set is the submanifold
$\mu_K^{-1}(0)/T \subset X/\!/T_{\CC}$. Therefore the normal bundle is
\[\mathcal{N}(\mu_K^{-1}(0)/T \subset X/\!/T_{\CC})\simeq \oplus_{\a \in \Delta^-}L_{\a}|_{\mu_K^{-1}(0)/T} \]
\item Let $\mathrm{vert}(\pi) \to \mu_K^{-1}(0)/T$ denote the relative tangent bundle for $\pi$. Then 
\[\mathrm{vert}(\pi)\simeq \oplus_{\a \in \Delta^+} L_\a |_{\mu_K^{-1}(0)/T}.\]
\item There is a complex orientation of $\mu_K^{-1}(0)/T$ such that the above isomorphisms are isomorphisms of complex-oriented spaces and vector bundles, with respect to the complex orientations of $X/\!/G$ and $X/\!/T_{\CC}$ induced by their symplectic forms.
\end{enumerate}
\end{proposition}
\begin{remark}
Recall that 0 is a regular value of $\mu_K$ if and only if the stabiliser in $K$ of every $x \in \mu_K^{-1}(0)$ is finite, and in the algebraic situation this happens if and only if semistability coincides with stability for the action of $G=K_\CC$. If 0 is a regular value of $\mu_K$ it does not follow that 0 is a regular value of $\mu_T$, but it does follow that the derivative of $\mu_T$ is surjective in a neighbourhood of $\mu_K^{-1}(0)$, and so the normal bundle to 
$\mu_K^{-1}(0)/T$ in  $X/\!/T_{\CC}$ is well defined, at least as an orbi-bundle.
\end{remark}

This proposition leads to the following results relating the topology of $X/\!/G$ and $X/\!/T_{\CC}$. %We assume for simplicity that $K$ acts freely on $\mu_K^{-1}(0)$, though the results can easily be generalised to the case when 0 is a regular value of $\mu_K$.

\begin{theorem}[\textbf{Cohomology rings} \cite{SM} Theorem A] %Suppose that $K$ acts freely on $\mu_K^{-1}(0)$. 
There is a natural ring isomorphism
\[H^*(X/\!/G,\QQ)\simeq \frac{H^*(X/\!/T,\QQ)^W}{ann(e)}.\]
Here $W$ denotes the Weyl group of $G$ which acts naturally on $X/\!/T$; the class $e \in H^*(X/\!/T)^W$ is the product of all roots $e=\prod_{\a \in \Delta} e(\a)$ and $\mathrm{ann}(e) \subset H^*(X/\!/T,\QQ)^W$ is the annihilator ideal 
\[\mathrm{ann}(e)=\{c \in H^*(X/\!/T,\QQ)^W| c \cup e=0\}.\]
\end{theorem}
Diagram \eqref{diagrammartin} provides a natural concept of a lift of a cohomology class on $X/\!/G$ to a class on $X/\!/T$, 
compatible with the above isomorphism: we say that $\tilde{a}\in H^*(X/\!/T)$ is a lift of $a\in H^*(X/\!/G)$ if $\pi^*a=i^*\tilde{a}$.
\begin{theorem}[\textbf{Integration formula}, \cite{SM} Theorem B]  Suppose that $K$ acts freely on $\mu_K^{-1}(0)$. Given a cohomology class $a \in H^*(X/\!/G)$ with lift $\tilde{a} \in H^*(X/\!/T)$ , then 
\[\int_{X/\!/G}a =\frac{n_0}{|W|}\int_{X/\!/T}\tilde{a} \cup e,\]
where $|W|$ is the order of the Weyl group of $G$, and $e$ is the cohomology class defined in Theorem A. Here $n_0>0$ depends on the generic sizes of stabilisers in $G$ and in $T$ of points of $X$, and $n_0=1$ if $K$ acts freely on $\mu_K^{-1}(0)$. 
\end{theorem}

\subsection{A nonabelian localisation formula}

Suppose as before that $X$ is a compact symplectic manifold acted on by a compact Lie group $K$  in a Hamiltonian fashion, with moment map $\mu_K: X \to \mathrm{Lie}(K)^*$. 
Recall that if the action of $K$ on the zero level set $\mu_K^{-1}(0)$ is free (or has finite stabilisers) there is a surjective map 
\begin{equation} \label{kmap} \kappa: H_K^*(X;\QQ) \to H^*(\mu_K^{-1}(0)/K;\QQ)\end{equation}
 from the equivariant cohomology of $X$ to the ordinary cohomology of $\mu^{-1}_K(0)/K$ given by composing the restriction map on rational equivariant cohomology to $\mu_K^{-1}(0)$ with an isomorphism between the rational equivariant cohomology of the latter and the ordinary rational cohomology of $\mu^{-1}_K(0)/K$ \cite{francesthesis}. In \cite{jeffreykirwan}  a formula is given  for the evaluation on the fundamental class of $\mu^{-1}_K(0)/K$ of $\kappa(\eta)$ for $\eta \in H_K^*(X;\QQ) $ whose degree is the dimension of $\mu^{-1}_K(0)/K$, in terms of the fixed point locus of a maximal torus $T$ of $K$ on $X$.

\begin{theorem}[\textbf{Localisation on reductive GIT quotients}, \cite{DuisHeckman}, \cite{jeffreykirwan},\cite{guilleminkalkman},\cite{vergne}]
\label{jeffreykirwan} Suppose that $0$ is a regular value of $\mu_K$.
Given any equivariant cohomology class $\eta$ on $X$ represented by an equivariant differential form
%For any Chern polynomial $\phi(V_\mu)\in \CC[c_1(V_\mu),c_2(V_\mu),\ldots]$
$\eta(z)$ whose degree is the dimension of $\mu^{-1}_K(0)/K$, we have 
\[\int_{\mu^{-1}_K(0)/K} \kappa(\eta) =n_K\mathrm{JKRes}^\Lambda \left(e^2(z)\sum_{F \subset \mathcal{F}} \int_F \frac{i_F^* (\eta(z))}{\mathrm{Euler}^T(\mathcal{N}_F)(z)}[dz]\right)\]
In this formula $z$ is a variable in $\liet_\CC$ so that a $T$-equivariant cohomology class can be evaluated at $z$, and $e(z)=\prod_{\gamma \in \Delta^+}\gamma(z)$ is the product of the positive roots. $\mathrm{JKRes}^\Lambda$ is a residue map which depends on a choice of a cone $\Lambda \subset \liet$, $\mathcal{F}$ is the set of components of the fixed point set of the maximal torus $T$ on $X$. If $ F\in \mathcal{F}$, $i_F$ is the inclusion of $F$ in $X$ and $\mathrm{Euler}^T(\mathcal{N}_F)$ is the $T$-equivariant Euler class of the normal bundle to $F$ in $X$. Finally, $n_K$ is a rational constant depending on $K$ and the size of the stabiliser in $K$ of a generic $x \in X$. 
\end{theorem}
The precise definition of this  residue map, whose domain is a suitable class of meromorphic differential forms on $\mathrm{Lie}(T)\otimes \CC$), is given in Definition 8.5 of \cite{jeffreykirwan}. Here we give the description of Jeffrey and Kogan \cite{jeffreykogan}. 

Let $V$ be an $r$-dimensional real vector space and let $A = \{\a_1,\ldots ,\a_n\}$ be a collection of (not necessarily distinct) non-zero vectors in $V^*$. We consider $\a_i$s as linear functions on $V$ . Let $\Lambda$ be a connected component of $V\setminus \bigcup_{i=1}^n \a_i^\perp$, where $\a_i^\perp=\{v \in V| \a_i(v)=0\}$. By this choice we have for all $i$ that either $\a_i \in \Lambda^v$ or $-\a_i \in \Lambda^v$, where $\Lambda^{v}=\{\beta \in V^* |\b(v)>0, \forall v\in \Lambda\}$ is the dual cone of $\Lambda$. Let $\xi \in \Lambda$ and choose a basis $\{z_1,\ldots ,z_r\}$ of $V^*$ such that $z_1(\xi)=1$ and $z_2(\xi)=\ldots =z_r(\xi)=0$. Let $\varepsilon=\varepsilon_1z_1+\ldots+\varepsilon_rz_r \in V^*$ and let $P\in \RR[V]$ be a polynomial.
\begin{definition}{\textbf{J-K residue} \cite{jeffreykirwan,jeffreykogan,szilagyi,szenes}}  We define   \label{defn:jeffreykirwan} 
\[\mathrm{Res}_{z_1}^+=\frac{P(\bz)e^{\varepsilon(\bz)}}{\prod_{i=1}^n\a_i(\bz)}dz_1=\begin{cases} \mathrm{Res}_{z_1=\infty} \frac{P(\bz)e^{\varepsilon(\bz)}}{\prod_{i=1}^n\a_i(\bz)}dz_1 & \text{if } \varepsilon_1\ge 0 \\ 0 & \text{if } \varepsilon<0\end{cases}\]
where $z_2,\ldots, z_r$ are taken to be constants when we take the residue with respect to $z_1$. The residue map is then defined as
\[\mathrm{JKRes}^\Lambda \frac{P(\bz)e^{\varepsilon(\bz)}}{\prod_{i=1}^n\a_i(\bz)}d\bz=\frac{1}{\det[(z_i,z_j)]_{i,j=1}^r}\mathrm{Res}_{z_r}^+\left(\ldots \left(\mathrm{Res}_{z_1}^+\frac{P(\bz)e^{\varepsilon(\bz)}}{\prod_{i=1}^n\a_i(\bz)}dz_1\right)\ldots \right)dz_r\]
where $\det[(x_i,x_j)]_{i,j=1}^r$ is the Gram determinant with respect to a fixed scalar product on $V^*$.
\end{definition}

Thus, after fixing coordinates and an inner product on $V=\liet$, and a cone $\Lambda \subset \liet$, the J-K residue in Theorem \ref{jeffreykirwan} can be interpreted as an iterated residue over those components of the torus-fixed point locus whose images under the torus moment map lie in $\Lambda$.

We will not need the definition in full generality in the following examples, just the simple case when $G=\GL(k,\CC)$ with maximal compact $K=U(k)$ acting on $X=\Hom(\CC^k,\CC^n)$ (or slightly  more generally on $X=\Hom(\CC^k,TX)$). 

\begin{example}[Localisation on $\grass_k(\CC^n)$, see also \S7 of \cite{SM}]
The Grassmannian $\grass_k(\CC^n)$ can be described as the symplectic quotient of the set of complex matrices $\Hom(\CC^k,\CC^n)$ with $n$ rows and $k$ columns by the unitary group $U(k)$ where $g\in U(k)$ acts on a matrix $A \in \Hom(\CC^k,\CC^n)$ by $A \circ g^{-1}$. The moment map $\Hom(\CC^k,\CC^n) \to u(k)$ is given by 
\[\mu(A)=\bar{A}^TA-\mathrm{Id}_{n\times n}\]
if we identify the dual of the Lie algebra $u(k)$ with Hermitian matrices via the pairing $\langle H,\xi \rangle=\frac{i}{2}\mathrm{trace}(H\bar{\xi}^T )$. 

The $k$ column vectors of the matrix $A\in \Hom(\CC^k, \CC^n)$ define vectors $v_1,\ldots, v_k\in \CC^n$, and the $(i, j)$-entry of $\bar{A}^TA$ is the Hermitian inner product $(v_j,v_i)$. Hence $\mu^{-1}(0)$ consists
of the unitary $k$-frames in $\CC^n$, and taking the quotient by $U(k)$ gives the Grassmannian:
\[\grass_k(\CC^n)=\mu^{-1}(0)/U(k).\]
Let $u_1,\ldots, u_k$ denote the weights of the maximal torus $T \subset U(k)$ acting on $\CC^k$. The Weyl group of $U(k)$ is the symmetric group on $k$ elements $S_k$, the roots of $U(k)$ can be enumerated by pairs of positive integers $(i, j)$ with $1\le i, j\le k$ and $i \neq j$, and the cohomology class corresponding to the root $(i, j)$ is $z_j-z_i$. Let $V$ denote the tautological rank $k$ vector bundle over $\grass_k(\CC^n)$. There is only one fixed point, namely the origin, so Theorem \ref{jeffreykirwan} gives us (see also Proposition 7.2 in \cite{SM})
\[\int_{\grass_k(\CC^n)} \phi(V)=\frac{1}{k!}\mathrm{Res}_{\bz=0} \frac{\phi(\bz) \prod_{i \neq j} (z_i-z_j) d\bz}{z_1^n\ldots z_k^n}.\]
\end{example}

\begin{example}[Localisation on $\flag_k(\CC^n)$]\label{example:flag}
We expect a similar residue formula for Chern numbers of the tautological bundle over the complex flag variety $\flag_k(\CC^n)=\Hom^\reg(\CC^k,\CC^n)/B_k$ where $B_k$ is the upper Borel in $\GL(k,\CC)$. The symplectic quotient description of $\flag_k(\CC^n)$ is a bit more delicate though, and our starting point is Proposition 1.2 of \cite{SM} and the corresponding diagram \eqref{diagrammartin}, which connects the different quotients $\mu_T^{-1}(0)/T, \mu_K^{-1}(0)/T$ and $\mu_K^{-1}(0)/K$. 

Take again $X=\Hom(\CC^k, \CC^n)$, $K=U(k), T=(S^1)^k$ as above, then $\mu_K^{-1}(0)$ consists of the unitary $k$-frames in $\CC^n$ and a $k$-tuple $(v_1,\ldots , v_k)\in \Hom(\CC^k, \CC^n)$ lies in $\mu^{-1}_T(0)$ precisely when each $v_i$ has length $1$. Therefore 
\[\mu_K^{-1}(0)/K=\grass_k(\CC^n), \mu_K^{-1}(0)/T=\flag_k(\CC^n), \mu_T^{-1}(0)/T=(\CC\PP^{n-1})^k\]
and according to Proposition 1.2 of \cite{SM} the equivariant Euler class of $\mu_K^{-1}(0)/T$ in $\mu_T^{-1}(0)/T$ is the product of the negative roots $\prod_{1\le i<j \le k}(z_i-z_j)$ of $U(k)$. This combined with Theorem \ref{jeffreykirwan} applied to $\mu_T^{-1}(0)/T=(\CC\PP^{n-1})^k$ gives us the following result.
\begin{proposition}\label{propflag} Let $V\to \flag_k(\CC^n)$ denote the tautological rank $k$ bundle over the flag manifold. Then 
\[\int_{\flag_k(\CC^n)} \phi(V)=\mathrm{Res}_{\bz=0} \frac{\phi(\bz) \prod_{1\le i<j\le k} (z_i-z_j) d\bz}{z_1^n\ldots z_k^n}.\]
\end{proposition}
In \cite{bsz} a different proof of this proposition is given (see \cite{bsz} Proposition 5.4) which uses the Atiyah-Bott localisation instead of the  localisation theorem from \cite{jeffreykirwan}. 
\end{example}

\subsection{Nonabelian localisation for non-reductive actions}

Suppose that a nonsingular connected projective variety $X$ has a well-adapted linear action of a linear algebraic group $H=U\rtimes R$ with internally graded unipotent radical $U$, satisfying H-stability=H-semistability in the strong sense of Definition \ref{def:s=ss}. Let $\Omega$ be an $H$-equivariant K\"ahler structure  on $X$
constructed using Fubini--Study K\"ahler forms %on $Y=\PP(\CC^{n+1})$ under the action of $G = \GL(n+1)$
 as in 
Example \ref{remarkFS}, where the $H$-equivariant K\"ahler structure $\Omega \subseteq \Omega_{G,Y}$ is defined using a nonsingular projective variety $Y$ acted on by a complex reductive group $G=K_\CC \geqslant H$ with Lie algebra $\lieg=\liek \oplus i\liek$ such that $Q=K\cap H$ is a maximal compact subgroup of $R$. Let $\l:\CC^* \to R$ %$H=U \rtimes R$ be a complex linear algebraic group, with unipotent radical $U$ graded by
be a central one-parameter subgroup  of $R$ grading $U$, with $\hat{U}=U\rtimes \CC^* \leqslant H$.  

Let $(K,\omega) \in \Omega$ and let $\newmmxh: \Omega \times X \to {\lieh}^*$ be an {$\Omega$-moment map} for the $H$-action on $X$ in the sense of Definition \ref{defn:omegaH}. Let $\mu^{H}_{(K,\omega)}: X \to {\lieh}^*$ be the   restriction %$\mu^H_{(K,\omega)}: X \to \lieh^*$
 of $\newmmxh:\Omega \times X \to {\lieh}^*$ to $X$ identified with $\{(K,\omega)\} \times X$. The well-adaptedness of the linearisation means that if $\delta_0 >0$ is chosen as in the slice theorems of $\S$5.3 (Propositions \ref{sliceU}, \ref{sliceUhat}, \ref{sliceH}) then we can assume that $$\weight_{\min} < 0 < \weight_{\min} + \delta_0.$$
 We assume also that %stable points exist for the $\hU$-action in the sense of Definition \ref{cond starK2}. 
 $(\mu^{H}_{(K,\omega)})^{-1}(0) \neq \emptyset$.

Let $T^H$ be a maximal torus for $Q$ (and thus a maximal compact torus for $H$) with Lie algebra $\liet^H$, and let $U_{\max}^R$ be a maximal unipotent subgroup of $R=Q_\CC$, with Lie algebra $\lieu_{\max}^R$, which is normalised by $T^H$ and corresponds to a positive Weyl chamber $\liet_{+}^H \subseteq \liet^H$. Then $U_{\max}^H = U \rtimes U_{\max}^R$ is a maximal unipotent subgroup of $H$ which is normalised by the maximal complex torus $T^H_\CC$ of $R$ (and of $H$). Since $\l(\CC^*)$ is central in $R$ and its adjoint action on $\lieu$ has all weights strictly positive, there is a small perturbation of the point it represents in $\liet^H$ which acts on $\lieu_{\max}^H = \Lie(U_{\max}^H)$ with all weights strictly positive. This means that we can choose a maximal torus $T$ of $K$ containing $T^H$ and  a positive Weyl chamber $\liet_+ \subseteq \liet$ such that the corresponding maximal unipotent subgroup $U_{\max}$ of $G$ contains $U_{\max}^H$, the corresponding Borel subgroup $B = U_{\max} T_\CC$ contains the maximal solvable subgroup $B^H = U_{\max}^H T^R$ of $H$, and $G = KB = K \exp(i\liet) \exp(\lieu_{\max})$.

Let $\lier^+$ denote the Lie algebra of the maximal unipotent subgroup $U_{\max}^R$ of $R$, and let $\lier^-$ denote the Lie algebra of the opposite maximal unipotent subgroup of $R$. We have  
 the Cartan decomposition\[\Lie(R)=\lier^- \oplus \liet^H_\CC  \oplus \lier^+ %\text{ where } \lier^{\pm}=\oplus_{\a \in \Delta_{L}^{\pm}}e(\a)
 \]
 and the corresponding decomposition
$$\Lie(H) = \lieu \oplus \lier^- \oplus \liet^H_\CC  \oplus \lier^+.$$

Recall from Theorem \ref{quotientshomeoH} that 
 $H (\mu_{(K,\omega)}^{H})^{-1}(0)$ coincides with the $H$-stable locus $X^{s,H}_{\min+}$, and  the inclusion of $(\mu_{(K,\omega)}^{H})^{-1}(0)$ %\mu_{U}^{-1}(0)\cap \mu_{\l(S^1)}^{-1}(-\infty,  \delta)$
    in $X^{s,H}_{\min+}$ induces a %surjective local
     diffeomorphism of orbifolds  from $(\mu_{(K,\omega)}^{H})^{-1}(0)/\l(K\cap H)$ %\mu_{\l(S^1)}^{-1}(-\infty,  \delta)
     to the GIT quotient $X \env H$.
This quotient fits into a diagram analogous to \eqref{diagrammartin}: 
\begin{equation}\label{diagrammartin2}
\xymatrix{(\mu_{(K,\omega)}^{H})^{-1}(0)/T^H  \ar@{^{(}->}[r]^-{i} \ar[d]^{\pi} & (\mu_{(K,\omega)}^{T^H})^{-1}(0)/T^H=X/\!/T^H_\CC  \\ 
%\mu_H^{-1}(0)/Q\cong
 X/\!/H}
\end{equation}
where $\pi: (\mu_{(K,\omega)}^{H})^{-1}(0)/T^H \to X\env H$ is the composition of the natural map $(\mu_{(K,\omega)}^{H})^{-1}(0)/T^H \to (\mu_{(K,\omega)}^{H})^{-1}(0)/H$ with the  diffeomorphism of orbifolds   $(\mu_{(K,\omega)}^{H})^{-1}(0)/\l(K\cap H) \to X \env H$ induced by  the inclusion of $(\mu_{(K,\omega)}^{H})^{-1}(0)$ %\mu_{U}^{-1}(0)\cap \mu_{\l(S^1)}^{-1}(-\infty,  \delta)$
    in $X^{s,H}_{\min+}$.

The description of the relationships between (orbifold) tangent bundles is just as in Martin's paper \cite{SM}, using the line bundles corresponding to the weights for the action of $T^H$ on $\Lie(H)/\liet^H_\CC$.
\begin{definition}\label{def:bundles}
% \begin{enumerate}
%\item

(i) For a weight $\a$ of $T_\CC^H$ let $\CC_\a$  denote the corresponding $1$-dimensional complex representation of $T^H$ and let 
\[L_\a:=(\mu_{(K,\omega)}^{T^H})^{-1}(0)\times_{T^H} \CC_\a \to X/\!/T^H_\CC,\]
denote the associated complex line bundle whose Euler class is denoted by $e(\a) \in H^2(X/\!/T^H_{\CC})\simeq H_{T^H}^2(X)$.  
%\item
(ii) For a complex $T^H_\CC$-module $\mathfrak{a}$  let  $V_{\mathfrak{a}}=(\mu_{(K,\omega)}^{T^H})^{-1}(0) \times_{T^H_\CC} \mathfrak{a} \to X/\!/T^H_{\CC}$
denote the corresponding complex vector bundle. Let $V^*_{\mathfrak{a}}=V_{\mathfrak{a}^*}$ denote the dual bundle.
%\end{enumerate}
\end{definition}
\begin{proposition}\label{propmartin} Suppose that $X$  is a nonsingular projective variety endowed with a well-adapted linear action of $H$ such that semistability coincides with stability for the action of $H$ in the strong sense of Definition \ref{def:s=ss}. 
\begin{enumerate}[(i)]
\item The vector bundle $V^*_{\lieu \oplus \lier^-} \to X/\!/T^H_\CC$ has a $C^{\infty}$ section $s$ which
is transverse to the zero section %(transverse in the sense that $s$ is a section of a real bundle over the real, oriented manifold $X/\!/T^H_\CC$)
 and whose zero set is the submanifold
$(\mu_{(K,\omega)}^{H})^{-1}(0)/T^H \subset X/\!/T^H_\CC$, so the $T^H_{\CC}$-equivariant normal bundle is
\[\mathcal{N}(i)\simeq V^*_{\lieu \oplus \lier^-}. \]
\item Let $\mathrm{vert}(\pi) \to (\mu_{(K,\omega)}^{H})^{-1}(0)/T^H$ denote the relative tangent bundle for $\pi$. Then 
\[\mathrm{vert}(\pi)\simeq V^*_{\lier^+}|_{(\mu_{(K,\omega)}^{H})^{-1}(0)/T^H}.\]
\end{enumerate}
\end{proposition}
\proof
The argument of \cite{SM} works with minor changes. For (i) and (ii) we use the diagram
\begin{equation}\label{Hmomentmap}
\xymatrix{X \ar[r]^-{\mu_{(K,\omega)}^{G}} \ar[rd]^{\mu_{(K,\omega)}^{H}} \ar[rdd]_{\mu_{(K,\omega)}^{T^H}} &  \lieg^*=\Lie(K)^* \oplus i\Lie(K)^* \ar[d]^{p^*}  \\
& \ \ \ \ \ \  \Lie(H)^*=\lieu^* \oplus \Lie(R)^*  \ar[d]^{q^*}   \\
& \Lie(T_\CC^H)^*=\liet^H \oplus i \liet^H }
\end{equation}
where $q^*$ is the projection and $p^*$ is the dual of the inclusion
\[p: \Lie(H) \hookrightarrow \liek,\ A \mapsto A-\bar{A}^T,\]
see \S\ref{Hactions}. % Let  $\hat{\lieu}=\RR^* \oplus i\RR^* \oplus \lieu$ and let $q_\CC: \hat{\lieu}^* \to \RR^* \oplus i\RR^*$ be the complex projection, 
Let $W=(q^*)^{-1}(0)$ denote the complex $T^H$-invariant subspace of $\lieg^*$ sent to $0$ by $q^*$. Note that $\mu_{(K,\omega)}^{T^H}$ is a $T^H$-equivariant map and the coadjoint action of $T^H$ on $H$ preserves the subspace $W$. Therefore $W$ is isomorphic to $ \lieu^* \oplus (\lier^-)^*$. %via the map 
%\[\lieu^* \oplus (\lier^-)^* \simeq (q^*)^{-1}(0) \ \ \ (A,B) \mapsto A\oplus (B-\bar{B}^T)\]
Then $(\mu_{(K,\omega)}^{T^H})^{-1}(0)=(\mu_{(K,\omega)}^{H})^{-1}(W)$, and the fact semistability coincides with stability for the $H$-action tells us that $0$ is a regular value for $\mu_{(K,\omega)}^{H}$, which implies that the subspace $W$ is transverse to the map $\mu_{(K,\omega)}^{H}$.  The restriction of $\mu_{(K,\omega)}^{H}$ to $(\mu_{(K,\omega)}^{T^H})^{-1}(0)$ defines an $T^H$-equivariant map $\tilde{s}: (\mu_{(K,\omega)}^{T^H})^{-1}(0)\to W$ whose quotient by $T^H$ defines a section $s$ of the associated bundle $V^*_{\lieu \oplus (\lier^-)}=(\mu_{(K,\omega)}^{T^H})^{-1}(0)\times_{T^H} W \to X/\!/T_{\CC}^H$. Since $0 \in \lieg^*$ is a regular value of $\mu_{(K,\omega)}^{H}$ it follows that $0 \in W$ is a regular value of $\tilde{s}$, and hence $s$ is transverse to the zero section, as required for (i).

(ii) follows from  Proposition 1.2 of \cite{SM}.

\qed

The next results relating the topology of $X/\!/H$ and $X/\!/T^H_\CC$ follow from Proposition \ref{propmartin} using Remark \ref{surjectivity} in the same way that Theorems A and B of \cite{SM} follow from Proposition 1.2 of \cite{SM}. In fact, just as Theorem B is proved first in \cite{SM} and then is used to prove Theorem A,  Theorem \ref{thm:integration}  follows directly from Proposition \ref{propmartin} and Theorem \ref{thm:cohomologyrings} can then be deduced
using Remark \ref{surjectivity}.
\begin{theorem}[\textbf{Cohomology rings}]\label{thm:cohomologyrings} Let $X$ be a smooth projective variety endowed with a well-adapted action of $H=U \rtimes R$ such that $H$-stability=$H$-semistability holds in the strong sense of Definition \ref{def:s=ss}. Then there is a natural ring isomorphism
\[H^*(X/\!/H,\QQ)\simeq \frac{H^*(X/\!/T^H_\CC,\QQ)^W}{ann(\mathrm{Euler}(V^*_{\lieu \oplus \lier^-\oplus \lier^+}))}.\]
Here $W$ denotes the Weyl group of $R$, which acts naturally on $X/\!/T_{\CC}^H$, while $\mathrm{Euler}(V^*_{\lieu \oplus \lier^- \oplus \lier^+}) \in H^*(X/\!/T_{\CC}^H)^W$ is the Euler class of the bundle $V_{\lieu \oplus \lier^- \oplus \lier^+}$ as in Definition \ref{def:bundles}, and  
\[\mathrm{ann}(\mathrm{Euler}(V^*_{\lieu \oplus \lier^-  \oplus \lier^+}))=\{c \in H^*(X/\!/T_{\CC}^*,\QQ)| c \cup \mathrm{Euler}(V_{\lieu \oplus \lier^- \oplus \lier^+})=0\} \subseteq H^*(X/\!/T_{\CC}^H,\QQ)\]
is the annihilator ideal. 
\end{theorem}

Again, diagram \eqref{diagrammartin2} provides a natural way to define a lift of a cohomology class on $X/\!/H$ to a class on $X/\!/T_{\CC}^H$: we say that $\tilde{a}\in H^*(X/\!/T_{\CC}^H)$ is a lift of $a\in H^*(X/\!/H)$ if $\pi^*a=i^*\tilde{a}$.

\begin{theorem}[\textbf{Integration formula}]\label{thm:integration} Let $X$ be a smooth projective variety endowed with a well-adapted action of $H=U \rtimes R$ such that $H$-stability=$H$-semistability holds in the strong sense of Definition \ref{def:s=ss}.  Given a cohomology class $a \in H^*(X/\!/H)$ with a lift $\tilde{a}\in H^*(X/\!/T_{\CC}^H)$, then 
\[\int_{X/\!/H}a=\frac{n_0}{|W|}\int_{X/\!/T^H_\CC}\tilde{a} \cup \mathrm{Euler}(V^*_{\lieu \oplus \lier^- \oplus \lier^+}),\]
where we use the notation of Theorem \ref{thm:cohomologyrings}, and $n_0 > 0$ is the ratio of the sizes of generic  stabilisers in $H$ and $T^H$ of points in $X$. 
\end{theorem}

\subsection{Localization for quotients by non-reductive groups with internally graded unipotent radicals}
Let $X$ be a smooth projective variety with a well-adapted linear action of $H=U \rtimes R$ such that $H$-stability=$H$-semistability holds in the strong sense of Definition \ref{def:s=ss}. 
We follow the notation and conventions of the previous section. We have two surjective ring homomorphisms 
\[\kappa_{T^H_\CC}: H_{T^H_\CC}^*(X) \to H^*(X/\!/T^H_\CC) \text{  and  }  \kappa_{H}: H_{R}^*(X)=H_{T^H_\CC}^*(X)^W \to H^*(X/\!/H)\]
from the equivariant cohomology of $X$ to the ordinary cohomology of the corresponding GIT quotient. The fixed points of the maximal torus $T^H_\CC$ on $X\subset \PP^n$ correspond to the weights of the $T^H_\CC$ action on $X$, and since this action is well-adapted, these weights satisfy
\[\omega_{\min}=\omega_0<0<\omega_1 <\ldots <\omega_{r}.\]

We can now combine Theorems \ref{jeffreykirwan} and \ref{thm:integration} to obtain the following result.

\begin{theorem}[\textbf{Localisation for non-reductive quotients}]
\label{jeffreykirwannonred} Let $X$ be a smooth projective variety endowed with a well-adapted linear  action of $H=U \rtimes R$ such that $H$-stability=$H$-semistability holds in the strong sense of Definition \ref{def:s=ss}.
For any %Chern polynomial $\phi(V_\mu)\in \CC[c_1(V_\mu),c_2(V_\mu),\ldots]$
equivariant cohomology class $\eta$ on $X$ represented by an equivariant differential form $\eta(z)$ 
 whose degree is the dimension of $X/\!/H$, we have (in the notation of Theorem \ref{jeffreykirwan} and Definitions \ref{defn:jeffreykirwan} and \ref{def:bundles})
\[\int_{X/\!/H} \kappa_H(\eta)=n_H\mathrm{JKRes}^\Lambda \left(\sum_{F \subset \mathcal{F}} \int_F \frac{i_F^* (\eta(z))\mathrm{Euler}(V_{\lieu \oplus \lier^- \oplus \lier^+})(z)}{\mathrm{Euler}^T(\mathcal{N}_F)(z)}[dz]\right).\]
In this formula $z$ is a variable in $\liet^H_\CC$ so that a $T^H$-equivariant cohomology class can be evaluated at $z$, while $\mathrm{JKRes}^\Lambda$ is the JK residue map which depends on a choice of a cone $\Lambda$ in the dual of the Lie algebra of the maximal compact subgroup of $T^H$, and $\calf$ is the set of connected components of the fixed point set of the $T^H$ action on $X$. Here $n_H$ depends only on $H$ and the size of the stabiliser in $H$ of a generic $x \in X$.
\end{theorem}

\begin{remark}
The general definition of the residue map $\mathrm{JKRes}^\Lambda$ is given in $\S$8 of \cite{jeffreykirwan}. It can be expressed as a multivariable residue whose domain is a class of meromorphic differential forms on the Lie algebra of $T^H$, and whose value on this domain is independent of the choice of the cone $\Lambda$. However in order to apply it to individual terms in the  formula given in Theorem \ref{jeffreykirwannonred}, a choice of cone $\Lambda$ has to be made which does not affect the residue of the whole sum. After this choice has been made, the contribution of any $F \in \calf$ such that $\mu_{T^H}(F)$ does not lie in the cone $\Lambda$ is zero.
\end{remark}

In the special case when $H=\hU=U \rtimes \lambda(\CC^*)$ %is a graded unipotent group,
 the theorem has a particularly nice form when an appropriate choice of $\Lambda$ has been made. In this case $T^H = \lambda(\CC^*)$ and its maximal compact subgroup is the circle $\lambda(S^1)$, whose Lie algebra is usually identified with $\RR$ via $x \mapsto 2\pi i x$ for $x \in \RR$. Then it is traditional to take the cone $\Lambda$ to consist of the  positive reals, which means that $F \in \calf$ only contributes to the residue formula if its corresponding weight is strictly positive (see for example the version of Theorem \ref{jeffreykirwan} for a one-dimensional torus, given by \cite{JKintersection} %Jeffrey Kirwan Intersection theory on moduli spaces of holomorphic bundles of arbitrary rank on a Riemann surface
Corollary 3.2 and Remark 3.4). For a well-adapted linearisation, however, it is more convenient to take $\Lambda$ to consist of the negative reals, so that only $F_{\min} = Z_{\min}(X)$ contributes to the residue formula. This gives us a version of the special case of Theorem \ref{jeffreykirwan} for a one-dimensional torus which can be obtained from \cite{JKintersection} Corollary 3.2 by substituting $-x$ for $x$, and in combination with Theorem \ref{thm:integration} gives us the following result.

\begin{corollary}[\textbf{Localisation for $\hU$ quotients}]
\label{jeffreykirwannonredb}
Let $X$ be a smooth projective variety with a well-adapted action of $\hU=U \rtimes \lambda(\CC^*)$ such that $\hU$-stability=$\hU$-semistability holds in the sense of Definition \ref{cond star}. Let $z$ be the standard coordinate on the Lie algebra of $\CC^*$. For any %Chern polynomial $\phi(V_{\mu_{\hU}})\in \CC[c_1(V_{\mu_{\hU}}),c_2(V_{\mu_{\hU}}),\ldots]$ 
equivariant cohomology class $\eta$ on $X$ represented by an equivariant differential form $\eta(z)$ 
whose degree is the dimension of $X/\!/\hU$ we have 
\[\int_{X/\!/\hU}\kappa_{\hU}(\eta)=n_{\hU}\mathrm{Res}_{z=0} \int_{F_{\min}}\frac{i_{F_{\min}}^* (\eta(z)\cup \mathrm{Euler}(V_\lieu)(z))}{\mathrm{Euler}^T(\mathcal{N}_{F_{\min}})(z)}dz\]
where $F_{\min}=Z_{\min}(X)$ is the union of those connected components of the fixed point locus $X^{\lambda(\CC^*)}$ on which the $\lambda(S^1)$ moment map takes its minimum value $\omega_{\min}$, and $n_{\hU}$ is a strictly positive rational number which  depends only on $\hU$ and the size of the stabiliser in $\hU$ of a generic $x \in X$. 
\end{corollary}

\bibliographystyle{abbrv}
\bibliography{BercziKirwanCohomology.bib}

\end{document}